\documentclass[a4paper,11pt]{article}

\usepackage{BeaCed_article_1}
\newcommand{\new}{\textcolor{black}}

\title{Minimal bipartite dimers and higher genus Harnack
curves}
\author{%
C\'edric Boutillier%
  \thanks{%
    {\small Sorbonne Université, CNRS,
      Laboratoire de Probabilités Statistique et Modélisation, LPSM, UMR 8001,
      F-75005 Paris, France; Institut Universitaire de France.
      \texttt{cedric.boutillier@sorbonne-universite.fr}
}},
  David Cimasoni%
  \thanks{%
    {\small Université de Genève, Section de Mathématiques, 1211 Genève 4, Suisse.
     \texttt{david.cimasoni@unige.ch}
    }},
  B\'eatrice de Tili\`ere%
  \thanks{{\small %
PSL University-Dauphine, CNRS, UMR 7534, CEREMADE, 75016 Paris, France; Institut Universitaire de France.}
{\small\texttt{detiliere@ceremade.dauphine.fr}}
}
}

\begin{document}
\maketitle
\begin{abstract}
This paper completes the comprehensive study of the dimer model on infinite minimal graphs with Fock's weights~\cite{Fock}
initiated in~\cite{BCdTelliptic}: the latter article dealt with the elliptic case, \emph{i.e.}, models whose associated spectral curve is of
genus~1, while the present work applies to models of arbitrary genus.
This provides a far-reaching extension of the genus~0 results of~\cite{Kenyon:crit,KO:Harnack}, from isoradial graphs with
critical weights to minimal graphs with weights defining an arbitrary spectral data.
For any minimal graph with Fock's weights, we give an explicit local expression for a two-parameter
family of inverses of the associated Kasteleyn operator.
In the periodic case, this allows us to prove local formulas for all ergodic Gibbs measures,
thus providing an alternative description of the measures constructed in~\cite{KOS}.
We also compute the corresponding slopes, exhibit an explicit parametrization of the spectral curve,
identify the divisor of a vertex,
and build on~\cite{KO:Harnack,GK} to
establish a correspondence between Fock's models on periodic minimal graphs and Harnack curves
endowed with a standard divisor.
\end{abstract}

\section{Introduction}
\label{sec:intro}

This paper is a follow up to~\cite{BCdTelliptic}. The latter aimed at giving a
comprehensive study of the dimer model on infinite minimal graphs with Fock's
\emph{elliptic} weights, \emph{i.e.}, with underlying Riemann surface of genus~1.
This extended the \emph{rational} (genus 0) results of~\cite{Kenyon:crit,KO:Harnack},
since the rational case can be interpreted as a degeneration of the elliptic one~\cite[Section~8.1]{BCdTelliptic}.
We now turn to the general case, and consider Fock's weights in any genus~$g>0$~\cite{Fock}.

As in many questions involving compact Riemann surfaces,
moving from~$g=1$ to~$g\ge 1$ is by no means a trivial extension.
Moreover, for the dimer model to have a probabilistic meaning, it needs to be parametrized by
a specific type of abstract, compact Riemann surface know as an \emph{M-curve}~\cite{Harnack,Mikhalkin1}, which we need to thoroughly understand.
Therefore, a significant part of this paper, namely the whole of Section~\ref{sec:M}, is devoted to establishing the results on M-curves necessary for the study of the corresponding dimer models: we describe
their period matrix, Abel-Jacobi map, Riemann theta functions and prime form.

Before turning to the statistical mechanics implications, let us recall the context of Fock's work~\cite{Fock} and ours,
that is, the rich interplay between dimer models on the one hand,
real algebraic geometry and complex analysis on the other.
In their seminal paper~\cite{KOS}, Kenyon, Okounkov and Sheffield show that the spectral curve of a dimer model
on a~$\ZZ^2$-periodic, bipartite graph is of a very special type, namely a \emph{Harnack curve}.
In the subsequent articles~\cite{KO:Harnack,GK}, the authors define the \emph{spectral data} of such a dimer model
as the spectral curve~$\C$ together with a divisor consisting of one point on each of the ovals of~$\C$.
Furthermore, they show that given any Harnack curve and any such \emph{standard divisor}, there exists a dimer model realizing this spectral data,
and that the dimer model can be chosen on a \emph{minimal} graph~\cite{Thurston,GK}.
We refer the reader to Section~\ref{sub:dimer} for the definition of
minimal graphs and of the related notion of \emph{train-tracks}. Let us point out that throughout this article
and unless otherwise stated, the graphs are locally finite, embedded in the plane, with faces consisting of bounded
topological discs; in particular, they are infinite.

The articles~\cite{KO:Harnack,GK} contain an explicit construction of a
periodic minimal graph from the spectral curve,
a characterisation of genus~0 spectral curves as coming from isoradial graphs with critical weights~\cite{Kenyon:crit},
but no determination of the actual dimer model in general.
In other words, these powerful results do not answer the following question:
\emph{given a Harnack curve~$\C$, can we explicitly construct a dimer model on a periodic bipartite graph whose spectral curve is~$\C$?}

The remarkable contribution of Fock~\cite{Fock} consists in filling this gap.
More precisely,
Fock starts with an arbitrary complex curve~$\C$ (not necessarily Harnack)
and an arbitrary divisor of the appropriate degree (not necessarily a standard divisor),
and constructs an explicit ``dimer model'' on a periodic minimal graph
whose spectral data is the curve~$\C$ together with this divisor.
The quotation marks are due to the fact that in Fock's construction, dimer weights are complex.
From our statistical mechanics perspective, an important question is to understand in which setting
the ``dimer model'' is indeed a probabilistic model, \emph{i.e.}, has positive edge-weights.
This is the first main contribution of the present paper and the content of Proposition~\ref{prop:Kasteleyn},
whose proof heavily relies on the study of M-curves
of Section~\ref{sec:M}.
Let us briefly introduce the tools required to explain this statement,
referring the reader to the relevant parts of Sections~\ref{sec:M} and~\ref{sec:Fock} for more complete definitions.

Fix a compact Riemann surface~$\Sigma$ of positive genus together with an element~$t$  of its Jacobian variety~$\Jac(\Sigma)$ \new{and a theta characteristic~$\binom{\delta'}{\delta''}\in(\frac{1}{2}\ZZ/\ZZ)^{2g}$}.
Let~$\Gs$ be a bipartite graph (not necessarily periodic), and let~$\mapalpha\colon\T\to\Sigma$ be a map
assigning to each train-track of~$\Gs$ an element of~$\Sigma$ called its \emph{angle}.
Let~$\mapd\colon\{\text{faces of~$\Gs$}\}\to\Jac(\Sigma)$ be the \emph{discrete Abel map}, uniquely defined up to an additive constant
by the local rule described in Figure~\ref{fig:around_rhombus}.
\emph{Fock's adjacency operator} is represented by an infinite matrix~$\Ks$ whose rows are indexed by white
vertices of~$\Gs$, columns by black ones, and whose non-zero \new{entries}
correspond to edges of~$\Gs$ and are given by: for every edge~$\ws\bs$ of~$\Gs$ as in  Figure~\ref{fig:around_rhombus},
\begin{equation}\label{eq:def_Kast_intro}
    \K_{\ws,\bs}=\frac{E(\alpha,\beta)}{\theta\thchar{\delta'}{\delta''}(t+\mapd(\fs))\,
    \theta\thchar{\delta'}{\delta''}(t+\mapd(\fs'))}\,,
\end{equation}
where~$E$ is the prime form of~$\Sigma$
and~$\theta\thchar{\delta'}{\delta''}$ the theta function with theta characteristic~$\binom{\delta'}{\delta''}$.

\begin{rem}\label{rem:weight_Fock_intro}
In the genus~0 case~\cite{Kenyon:crit}, the weight of the edge~$\ws\bs$ only depends on the angles~$\alpha,\beta$ of the train-tracks crossing $\ws\bs$, and is therefore referred to as a \emph{local weight}.
In the general case defined above, the situation is just about as favorable, but not quite.
Indeed, the weight~$\K_{\ws,\bs}$ depends on~$\alpha,\beta$, but also on the value of~$\mapd$ at neighbouring faces.
The map~$\mapd$ is non-local, but is defined via a local rule: it should be thought of as a discrete primitive of the angle map~$\mapalpha$
with initial condition given by~$t$.
Despite this subtle difference, we make a slight abuse of terminology and still refer to the weights defined in Equation~\eqref{eq:def_Kast_intro} as \emph{local weights}.
\end{rem}

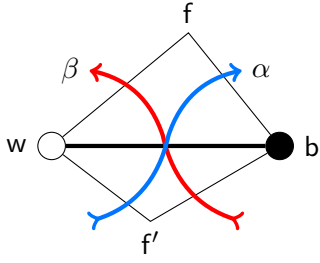
\begin{figure}
  \centering
  \definecolor{electricblue}{RGB}{0,117,255}
  \begin{tikzpicture}[auto]
    \node  (w) at (-1.5,0) [shape=circle,draw,fill=white,label=left:$\ws$] {};
    \node (b) at (1.5,0) [shape=circle,draw,fill=black,label=right:$\bs$] {};
    \node (o) at (0,0) {};
    \coordinate [label=above:$\fs$](f)  at (0.3,1.5)  ;
    \coordinate [label=below:$\fs'$](g)  at (-0.2,-1)  ;
    \draw  (b) -- (f) -- (w) -- (g) --  (b);
    \draw[ultra thick] (w) -- (b);
    \coordinate [label=left:$\beta$] (d1) at (-1,1);
    \coordinate [label=right:$\alpha$] (d2) at (1,1);
    \path[<-<,draw=red,ultra thick]  (-1,1) to [bend left]   (o)  to [bend right] (1,-1) ;
    \path[<-<,draw=electricblue,ultra thick]   (1,1) to [bend right]   (o)  to [bend left] (-1,-1) ;
  \end{tikzpicture}
  \caption{%
An arbitrary edge~$\ws\bs$ of~$\Gs$ with the two adjacent faces~$\fs,\fs'$ \new{depicted as dual vertices.
The four corresponding edges of the quad-graph~$\GR$ are also drawn, together with the two}
incident train-tracks with
angles~$\alpha,\beta\in\Sigma$.
\new{The discrete Abel map~$\mapd$ satisfies the local rule}~$\mapd(\fs')-\mapd(\fs)=\alpha-\beta\in\Pic^0(\Sigma)\simeq\Jac(\Sigma)$.
  }
  \label{fig:around_rhombus}
\end{figure}

An \emph{M-curve} is a compact Riemann surface of genus~$g$ endowed with an anti-holomorphic
involution~$\sigma$
whose set of fixed points is given by~$g+1$ circles, called \emph{real components} (see Figure~\ref{fig:surface}).
Proposition~\ref{prop:Kasteleyn} can now be stated as
follows.

\begin{prop}\label{prop:Kasteley_intro}
Let us assume that
\begin{itemize}
\item[(i)] the surface~$\Sigma$ is an M-curve (of genus~$g>0$ and period matrix~$\Omega$);
\item[(ii)] the element~$t$ of $\Jac(\Sigma)$ is real, \emph{i.e.}, belongs to~$(\RR/\ZZ)^g+\Omega\delta$ for some~$\delta\in(\frac{1}{2}\ZZ/\ZZ)^g$;
\item[(iii)] the theta characteristic~$\binom{\delta'}{\delta''}$ satisfies~$\delta'=\delta$;
\item[(iv)] the graph~$\Gs$ is minimal;
\item[(v)] the image of~$\mapalpha$ is contained in a real (oriented) component~$A_0$ of $\Sigma$, and~$\mapalpha\colon\T\to A_0$ is monotone with respect to the natural cyclic orders on~$\T$ and on~$A_0$.
\end{itemize}
Then,
Fock's adjacency operator\eqref{eq:def_Kast_intro} is a Kasteleyn operator, \emph{i.e.}, defines
a dimer model on the graph $\Gs$.
\end{prop}

In the periodic setting, these constraints can be heuristically explained as follows.
As we will see in more detail below, the abstract Riemann surface~$\Sigma$ serves as a parametrization domain for the spectral curve~$\C$;
this latter curve being Harnack, it has the maximal number of real components.
As for the element~$t\in\Jac(\Sigma)$, it corresponds via the Abel-Jacobi map and the identification~$\Sigma\simeq\C$ to the divisor
of the model, which consists of one point on each of the ovals of~$\C$; such divisors are mapped to real elements of the Jacobian variety.
Finally, the train-tracks of~$\Gs$ correspond to the ``points at infinity'' of~$\C$, which according to the definition of a Harnack curve,
must be arranged in a natural cyclic order; this forces the angle map~$\mapalpha$ to be monotone.
As proved in~\cite{BCdTimmersions}, the minimality of~$\Gs$ then ensures that this (global) monotonicity
implies the corresponding (local) monotonicity around each face of~$\Gs$.
This can finally be translated into the
Kasteleyn condition using properties of the prime form restricted to the real component of an M-curve.

To a certain extent, the arguments given above show that conditions~$(i)$,~$\mathit{(ii)}$ and~$(v)$
are actually necessary for~$\Ks$ to be a Kasteleyn operator, at least in the periodic setting.
The fact that the graph~$\Gs$ needs to be minimal for the theory to apply is discussed at the end of Section~\ref{sub:inv}, see also~\cite[Theorem~31]{BCdTimmersions}. Finally, there is little hope for Proposition~\ref{prop:Kasteley_intro}
to hold without condition~$\mathit{(iii)}$ because of the last point of Lemma~\ref{lem:ident_2}.

On the other hand, one shows that all theta
characteristics~$\binom{\delta'}{\delta''}\in(\frac{1}{2}\ZZ/\ZZ)^{2g}$ yield
gauge equivalent models, so we set~$\delta'=\delta''=0$ for definiteness in our
final definition,
see the third and fourth points of Remark~\ref{rem:Kast}, and Definition~\ref{def:K}.

Note that Fock originally defined his operator for periodic graphs,
and some of the heuristic arguments above only hold in this restricted situation.
However, Proposition~\ref{prop:Kasteley_intro} is valid for any minimal graph, allowing us to
harness the power of Kasteleyn theory in this very general setting.
Before doing so, let us mention that even with the constraints listed in Proposition~\ref{prop:Kasteley_intro},
the dimer models given by Fock's weights span all periodic dimer models, when considered from the point of view of their spectral data.
Indeed, we prove the following result, see Theorem~\ref{thm:Harnack} for a full statement.

\begin{thm}\label{thm:Harnack_intro}
For any Harnack curve~$\C$ and standard divisor~$D$, there exists an~M-curve~$\Sigma$,
a periodic minimal graph~$\Gs$, a monotone angle map~$\mapalpha$ and a
real element~$t$ of~$\Jac(\Sigma)$, such that the associated Fock operator~$\Ks$ is periodic, and the spectral
data of the corresponding dimer model coincides with~$(\Cscr,D)$.
\end{thm}

\new{Together with~\cite[Theorem~7.3]{GK}, this implies that
two periodic dimer models on the same minimal graph~$\Gs$ arising from the same M-curve~$\Sigma$,
the same angle map~$\mapalpha$, and elements~$t,t'\in(\RR/\ZZ)^{g}$
are gauge equivalent if and only if~$t=t'$, see Remark~\ref{rem:Harnack}.}

We have another result worth mentioning on the topic of spectral curves.
Consider the dimer model on a periodic, minimal graph~$\Gs$ with Fock's weights given by
parameters~$\Sigma,t$ and~$\mapalpha$ as in Proposition~\ref{prop:Kasteley_intro}.
In Proposition~\ref{prop:param_curve}, we provide an explicit birational parametrization~$\Sigma\to\C$ of
the spectral curve.
This allows us to transport several notions, such as the standard divisor of the model, from the spectral curve to the underlying abstract M-curve, thus extending these notions beyond the periodic setting.
Most notably,
the phase diagram of the dimer model can be tranported
from~$\C$ to~$\Sigma$.
This is related to probabilistic questions which we now address.

Before turning to our results, let us describe the setting and motivations.
Consider a dimer model on a periodic bipartite graph~$\Gs$.
Kenyon, Okounkov and Sheffield~\cite{Sheffield,KOS} prove that there is a
two-parameter family of ergodic Gibbs measures, indexed by the \emph{slope},
and that the set of allowed slopes coincides with the \emph{Newton polygon}~$N(\Gs)$.
Moreover,
they provide an explicit expression for this family of measures by
taking the weak limit of the Boltzmann measures on a toroidal exhaustion, with
weights modified by magnetic field coordinates.
They also
prove that the dimer model has three phases, liquid, solid and gaseous, and that
the phase diagram is given by the amoeba~$\Ascr$
of the spectral curve.
Using different techniques, Kenyon~\cite{Kenyon:crit} establishes an
explicit expression for the maximal entropy Gibbs measure in the case of
isoradial graphs with critical weights, which has the remarkable property of
being \emph{local}: this means that edge probabilities can be computed using
geometric information of paths joining
these edges. Note that by
uniqueness, we know that the expressions of~\cite{KOS} and~\cite{Kenyon:crit}
are equal, but this is only explicitly understood since~\cite{BdTR1}.
Local expressions have now been obtained for dimer models related to the Ising
model~\cite{BeaCed:isogen,BdTR2}, to rooted spanning forests~\cite{BdTR1}, and for the
two parameter family of Gibbs measures of the dimer model with Fock's elliptic weights~\cite{BCdTelliptic}.
Note that these are non trivial extensions of the result of~\cite{Kenyon:crit},
two of the main difficulties being to find the appropriate extension of the discrete exponential functions of Mercat~\cite{Mercat:exp}
and to define suitable paths of integration.
Having local expressions for Gibbs measures
opens the way to computing precise asymptotics, and
to constructing Gibbs measures for general, possibly non-periodic graphs.
This latter application requires an additional argument, however, and such extensions to non-periodic graphs have only been
obtained in some specific rational and elliptic cases~\cite{BeaQuad,BeaCed:isogen,BdTR1, BdTR2}.

These results yield the following question: \emph{can we obtain an explicit
local expression for the two-parameter family of ergodic Gibbs measures of the
dimer model on periodic, bipartite graphs?}
We give a positive answer for all dimer models with Fock's weights on minimal periodic graphs.

Recall that, by Theorem~\ref{thm:Harnack_intro} and the general theory of~\cite{KO:Harnack,GK},
any dimer model on a periodic minimal graph is gauge-equivalent to a model with Fock's weights (see Remark~\ref{rem:Harnack}).
In that sense, our result extends the theory initiated by Kenyon twenty years ago~\cite{Kenyon:crit}, originally valid for one measure
on isoradial graphs, to the full set of ergodic Gibbs measures of \emph{any} dimer model on a periodic minimal graph.

We now state this result in two steps.
Let us fix a minimal graph~$\Gs$ together with parameters~$\Sigma,t,\mapalpha$ as in Proposition~\ref{prop:Kasteley_intro},
and consider
the associated Kasteleyn operator~$\Ks$ defined in Equation~\eqref{eq:def_Kast_intro}.
Let~$\Sigma^+$ denote the upper half of the M-curve~$\Sigma$, see Figure~\ref{fig:surface},
and set~$\D=\Sigma^+\setminus\mapalpha(\T)$.
The first step is an explicit, \emph{local} expression for a two parameter family of
inverses~$(A^{u_0})_{u_0\in\D}$ of Fock's Kasteleyn operator~$\Ks$,
see Definition~\ref{defi:inv} and Theorem~\ref{thm:Kinv} for details.

\begin{thm}\label{thm:Kinv_intro}
For every~$u_0\in\D$, consider the operator~$\As^{u_0}$
defined as follows: for every black vertex~$\bs$ and white vertex~$\ws$ of~$\Gs$, set
\[
\A_{\bs,\ws}^{u_0}=\frac{1}{2i\pi}\int_{\sigma(u_0)}^{u_0} g_{\bs,\ws}\,,
\]
where~$g_{\bs,\ws}$ is the meromorphic 1-form on~$\Sigma$ with explicit local expression given in Section~\ref{sub:kernel},
and the integration path in~$\Sigma$ from~$\sigma(u_0)$ to~$u_0$ is defined in Section~\ref{sub:inverse}.
Then, the operator~$A^{u_0}$ is an inverse of the Kasteleyn operator~$\Ks$.
\end{thm}

\begin{rem}\label{rem:inverse_intro}\leavevmode
\begin{enumerate}
\item The terminology \emph{local} is used in the same sense as in Remark~\ref{rem:weight_Fock_intro}: when $g\geq 1$, there is some non-local information, all encoded in the discrete Abel map $\mapd$.
\item The cornerstone of the proof is Fay's celebrated identity~\cite{Fay}, see Section~\ref{subsub:Fay}. This identity is also the reason why (and in some precise sense, equivalent to the fact that) the dimer model with Fock's weights
is invariant under natural local transformations, see Section~\ref{sub:inv}.
\end{enumerate}
\end{rem}

We are now ready to state our result for Gibbs measures on periodic minimal
graphs. This is a combination of
Theorem~\ref{prop:ABAu0} and Corollary~\ref{cor:Gibbs}.

\begin{thm}\label{thm:Gibbs_intro}
For every~$u_0\in\D$, consider the measure~$\PP^{u_0}$ whose expression on cylinder sets is given as follows:
for every set~$\{\es_1=\ws_1\bs_1,\dotsc,\es_k=\ws_k\bs_k\}$ of distinct edges of~$\Gs$,
\[
\PP^{u_0}(\es_1,\dotsc,\es_k)=\Bigl(\prod_{j=1}^k \Ks_{\ws_j,\bs_j}\Bigr)\times \det_{1\leq i,j\leq k} \Bigl(\As^{u_0}_{\bs_i,\ws_j}\Bigr)\,.
\]
This defines an ergodic Gibbs measure on dimer configurations of~$\Gs$.

Moreover,
the measures~$(\PP^{u_0})_{u_0\in\D}$ form the two-parameter family of ergodic Gibbs measures of~\cite{KOS},
where~$u_0\in\D\subset\Sigma$ is related to the magnetic field coordinates in~$\Ascr$ via the composition of
the explicit parametrization~$\Sigma\to\C$ from Proposition~\ref{prop:param_curve}
with the amoeba map~$\C\to\Ascr$.

Finally, if~$u_0$ belongs to the real component~$A_0$ of $\Sigma$ (resp. to the complement of~$A_0$ in the real locus of~$\Sigma$, to the interior of~$\D$), then the corresponding dimer model is in a solid (resp. gaseous, liquid) phase.
\end{thm}

\new{Note that if~$u_0$ and~$u_1$ belong to the same connected component of the real locus of~$\D$,
then the operators~$\As^{u_0}$ and~$\As^{u_1}$ coincide, yielding identical measures~$\PP^{u_0}=\PP^{u_1}$.}

As mentioned above, this set of ergodic Gibbs measures is also naturally parametrized by slopes,
and it is natural to wonder whether a simple expression can be given for the slope
of the ergodic Gibbs measure~$\PP^{u_0}$.
This is done in Section~\ref{sub:slope}, but we shall not attempt to summarize these results here.
Let us only mention that, in our setting, the identification of the~$g$ distinct slopes corresponding to the gaseous phases uses
as its main ingredient the Riemann bilinear relation, see Corollary~\ref{cor:slope-gas}.
Once again, this illustrates the very rich interplay between statistical physics and complex analysis on compact Riemann surfaces at work in this theory.

\subsubsection*{Outline of the paper}
\begin{itemize}
  \item Section~\ref{sec:M} gathers all the results of complex analysis needed for our study of minimal bipartite dimers.
    After introducing abstract M-curves in Section~\ref{sub:M-curves},
    we recall the definition of several classical objects for an arbitrary compact Riemann surface~$\Sigma$ and study their special properties when~$\Sigma$ is an~M-curve:
    first the period matrix in Section~\ref{sub:period},
    then the Abel-Jacobi map
    in Section~\ref{sub:Abel-Jacobi}, the Riemann theta functions in Section~\ref{sub:theta}, and finally the prime form in Section~\ref{sub:prime}.
  \item In Section~\ref{sec:Fock}, after briefly recalling some background material on minimal bipartite dimers (Section~\ref{sub:dimer}),
  the discrete Abel map and monotone angle maps (Section~\ref{sub:mapd}), we
  prove Proposition~\ref{prop:Kasteley_intro} in Section~\ref{sub:Kast}.
   In Section~\ref{sub:kernel}, we construct explicit forms in the kernel
   of~$\Ks$, which we use in Section~\ref{sub:inverse} to construct a~two-parameter family of inverses of~$\Ks$,
   proving Theorem~\ref{thm:Kinv_intro}.
   \item Section~\ref{sec:periodic} deals with the case of~$\ZZ^2$-periodic models. After the preliminary Section~\ref{sub:prelim}, we study the periodicity of~$\Ks$ in Section~\ref{sec:Kast-per}.
   In Section~\ref{sub:curve}, we give an explicit parametrization of the spectral curve by the abstract~M-curve~$\Sigma$,
   and prove Theorem~\ref{thm:Harnack_intro}.
   In Section~\ref{sub:Gibbs}, we study the full set of ergodic Gibbs measures, proving Theorem~\ref{thm:Gibbs_intro}.
   Finally, we derive explicit formulas for the slopes of the Gibbs measures in Section~\ref{sub:slope},
and study the surface tension and free energy in Section~\ref{sub:free}.
   \item The more informal Section~\ref{sec:more} deals with miscellaneous additional features of our theory: the construction of Gibbs measures beyond the periodic case in Section~\ref{sub:beyond},
   the invariance of the model under local transformations in Section~\ref{sub:inv}, and its relation to known models in Section~\ref{sub:known}. This concluding section also contains various future perspectives.
\end{itemize}

\subsubsection*{Acknowledgements}
The authors express their gratitude to Vladimir Fock for stimulating discussions, as well as Erwan Brugall\'e, Elisha Falbel, Ilia Itenberg, Nicolas Lerner, Florent Schaffauser, and Evgeny Verbitskiy. \new{They also thank the two anonymous referees for their very careful reading of the paper and their valuable suggestions. The authors}
acknowledge that Alexander Bobenko, Nikolai Bobenko and Yuri Suris informed them that they are working on a related project.
The first- and third-named authors are partially supported by the \emph{DIMERS}
project ANR-18-CE40-0033 funded by the French National Research Agency.
The second-named author is partially supported by the Swiss NSF  grant 200020-200400.

\section{Compact Riemann surfaces and~M-curves}
\label{sec:M}

This section contains all the results in complex analysis that are needed for our study of dimers on minimal graphs.
More precisely, we recall classical statements about Riemann surfaces, referring to~\cite{Jost_2006,Farkas-Kra,ThetaTata1,ThetaTata2} for proofs and details,
and explain what more can be said in the case of~M-curves.

We start in Section~\ref{sub:M-curves} by recalling the definition of this special class of compact Riemann surfaces, and provide several examples.
In Section~\ref{sub:period}, we briefly summarise the theory of period matrices, whose \new{entries} are showed to be purely imaginary in the case of~M-curves (Lemma~\ref{lem:imaginary}).
Section~\ref{sub:Abel-Jacobi} deals with the Abel-Jacobi map, whose behaviour for~M-curves is described in Lemma~\ref{lemma:real}.
In Section~\ref{sub:theta}, we recall the definition of the Riemann theta functions along with their well-known general properties (Lemma~\ref{lem:ident_1}), and lesser-known behaviour for purely imaginary period matrices (Lemma~\ref{lem:ident_2}).
Finally, Section~\ref{sub:prime} deals with the general theory of prime forms, with Lemmas~\ref{lem:primeform_conj} and~\ref{lem:prime} containing the results needed in the case of~M-curves.

\subsection{Abstract M-curves}
\label{sub:M-curves}

Recall that an \emph{anti-holomorphic involution} on a
Riemann surface~$\Sigma$ is a smooth
involution~$\sigma\colon\Sigma\to\Sigma$ whose induced map~$\sigma_*\colon T\Sigma\to T\Sigma$
satisfies~$\sigma_*\circ J=-J\circ\sigma_*$, where~$J$ denotes the almost-complex structure on~$\Sigma$.
The points of~$\Sigma$ that are fixed by~$\sigma$ are said to be \emph{real}.

One easily shows that if~$\sigma$ is an \new{anti-holomorphic} involution on a compact orientable surface of genus~$g$,
then its set of fixed points consists of at most~$g+1$ topological circles. (This is Harnack's theorem,
whose proof follows from an Euler characteristic argument.)

\begin{defi}
An \emph{(abstract)~M-curve} is a compact Riemann surface~$\Sigma$ endowed with an
anti-holomorphic involution $\sigma$ whose set of fixed points is
given by~$g+1$ topological circles, where~$g$ is the genus of~$\Sigma$.
\end{defi}

The~M in~M-curve stands for~`maximal'. We now give some examples.

\begin{exm}
Any genus~$0$ Riemann surface is isomorphic to the Riemann sphere,
which is trivially an~M-curve with respect to complex conjugation.
This case being well-known, we assume from now on that~$g$
is positive. Note however that this rational case can be recovered
as a degeneration of the elliptic case, \emph{i.e.}~$g=1$,
as explained in~\cite[Section~8.1]{BCdTelliptic}.
\end{exm}

\begin{exm}
\label{ex:Melliptic}
A Riemann surface of genus~$1$ is isomorphic to
a torus~$\TT(\tau)=\mathbb{C}/(\mathbb{Z}+\tau\mathbb{Z})$ of modular parameter~$\tau$
with~$\Im(\tau)>0$.
The complex conjugation admits as real locus the curve~$\RR/\ZZ$,
together with the curve~$\RR/\ZZ+\frac{\tau}{2}$
if and only if~$\tau$ is purely imaginary.
Therefore,~$\TT(\tau)$ is an~M-curve if and only if~$\tau$ is purely imaginary.
This case is treated extensively in~\cite{BCdTelliptic}.
\end{exm}

\begin{exm}
By definition, (the toric closure of) a Harnack curve in~$(\CC^*)^2$
is an~M-curve with respect to the anti-holomorphic involution given by~$\sigma(z,w)=(\overline{z},\overline{w})$.
\end{exm}

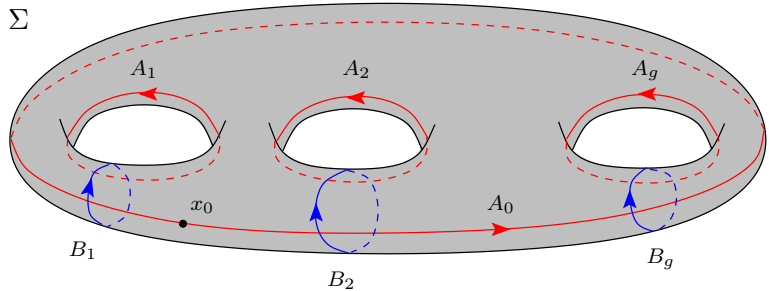
\begin{figure}
\center
\begin{overpic}[width=10cm]{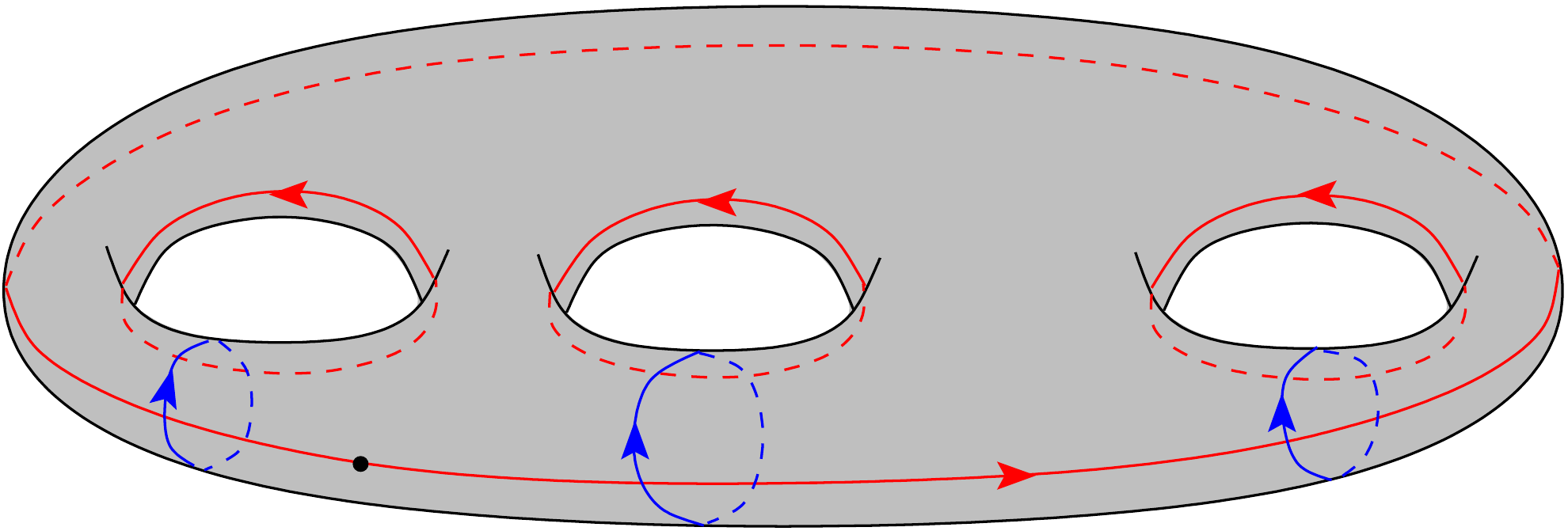}
\put(0,30){$\Sigma$}
\put(24,6){\scriptsize $x_0$}
\put(63,6){\scriptsize $A_0$}
\put(16,24){\scriptsize $A_1$}
\put(44,24){\scriptsize $A_2$}
\put(82,24){\scriptsize $A_g$}
\put(8,0){\scriptsize $B_1$}
\put(42,-4){\scriptsize $B_2$}
\put(84,-1){\scriptsize $B_g$}
\end{overpic}
  \medskip
\caption{The surface~$\Sigma$ together with the cycles~$A_0,A_1,\dotsc,A_g$ (in red) and~$B_1,\dotsc,B_g$ (in blue).
In this picture, the \new{anti-holomorphic} involution~$\sigma$ should be understood as the
refle\new{ct}ion \new{across} the horizontal plane containing $A_0,A_1,\dotsc,A_g$.}
\label{fig:surface}
\end{figure}

\subsection{The period matrix}
\label{sub:period}

This section is devoted to the study of the period matrix of an abstract~M-curve.

We fix a real point~$x_0$ in an~M-curve~$\Sigma$, denote by~$A_0$ the corresponding real circle, and
number the remaining ones as~$A_1,\dotsc,A_g$.
Note that the real locus necessarily separates~$\Sigma$ into two connected surfaces with boundary;
we fix an orientation of the real locus so that the oriented boundary of one of these surfaces, denoted by~$\Sigma^+$,
is equal to~$A_0-(A_1+\dotsb+A_g)$. Finally, we use the same symbol~$A_j$ for the oriented cycle in~$\Sigma$
and its homology class in~$H_1(\Sigma;\mathbb{Z})$.

Note that there are homology classes~$B_1,\dotsc,B_g\in H_1(\Sigma;\mathbb{Z})$
with~$\sigma_*(B_i)=-B_i$ and such that~$\{A_1,\dotsc,A_g,B_1,\dotsc,B_g\}$ forms a basis of~$H_1(\Sigma;\ZZ)$
satisfying the equalities
\[
A_i \wedge A_j =0,
 \quad
  B_i \wedge B_j = 0,
  \quad
  A_i \wedge B_j = \delta_{i,j},
\]
for all~$1\le i,j\le g$, where~$\wedge$ denotes the intersection form.
This is illustrated in Figure~\ref{fig:surface}.

The complex vector space of holomorphic differential forms has dimension~$g$.
Let us denote by~$\vec{\omega}=(\omega_1,\dotsc,\omega_g)$ the basis of this space determined by
\begin{equation*}
  \forall\ 1\leq i,j\leq g,\quad \int_{A_i}\omega_j = \delta_{i,j}\,.
\end{equation*}

We let~$\Omega$ be the matrix \new{with entries}~$\new{\Omega_{ij}:=\int_{B_j}\omega_i}$.
This is the
``interesting part" of the \emph{period matrix} $\begin{pmatrix}I_g &
\Omega\end{pmatrix}$ of $\Sigma$ in the basis $\vec{\omega}$.
By the general theory,~$\Omega$ \new{is a symmetric matrix whose} imaginary part is positive definite, and the columns of the period matrix are linearly independent over~$\mathbb{R}$.
They generate a full rank lattice $\Lambda=\mathbb{Z}^g\new{\oplus}\Omega\mathbb{Z}^g$ in
$\mathbb{C}^g$.

In the setting of~M-curves, the \new{entries} of~$\Omega$ are purely imaginary.
This is the subject of the following lemma.

\begin{lem}
\label{lem:imaginary}
If~$\Sigma$ is an abstract~M-curve with \new{anti-holomorphic} involution~$\sigma$, then the following holds.
\begin{enumerate}
\item For all~$1\leq j\leq g$, we have the equality~$\sigma^*\omega_j=\overline{\omega}_j$.
\item For all~$0\leq j\leq g$, the subspace~$TA_j\subset T\Sigma$ is fixed pointwise by~$\sigma_*$.
\item The \new{entries} of~$\Omega$ are purely imaginary.
\end{enumerate}
\end{lem}

\begin{proof}
To show the first point, consider the~$1$-forms defined by~$\omega_j':=\overline{\sigma^*\omega_j}$
for~$1\le j\le g$. We now check that these are holomorphic forms.
By definition, if a form~$\omega$ is holomorphic, it is a~$(1,0)$-form with~$\overline\partial\omega=0$.
It follows that~$\sigma^*\omega$ is a~$(0,1)$-form with~$\partial(\sigma^*\omega)=0$. Indeed,
the involution~$\sigma$ being anti-holomorphic, we have
\[
\sigma^*\omega(J(v))=\omega(\sigma_*(J(v))=\omega(-J(\sigma_*(v)))=-i\omega(\sigma_*(v))=-i\sigma^*\omega(v)
\]
for all~$v\in T\Sigma$, checking the first claim. The second follows from the naturality of~$\sigma^*$ via
\[
\partial(\sigma^*\omega)=d(\sigma^*\omega)=\sigma^*(d\omega)=\sigma^*(\overline\partial\omega)=\sigma^*(0)=0\,.
\]
The fact that~$\omega'=\overline{\sigma^*\omega}$ is a~$(1,0)$-form with~$\overline\partial\omega'=0$
now follows easily, proving that~$\omega_j'=\overline{\sigma^*\omega_j}$ is a holomorphic form.
Next, observe that these holomorphic forms satisfy
\[
\int_{A_i}\omega_j'=\int_{A_i}\overline{\sigma^*\omega_j}=
\overline{\int_{\sigma_*A_{\new{i}}}\omega_j}=
\overline{\int_{A_{\new{i}}}\omega_j}=\overline{\delta_{i,j}}=\delta_{i,j},
\]
for all~$1\le i,j\le g$. Since these properties characterize the basis of holomorphic forms,
this shows the equality~$\omega_j'=\omega_j$, and the first point.

To prove the second one, simply observe that an element of~$TA_j$ is the form~$\gamma'(0)$ with~$\gamma$
a parametrization of~$A_j$. Since this curve is fixed pointwise by~$\sigma$, we have~$\sigma\circ\gamma=\gamma$
and~$\sigma_*(\gamma'(0))=\frac{d}{dt}(\sigma\circ\gamma)(0)=\gamma'(0)$.

The third point follows from the first one via
\[
\overline{\int_{B_i}\omega_j}=\int_{B_i}\overline{\omega_j}=
\int_{B_i}\sigma^*\omega_j=\int_{\sigma_*B_{\new{i}}}\omega_j=
\int_{-B_{\new{i}}}\omega_j=-\int_{B_{\new{i}}}\omega_j\,.
\]
This concludes the proof.
\end{proof}

\begin{exm}
\label{ex:period-elliptic}
As mentioned above, a Riemann surface of genus~$1$ is isomorphic to
a torus~$\TT(\tau)=\mathbb{C}/(\mathbb{Z}\new{\oplus}\tau\mathbb{Z})$ of modular parameter~$\tau$
with~$\Im(\tau)>0$.
It is an~M-curve if and only if~$\tau$ is purely imaginary.
In such a case, the imaginary axis~$\tau\RR/\tau\ZZ$ can be chosen as the cycle representing the
class~$B_1$.
The basis of holomorphic forms is then given by~$\omega_1=dz$,
since its integral along~$A_1=\mathbb{R}/\mathbb{Z}+\frac{\tau}{2}$ is equal to~$1$.
Along~$B_1=\tau\RR/\tau\ZZ$, the integral is~$\tau$,
so the period matrix is simply given by~$\begin{pmatrix}1&\tau\end{pmatrix}$.
\end{exm}

\begin{exm}
In the case of a Harnack curve,
there is a very concrete way to describe~$\vec{\omega}$ and compute~$\Omega$, explained in
the proof of Proposition~6 of~\cite{KO:Harnack}, see also the proof of
Theorem~3 of~\cite{Cretois-Lang} for the general setting.
\end{exm}

We need an additional lemma,
part of which is known to hold for Harnack curves, see the end of Section~2 of~\cite{KO:Harnack}.
We now show that it is valid in the more general setting of~M-curves (with a simpler proof).

\begin{lem}
\label{lem:non-zero}
Let~$\Sigma$ be an~M-curve, with real circles~$A_0,A_1,\dotsc,A_g$ and associated basis of holomorphic forms~$\vec{\omega}=(\omega_1,\dotsc,\omega_g)$ as above.
Then, for any~$1\le i\le g$ and~$\alpha\neq\beta\in A_0$, resp.~$A_i$, we have~$\int_\alpha^\beta\omega_i> 0$,
where the integration path follows the orientation of~$A_0$, resp.~$A_i$.
\end{lem}

\begin{proof}
\new{Fix~$1\le i\le g$. For any~$1\le j\le g$ with~$j\neq i$}, we have~$\int_{A_j}
\omega_i=0$, and the form~$\omega_i$ is real on~$A_j$.
Therefore, it can be written~$\omega_i=df$ in a neighborhood of~$A_j$, with~$f$
a complex-valued function taking real values on $A_j$.
This form not being identically zero,
the function~$f$ is non-constant, and~$\omega_i$
admits at least two
zeros on~$A_j$.
Furthermore, being a holomorphic differential form, it admits exactly
$2g-2$ zeros (counted with multiplicity).
In conclusion,  for all~\new{$1\le j\le g$ with}~$j\neq i$, the form~$\omega_i$ admits exactly~$2$ simple zeros on~$A_j$,
and no zero elsewhere.
In particular, it has no zero on~$A_i$ and on~$A_0$, so the integration along these real components is monotone.
Finally, by definition of the orientation of $A_0$, the integration along
these full loops is given by
\begin{equation*}
\int_{A_0}\omega_i = \sum_{j=1}^g \int_{A_j}\omega_i =\int_{A_i}\omega_i =1>0\,.
\end{equation*}
This implies both claims.
\end{proof}

\subsection{The Abel-Jacobi map}
\label{sub:Abel-Jacobi}

In this section, we briefly recall the definition of the Abel-Jacobi map associated with an arbitrary compact Riemann surface, before explaining its special features in the case of an~M-curve in Lemma~\ref{lemma:real}.

Recall that a \emph{divisor} on~$\Sigma$ is a formal linear combination of points of~$\Sigma$ with integer coefficients.
The set of divisors on~$\Sigma$ is endowed with a natural grading~$\Div(\Sigma)=\bigoplus_{n\in\ZZ}\Div^n(\Sigma)$,
where the degree of a divisor is the sum of its
integer coefficients.
A divisor~$D$ is said to be \emph{principal} if it represents the zeros and poles of
a non-zero meromorphic function~$f$ on~$\Sigma$, \emph{i.e.}, if it is of the form
\begin{equation*}
  D=\sum_z \operatorname{ord}_f(z) z.
\end{equation*}
Two divisors are said to be \emph{linearly equivalent} if their difference is a
principal divisor.
Since the number of zeros and poles of a non-zero meromorphic function coincide, \emph{i.e.}, the degree of a principal divisor vanishes,
the set of linear equivalence
classes of divisors forms a~$\ZZ$-graded Abelian group,
denoted by~$\mathrm{Pic}(\Sigma)=\bigoplus_{n\in\ZZ}\mathrm{Pic}^n(\Sigma)$.

Abel's theorem~\cite[Theorem~5.9.1]{Jost_2006} implies
that there is an injection from~$\mathrm{Pic}^0(\Sigma)$ to the
\emph{Jacobian variety}~$\Jac(\Sigma)=\mathbb{C}^g/\Lambda$ of~$\Sigma$ through the so-called \emph{Abel-Jacobi map}
\begin{equation*}
 D=\sum_i(y_i-x_i) \longmapsto \sum_i \int_{x_i}^{y_i} \vec{\omega} \in\mathbb{C}^g.
\end{equation*}
The decomposition of~$D\in\mathrm{Pic}^0(\Sigma)$ is not unique, and the right-hand side depends on the
choice of paths between~$x_i$ and~$y_i$ on $\Sigma$. However, two possible results
differ by an element of~$\Lambda$, so the formula displayed above gives a well-defined map~$\Phi\colon\mathrm{Pic}^0(\Sigma)\to\Jac(\Sigma)$.

Jacobi's inversion theorem~\cite[Theorem 5.9.2]{Jost_2006} states that this map
induces an isomorphism of Abelian groups~$\mathrm{Pic}^0(\Sigma)\simeq\Jac(\Sigma)$.
More concretely, the Abel-Jacobi map can be inverted as follows:
given~$\lambda\in\Jac(\Sigma)$ and a fixed point~$x_0\in\Sigma$,
one can find a divisor~$D=\sum_{i=1}^g x_i$ of degree~$g$ with only positive
coefficients (the $x_i$'s may not be distinct) such that~$\Phi(\sum_{i=1}^g
(x_i-x_0)) = \lambda$.
Following standard practice, we use the same notation for (the equivalence class of) a degree~$0$ divisor
and for its corresponding element in~$\Jac(\Sigma)$.

Note that one can define a map~$\Phi\colon\mathrm{Pic}(\Sigma)\to\Jac(\Sigma)$
by first sending~$\mathrm{Pic}^n(\Sigma)$ to~$\mathrm{Pic}^0(\Sigma)$ via~$D\mapsto D-nx_0$ and then
applying the Abel-Jacobi map.
In particular, this gives a well-defined map~$\bigsqcup_{n>0}\Sigma^n\to\Jac(\Sigma)$.
By abuse of notation, we simply denote it by~$\Phi$,
even though it does depend on the choice of~$x_0$.

In the case of Harnack curves, the real torus~$A_1\times\dotsb\times A_g$
is known to inject into~$\Jac(\Sigma)$ and form one of its real components,
see~\cite{KO:Harnack}, end of Section~2. We now show that this
still holds in the more general setting of~M-curves.

\begin{lem}
\label{lemma:real}
Let~$\Phi\colon\bigsqcup_{n>0}\Sigma^n\to\Jac(\Sigma)$ be the Abel-Jacobi map associated to an~M-curve~$\Sigma$, defined with respect to
a fixed real point~$x_0\in A_0$.
Let~$(e_1,\dotsc,e_g)$ denote the canonical basis of~$\CC^g$ (and of~$\ZZ^g$), and set~$\II=e_1+\dotsb+e_g$.
\begin{enumerate}
\item The real locus of~$\Jac(\Sigma)$ is equal
to~$(\RR/\ZZ)^g\new{\oplus}\Omega(\frac{1}{2}\ZZ/\ZZ)^g$, \emph{i.e.}, it consists of~$2^g$ real tori of dimension~$g$
indexed by~$(\frac{1}{2}\ZZ/\ZZ)^g$.
\item For every $1\leq i\leq g$, the map~$\Phi\colon\Sigma\to\Jac(\Sigma)$ sends the real component~$A_i$ to a cycle of homology class~$e_i\in\ZZ^g$
inside the real torus indexed by~$\frac{e_i}{2}\in(\frac{1}{2}\ZZ/\ZZ)^g$,
strictly increasing in the~$e_i$-direction.
\item The map~$\Phi\colon\Sigma\to\Jac(\Sigma)$ sends the component~$A_0$ to a cycle of
homology class~$\II\in\ZZ^g$ inside the real torus indexed by~$0\in(\frac{1}{2}\ZZ/\ZZ)^g$,
strictly increasing in the~$e_i$-direction for all~$1\le i\le g$.
\item The restriction of~$\Phi\colon\Sigma^g\to\Jac(\Sigma)$ to~$A_1\times\dotsb\times A_g$ defines a homeomorphism onto
the real torus indexed by~$\frac{1}{2}\II\in(\frac{1}{2}\ZZ/\ZZ)^g$.
\end{enumerate}
\end{lem}
\begin{proof}
\new{To check the first point, consider an element~$t\in\Jac(\Sigma)=\CC^g/\Lambda$, represented by~$x+y\in\CC^g$ with~$x\in\RR^g$ and~$y\in i\RR^g$. This element~$t$ of~$\Jac(\Sigma)$ is real if and only if the difference~$(x+y)-(\overline{x+y})=2y$ belongs to~$\Lambda=\ZZ^g\oplus\Omega\ZZ^g$, which is equivalent to~$y\in\Omega\frac{1}{2}\ZZ^g$. In conclusion, the real locus of~$\Jac(\Sigma)$ is indeed given by
\[
(\RR^g\oplus{\textstyle{\frac{1}{2}}}\Omega\ZZ^g)/\Lambda=(\RR/\ZZ)^g\oplus\Omega({\textstyle{\frac{1}{2}}}\ZZ/\ZZ)^g\,.
\]}
We now check that the real components~$A_1,\dotsc,A_g$ are mapped to
this real locus of~$\Jac(\Sigma)$.
Indeed, fix any~$P_i\in A_i$. A path from~$x_0$ to~$P_i$ can be chosen
as a first path~$\gamma_0\subset A_0$ from~$x_0$ to the intersection of~$A_0$
with~$B_i$, then a path~$\beta_i\subset B_i$ \new{(following the orientation of~$B_i$)} to the intersection of~$B_i$ with~$A_i$, and a path~$\gamma_i\subset A_i$ to~$P_i$.
As in the proof of Lemma~\ref{lem:imaginary}, we can compute
\[
(\new{\Phi(P_i)-\overline{\Phi(P_i)}})_j=\int_{\new{\gamma_0-\sigma_*\gamma_0}}\omega_j+\int_{\new{\beta_i-\sigma_*\beta_i}}\omega_j+\int_{\new{\gamma_i-\sigma_*\gamma_i}}\omega_j
=\int_{B_i}\omega_j=\new{\Omega_{ji}}\,.
\]
Hence, we see that~$\new{\Phi(P_i)-\overline{\Phi(P_i)}}$ belongs to~$\Omega\ZZ^g$,
and so $\Phi(P_i)$ and $\overline{\Phi(P_i)}$ define the same element of~$\Jac(\Sigma)$.
More precisely, since~$\new{\Phi(P_i)-\overline{\Phi(P_i)}}=\Omega e_i$, we see that~$A_i$ is mapped
inside the real torus indexed by~$\frac{e_i}{2}\in(\frac{1}{2}\ZZ/\ZZ)^g$.
Moreover, going once around~$A_i$ replaces~$\Phi(P)$ with~$\Phi(P)+\int_{A_i}\vec\omega=\Phi(P)+e_i$.
Hence, the component~$A_i$ is mapped to a cycle in this real torus with homology
class~$e_i\in\ZZ^g=H_1((\RR/\ZZ)^g;\ZZ)$.
The fact that it is strictly increasing in the~$e_i$-direction is a reformulation of Lemma~\ref{lem:non-zero},
and the second point is proved.

With our choice of base point~$x_0$ in~$A_0$, the real component~$A_0$ also clearly
embeds into the real locus
of the Jacobian via~$\Phi$, and its image~$\Phi(A_0)$ contains the origin.
Hence, the component~$A_0$ embeds in the real torus indexed by~$0\in(\frac{1}{2}\ZZ/\ZZ)^g$.
Note also that since the homology class of~$A_0$
is given by~$A_1+\dotsb +A_g$, going once around
the component~$A_0$ replaces~$\Phi(P)$ by~$\Phi(P)+\mathbf{1}$.
Therefore,~$\Phi(A_0)$ is a cycle in this real torus with homology class~$\II\in\ZZ^g=H_1((\RR/\ZZ)^g;\ZZ)$.
The monotonicity follows from Lemma~\ref{lem:non-zero}, showing the third point.

Understanding an element of~$A_1\times\dotsb\times A_g$ as a divisor~$P_1+\dotsb+P_g$ with~$P_i\in A_i$,
the last point now follows from the third one: indeed, the restriction of~$\Phi$ to such divisors
defines a map from the real torus~$A_1\times\dotsb\times A_g$ to the real torus~$(\RR/\ZZ)^g+\Omega\frac{1}{2}\II$, a map of degree~1, hence surjective. The injectivity follows from the monotonicity,
and the homeomorphism from compactness.
\end{proof}

We conclude this section by recalling the classical \emph{Riemann bilinear relation}, in the form stated in~\cite[Eq.~(7)]{Fay}, but accounting for the different normalisation of~$\omega_k$.
(This result is also known as a \emph{reciprocity law}, see e.g.~\cite[Theorem~5.3.1~ii]{Jost_2006}.)
Let~$\omega_D$ be a differential 1-form of the third kind
(that is, a meromorphic differential 1-form having only simple poles~\cite[Section~5.3]{Jost_2006}) with zero period along the~$A$-cycles and simple poles at~$\beta_j$ with integer residue~$r_j\in\mathbb{Z}^*$,
for~$1\leq j\leq n$.

Note that the corresponding degree~$0$ divisor~$D=\sum_{j=1}^n  r_j \beta_j$ splits as~$D=D^+-D^-$, where
  \begin{equation*}
	D^+=\sum_{j:\,r_j >0} r_j
    \beta_j\quad\text{and}\quad D^- =\sum_{j:\,r_j<0} (-r_j) \beta_j
  \end{equation*}
are \emph{effective divisors} (that is, with positive coefficients).
Then, for any $1\leq k \leq g$, we have the equality
  \begin{equation}\label{eq:RiemannBilin}
    \int_{B_k} \omega_D = 2i\pi\int_{D^-}^{D^+} \omega_k\,,
  \end{equation}
 where the paths of integration for the right-hand side between pairs of points
  of $D^-$ and $D^+$ are paths in the surface~$\Sigma$ cut along~$\{A_i, B_i:\,1\leq i\leq g\}$,
see Figure~\ref{fig:cutopen}.

\subsection{Theta functions}
\label{sub:theta}

In this section, we recall the definition of the Riemann theta functions,
following the conventions of~\cite[Chapter II]{ThetaTata1}
(see also~\cite[Chapter~VI]{Farkas-Kra}),
and state their basic properties in Lemma~\ref{lem:ident_1}.
In the case of a purely imaginary period matrix, more subtle properties are proved in Lemma~\ref{lem:ident_2}.

The \emph{Riemann theta function} $\theta(z|\Omega)$ associated with a Riemann surface~$\Sigma$ is a
higher-dimensional analog of the classical Jacobi theta functions~\cite{Lawden}.
For $z\in\mathbb{C}^g$, set
\begin{equation*}
  \theta(z|\Omega)=\sum_{n\in\ZZ^g} e^{i\pi(n\cdot \Omega n+2n\cdot z)},
\end{equation*}
where $\cdot$ represent the canonical scalar product in $\mathbb{C}^g$.
For~$\binom{\delta'}{\delta''}\in(\frac{1}{2}\mathbb{Z})^{2g}$, the
\emph{theta function with characteristic $\binom{\delta'}{\delta''}$}, denoted by
$\theta\thchar{\delta'}{\delta''}$, is defined by
\begin{align*}
  \theta\thchar{\delta'}{\delta''}(z|\Omega)
  &= \sum_{n\in\mathbb{Z}^g}e^{i\pi[(n+\delta')\cdot \Omega(n+\delta')+2(n+\delta')\cdot(z+\delta'')]}.
\end{align*}

\begin{exm}\leavevmode
\begin{enumerate}
\item The theta function with characteristic $\binom{0}{0}$ is the Riemann theta function $\theta$ defined above.
\item When $g=1$ and $\Omega=\tau$, $\theta(z|\Omega)$ coincides
with the Jacobi function $\theta_3(z\pi|\tau)$, see
~\cite[Equation~(1.2.13)]{Lawden}.
The theta functions corresponding to the four characteristics
$\binom{0}{0},\binom{0}{\frac{1}{2}},\binom{\frac{1}{2}}{0},\binom{\frac{1}{2}}{\frac{1}{2}}$ are
the rescaled versions of $\theta_3,\theta_4,\theta_2,-\theta_1$ respectively.
\end{enumerate}
\label{ex:theta}
\end{exm}

The following elementary identities between theta functions are well-known, see e.g.~\cite{Farkas-Kra}.

\begin{lem}\label{lem:ident_1}\leavevmode
\begin{enumerate}
 \item
 For all~$\binom{\delta'}{\delta''}\in(\frac{1}{2}\mathbb{Z})^{2g}$, we have the equality
 \[\theta\thchar{\delta'}{\delta''}(z|\Omega)=
 e^{i\pi(\delta'\cdot \Omega\delta'+2\delta'\cdot(z+\delta'')) }
 \theta(z+\Omega\delta'+\delta''|\Omega).\]
 \item
For all~$m,n\in\ZZ^g$, we have
 \[
\theta\thchar{\delta'}{\delta''}(z+m+ \Omega n|\Omega)=
e^{-i\pi n\cdot(2z +2\delta''+\Omega n)}e^{2i\pi\delta'\cdot m}\theta\thchar{\delta'}{\delta''}(z|\Omega)\,.
 \]
In particular, the function $\theta$ is periodic in the $\ZZ^g$ directions, and quasi-periodic in the $\Omega\ZZ^g$ directions:
\[
\theta(z+m+ \Omega n|\Omega)=e^{-i\pi n\cdot(2z+\Omega n)}\theta(z|\Omega).
\]
\item
For all~$\gamma',\gamma''\in\frac{1}{2}\ZZ^g$, we have
\[
\theta\thchar{\delta'}{\delta''}(z+\gamma''+ \Omega \gamma'|\Omega)=
e^{-i\pi \gamma'\cdot (2z+2\delta''+2\gamma''+\Omega \gamma')}\theta\thchar{\delta'+\gamma'}{\delta''+\gamma''}(z|\Omega).
\]
\item
For all~$m,n\in\ZZ^g$, we have
\[
\theta\thchar{\delta'+n}{\delta''+m}(z|\Omega)=e^{2i\pi \delta'\cdot m}\theta\thchar{\delta'}{\delta''}(z|\Omega)=\pm \theta\thchar{\delta'}{\delta''}(z|\Omega)\,.
\]
This justifies the notation~$\binom{\delta'}{\delta''}\in(\frac{1}{2}\mathbb{Z}/\mathbb{Z})^{2g}$, that we will now use \new{even though in practice, we always work with fixed representatives in~$\frac{1}{2}\mathbb{Z}^{2g}$}.
\item For every $z\in\CC^g$, we have
\[
\theta\thchar{\delta'}{\delta''}(-z|\Omega)=(-1)^{2\delta'\cdot 2\delta''}
\theta\thchar{\delta'}{\delta''}(z|\Omega),
\]
implying that $\theta\thchar{\delta'}{\delta''}$ is even,
resp. odd, if and only if~$2\delta'\cdot 2\delta''$ is even, resp. odd.
\end{enumerate}
\end{lem}

As showed in Lemma~\ref{lem:imaginary}, the matrix~$\Omega$ associated with an~M-curve
is purely imaginary. We will need the following properties of the corresponding theta functions.
The second one is of great importance to our work.

\begin{lem}\label{lem:ident_2}
Let us assume that~$\Omega$ is purely imaginary.
\begin{enumerate}
\item For all~$z\in\CC^g$, we have
    \begin{equation*}
      \theta\thchar{\delta'}{\delta''}(\bar{z}|\Omega)=
      \overline{\theta\thchar{\delta'}{\delta''}(z|\Omega)}\,.
    \end{equation*}
In particular,~$\theta\thchar{\delta'}{\delta''}(z|\Omega)$ is real for~$z\in\RR^g$.
\item If~$z$ belongs to~$\RR^g$,
then~$\theta\thchar{0}{\delta''}(z|\Omega)$ is strictly positive.
\item When $\delta'\neq 0$, $\theta\thchar{\delta'}{\delta''}(z|\Omega)$ takes
  strictly positive and negative values on $\mathbb{R}^g$.
\end{enumerate}
\end{lem}

\begin{proof}\leavevmode
\begin{enumerate}
\item Let us check the equality
\[
\overline{\theta\thchar{\delta'}{\delta''}(z|\Omega)}=(-1)^{2\delta'\cdot 2\delta''}\theta\thchar{\delta'}{\delta''}(-\overline{z}|\Omega)
\]
which, together with the last point of Lemma~\ref{lem:ident_1}, implies the first point of the statement.
Using that $\Omega$ is purely imaginary, we have
\begin{align*}
\overline{\theta\thchar{\delta'}{\delta''}(z|\Omega)}&=\sum_{n\in\mathbb{Z}^g}e^{i\pi[(n+\delta')\cdot \Omega(n+\delta')+2(n+\delta')\cdot(-\overline{z}-\delta'')]}\\
&=\sum_{n\in\mathbb{Z}^g}e^{i\pi[(n+\delta')\cdot \Omega(n+\delta')+2(n+\delta')\cdot(-\overline{z}+\delta'')]}
e^{-4i\pi(n+\delta')\cdot\delta''},
\end{align*}
and the proof is concluded using that $\delta''$ belongs to $(\frac{1}{2}\ZZ/\ZZ)^g$.

\item
By Point~1 of Lemma~\ref{lem:ident_1}, it is enough to prove the statement for
$\delta''=0$, \emph{i.e.}, for the Riemann theta function $\theta$.
By Point~2 of Lemma~\ref{lem:ident_1} and the first point above, for any~$t>0$
and~$\Omega$ purely imaginary, the real function~$z\mapsto\theta(z|t\Omega)$ is
well defined on the torus~$(\mathbb{R}/\mathbb{Z})^g$.
Furthermore, it converges as~$t\to 0$, in the sense of distributions, to the Dirac distribution on the torus
whose Fourier coefficients are all equal to~1.
Moreover, it satisfies the heat equation
  \begin{equation*}
  \frac{\partial}{\partial t} f = -\frac{1}{2}\Delta f\,,
\end{equation*}
  where~$\Delta=-\frac{1}{2i\pi}\sum_{j,k}\Omega_{j,k}\frac{\partial^2}{\partial
  z_j\partial z_k}$ is a positive definite Laplace operator (in a
  non-orthonormal system of coordinates) on the torus.
This means that~$\theta$ is the fundamental solution of this heat equation.
By the maximum principle,
we conclude that~$\theta(z|t\Omega)$ is strictly positive for any~$z$ and any~$t>0$, therefore in particular for~$t=1$.
(From a more probabilistic point of view, the function~$(z,z',t)\mapsto\theta(z'-z|t\Omega)$ is the transition kernel
of a non isotropic Brownian motion on the torus with no drift and a diffusivity
matrix given by~$\frac{1}{2i\pi}\Omega$, in other words, of the linear image of
a standard Brownian motion by a matrix~$A$ such that~$AA^T=\frac{1}{2i\pi}\Omega$.)

\item To show the last point, let us fix~$\delta'\neq 0$.
By Point~2 of Lemma~\ref{lem:ident_1}, we have
\[
\theta\thchar{\delta'}{\delta''}(z+m|\Omega)=e^{2i\pi\delta'\cdot m}\theta\thchar{\delta'}{\delta''}(z|\Omega).
\]
Now since $\delta'\neq 0$, it has at least one coefficient equal to~$\frac{1}{2}$, say the~$i^{\mathrm{\new{th}}}$; take~$m_i$ to be equal to zero except at position~$i$ where it is equal to 1. Then,
\[
\theta\thchar{\delta'}{\delta''}(z+m_i|\Omega)=-\theta\thchar{\delta'}{\delta''}(z|\Omega).
\]
Since~$\theta\thchar{\delta'}{\delta''}$ is real-valued on~$\mathbb{R}^g$ by the first point above, and
not identically zero, this concludes the proof.
\qedhere
\end{enumerate}
\end{proof}

From now on, the
matrix~$\Omega$ being fixed once and for all,
we simply write~$\theta(z|\Omega)$ as~$\theta(z)$.

We conclude this section by recalling a fundamental result, known as \emph{Riemann's theorem},
following~\cite[Theorem~3.1, p.~149]{ThetaTata1}.
Consider a fixed point~$x_0\in\Sigma$, and a lift~$\wt{x}_0$ in the universal cover~$\widetilde{\Sigma}$ of~$\Sigma$.
For any~$e\in\mathbb{C}^g$, the function~$f_{e}\colon\widetilde\Sigma\to\CC$ given
by~$f_e(\wt{x})=\theta(e+\int_{\wt{x}_0}^{\wt{x}}\vec{\omega})$ does not induce a well-defined function on~$\Sigma$,
because of~$\theta$ being only quasi-periodic. However, its zeros form a
periodic subset of~$\widetilde\Sigma$ which has a well-defined projection on~$\Sigma$.
A precise description of this set, called the \emph{theta divisor of~$\Sigma$},
is given as follows. There exists an element~$\Delta\in\mathbb{C}^g$ (depending on the choice of~$x_0$),
such that for any~$e$, if~$f_{e}$ is not identically equal to~$0$, then it
admits~$g$ zeros~$x_1,\dotsc,x_g\in\Sigma$ which satisfy the following equality in~$\Jac(\Sigma)$:
\begin{equation}
  \sum_{j=1}^g \int_{x_0}^{x_j}\vec{\omega} = -e + \Delta\,.
  \label{eq:RiemannDelta}
\end{equation}
Moreover, the theta divisor~$x_1+\dotsb+x_g$ is uniquely determined by this condition, see~\cite[Corollary~3.2, p.~153]{ThetaTata1}.
Note that the points~$x_j$ may be not distinct; in that case, they correspond to zeros with higher multiplicity, so that the total degree of the theta divisor is~$g$.

More can be said in the case of an M-curve~$\Sigma$ with fixed point~$x_0\in A_0$ and~$e\in\RR^g$.

\begin{lem}
  \label{lem:Delta}
Let~$\Sigma$ be an M-curve, and let~$x_0$ be an element of~$A_0$.
Then, for every~$e\in\RR^g$, each of the zeros~$x_j$ of the function~$f_e$ belongs to a different~$A_i$ with~$1\leq i\leq g$.
Thus, these zeros are distinct and satisfy~$x_j\in A_j$ for all~$1\leq j\leq g$
up to relabeling.
Consequently, the constant~$\Delta$ belongs to $\RR^g+\Omega\frac{1}{2}\mathbf{1}$.
\end{lem}

\begin{proof}
Let us fix an element~$x_0\in A_0$ and a real vector~$e\in\RR^g$.
For~$1\leq j\leq g$, let~$\delta_j\in\mathbb{R}^g$ denote the vector
whose coordinates are all zero, except the $j^{\mathrm{th}}$ one, which is equal to~$\frac{1}{2}$.
Finally, let~$\wt{x}_0\in\wt{\Sigma}$ be an arbitrary lift of~$x_0\in A_0\subset\Sigma$.
We now show that the function~$f_e\colon\wt{\Sigma}\to\CC$ given by
$f_e(\wt{x})=\theta\left(e+\int_{\wt{x}_0}^{\wt{x}}\vec{\omega}\right)$
vanishes at least once on any given lift~$\wt{A}_j\subset\wt{\Sigma}$ of the real component~$A_j\subset\Sigma$.

  Let~$\gamma_j\colon\mathbb{R}\rightarrow \wt{A}_j$ be a lift of a
  parametrization of~$A_j$ by~$\mathbb{R}/\mathbb{Z}$.
  By Lemma~\ref{lemma:real}, Point~2, the integral
  \begin{equation*}
  \int_{\wt{x}_0}^{\gamma_j(t)} \vec\omega
  \end{equation*}
  belongs to~$\mathbb{R}^g+\Omega(\delta_j+\mathbb{Z}^g)$ for every~$t\in\mathbb{R}$.
By continuity,
it belongs to~$\mathbb{R}^g+\Omega(\delta_j+v)$ for some fixed~$v\in\ZZ^g$.
 Moreover, this same Point~2 of Lemma~\ref{lemma:real} implies the equality
  \begin{equation*}
  \int_{\wt{x}_0}^{\gamma_j(t+1)} \vec\omega =
  \int_{\wt{x}_0}^{\gamma_j(t)} \vec\omega +2\delta_j
  \end{equation*}
for all~$t\in\RR$. Hence, by Lemma~\ref{lem:ident_1}, Point~2,  the function
  \begin{equation*}
    h_j:t\mapsto
    \theta\thchar{\delta_j}{0}\Bigl(e-\Omega(\delta_j+v)+\int_{\wt{x}_0}^{\gamma_j(t)}\vec{\omega}\Bigr)
\end{equation*}
satisfies~$h_j(t+1)=-h_j(t)$ for all $t\in\mathbb{R}$. Finally, this function is real by Lemma~\ref{lem:ident_2}, Point~1.
Therefore, there is a point~$\wt{x}_j=\gamma_j(t_j)$ in~$\wt{A}_j$ such that~$h_j(\wt{x}_j)=0$.
We then conclude by Lemma~\ref{lem:ident_1}, Points~1 and~2, that
\begin{equation*}
  \left|f_e(\wt{x}_j)\right|=\Bigl|\theta\Bigl(e+\int_{\wt{x}_0}^{\wt{x}_j}\vec{\omega}\Bigr)\Bigr|=
  \Bigl|\theta\thchar{\delta_j}{0}\Bigl(e-\Omega(\delta_j+v)+\int_{\wt{x}_0}^{\wt{x}_j}\vec{\omega}\Bigr)\Bigr|=
\left|h_j(\wt{x}_j)\right|=0.
\end{equation*}
This point~$\wt{x}_j$ is thus a zero of $f_e$, which projects on~$\Sigma$ to a point~$x_j\in A_j$ contributing to the theta divisor of~$\Sigma$.
The claim follows from the fact that this divisor is of degree~$g$.

  The constant~$\Delta$ belonging to~$\RR^g+\Omega\frac{1}{2}\mathbf{1}$ is now a direct consequence of
  Point~4 of Lemma~\ref{lemma:real}, together with
  Equation~\eqref{eq:RiemannDelta} applied to~$e=0$.
\end{proof}

\subsection{Prime form}
\label{sub:prime}

This section deals with the so-called \emph{prime form} associated with an arbitrary compact Riemann surface~$\Sigma$.
In Section~\ref{subsub:linebundles}, we start by studying the general theory of sections of line bundles,
a necessary formalism for the precise definition of the prime form given in Section~\ref{subsub:prime}.
When~$\Sigma$ is an~M-curve, this form exhibits special properties that are stated and proved
in Lemmas~\ref{lem:primeform_conj} and~\ref{lem:prime}.
Finally, Section~\ref{subsub:Fay} recalls Fay's trisecant identity,
which plays a crucial role in the rest of our work.

\subsubsection{Sections of line bundles and automorphic forms}
\label{subsub:linebundles}

The aim of this preliminary paragraph is to prove a statement that is most probably
standard: holomorphic sections of a fixed holomorphic line bundle on a Riemann surface~$\Sigma$ of
genus~$g>0$ can be understood as \emph{automorphic forms} on the universal cover~$\wt\Sigma$ of~$\Sigma$, \emph{i.e.}, holomorphic functions
with quasi-periodicity under the action of~$\pi_1(\Sigma)$ prescribed by the line bundle.

To check this fact, fix a holomorphic line bundle~$p\colon L\to\Sigma$
and denote by~$\pi\colon\wt\Sigma\to\Sigma$ the universal cover of~$\Sigma$.
This gives rise to the holomorphic line bundle~$\pi^*p\colon E\to\wt\Sigma$ induced by~$\pi$.
Since we assume~$g>0$, the Riemann surface~$\wt\Sigma$ is isomorphic to the open disc or the complex plane.
Since it is non-compact, the Weierstrass theorem implies that~$H^1(\wt\Sigma,\mathcal{O}^*)$ vanishes,
so all holomorphic line bundles on~$\wt\Sigma$ are trivial. In particular, there is an
isomorphism~$\varphi\colon\wt\Sigma\times\CC\to E$ such that~$\pi^*p\circ \varphi$ is equal to
the projection~$\mathit{pr}\colon\wt\Sigma\times\CC\to\wt\Sigma$.
In a nutshell, we have the commutative diagram
\[
\begin{xy}(-20,15)*+{\wt\Sigma\times\CC}="e";
(0,15)*+{E}="a"; (20,15)*+{L}="b";
(0,0)*+{\wt\Sigma}="c"; (20,0)*+{\Sigma,}="d";
{\ar "a";"b"}?*!/_2mm/{\Pi};{\ar "a";"c"}?*!/_4mm/{\pi^*p};{\ar "b";"d"}?*!/_2mm/{p};
{\ar "c";"d"}?*!/_2mm/{\pi};{\ar "e";"a"}?*!/_2mm/{\varphi};
{\ar "e";"c"}?*!/^4mm/{\mathit{pr}};
\end{xy}
\]
where~$\Pi\colon E\to L$ is the~$\pi_1(\Sigma)$-covering induced by~$p$.

Now, consider a holomorphic section~$\psi\colon\Sigma\to L$, which by definition
satisfies~$p\circ\psi=\mathit{id}$.
By the lifting property of the~$\pi_1(\Sigma)$-covering~$\Pi\colon E\to L$,
there is a map~$\wt\psi\colon\wt\Sigma\to E$ such that~$\Pi\circ\wt\psi=\psi\circ\pi$,
which is uniquely determined by
\new{its value on a lift~$\wt{x}_0\in\wt{\Sigma}$}
of a base point~$x_0\in\new{\Sigma}$.
Hence, we have
\[
\pi\circ\pi^*p\circ\wt\psi=p\circ\Pi\circ\wt\psi=p\circ\psi\circ\pi=\pi\,,
\]
and we see that~$\wt\psi$ is a holomorphic
section of~$\pi^*p$ up to action of the group~$\pi_1(\Sigma)$ of the covering~$\pi$.
Moreover, one can choose \new{the value~$\wt{\psi}(\wt{x}_0)\in E$} so that~$(\pi^*p\circ\wt{\psi})(\new{\wt{x}}_0)=\new{\wt{x}}_0$.
By continuity, we have~$\pi^*p\circ\wt{\psi}=\mathit{id}$, and~$\wt\psi$ is a holomorphic
section of~$\pi^*p$.
The composition~$\varphi^{-1}\circ\wt{\psi}$ now defines a holomorphic section of the trivial
bundle~$\mathit{pr}\colon\wt\Sigma\times\CC\to\wt\Sigma$, \emph{i.e.}, ~it is of the
form~$(\varphi^{-1}\circ\wt{\psi})(\new{\wt{x}})=(\new{\wt{x}},f_\psi(\new{\wt{x}}))$ for some holomorphic function
\[
f_\psi\colon\wt\Sigma\longrightarrow\CC
\]
uniquely determined by~$\psi$.

Moreover, its quasi-periodicity properties under the action of~$\pi_{\new{1}}(\Sigma)$
are uniquely determined by the holomorphic line bundle~$p\colon L\to\Sigma$. They are of the following form:
for~$\new{\wt{x}}\in\wt{\Sigma}$ and~$\gamma\in\pi_1(\Sigma)$, we have
\[
f_\psi(\gamma\cdot \new{\wt{x}})=p_\gamma(\new{\wt{x}}) f_\psi(\new{\wt{x}})
\]
for some holomorphic map~$p_\gamma\colon\wt\Sigma\to\CC^*$ \new{which only depends on~$\gamma$ and on the line bundle~$L\to\Sigma$.}
This is the \emph{factor of automorphy} of the automorphic form~$f_{\new{\psi}}$.

\begin{exm}
\label{ex:spin}
Consider the case of a spin structure~$L\to\Sigma$ with~$\Sigma=\TT(\tau)$.
It is a square root of the canonical line bundle, which is nothing but the trivial bundle.
Since~$\pi_1(\Sigma)=\ZZ+\ZZ\tau$ acts by translations on~$\wt{\Sigma}=\CC$,
the discussion above shows that holomorphic sections of~$L\to\Sigma$ (\emph{i.e.}, ~spinors)
can be understood as holomorphic maps~$f$ on~$\wt{\Sigma}=\CC$ with quasi-periodicity
\[
f(x+1)=p_1(x)f(x)\quad\text{and}\quad f(x+\tau)=p_\tau(x)f(x)
\]
for some~$p_1,p_\tau\colon\CC\to\CC^*$ with~$p_1^2=p_\tau^2=1$.
This gives~$4$ different spin structures corresponding
to the~$4$ possible signs~$p_1,p_\tau\in\{\pm 1\}$.
Note that the choice~$p_1=p_\tau=1$ leads to constant spinors,
while the other choices only allow for identically zero holomorphic sections.
This is coherent with the fact that the dimension of the space of spinors has the same parity as the spin
structure~\cite{Atiyah}. In the present case, the only odd spin structure corresponds to the trivial line bundle.
\end{exm}

\subsubsection{The prime form}
\label{subsub:prime}

We now give a definition with all necessary details for our purposes, but very little more,
referring the reader to~\cite{ThetaTata2} for additional information.

Given an arbitrary compact Riemann surface~$\Sigma$, let us fix a \emph{non-degenerate theta characteristic},
\emph{i.e.}, a theta characteristic such that the corresponding theta function satisfies
\begin{equation*}
  \ud_z\theta\thchar{\delta'}{\delta''}(0)\neq 0\,.
\end{equation*}
Such a theta characteristic is know\new{n} to hold by the Lefschetz embedding theorem.
Note that~$\binom{\delta'}{\delta''}$ must be odd, and thus also satisfy~$\theta\thchar{\delta'}{\delta''}(0)=0$.

By the general theory of spin structures on Riemann surfaces~\cite{Atiyah},
this theta characteristic corresponds to a spin structure, understood as a line bundle~$L$ whose
square is isomorphic to the canonical line bundle~$K$.
An explicit holomorphic section of this bundle $L$ can be constructed as follows.
Consider the holomorphic form
\[
\zeta=\sum_{i=1}^{g}\frac{\partial\theta\thchar{\delta'}{\delta''}}{\partial
z_i}(0) \omega_i
= \ud_z\theta\thchar{\delta'}{\delta''}(0)\cdot\vec\omega
\,,
\]
which is nothing but the differential of the function
$y\mapsto\theta\thchar{\delta'}{\delta''}(y-x)\new{:=}\theta\thchar{\delta'}{\delta''}(\int_x^y\vec\omega)$,
in the variable $y\in\Sigma$, evaluated at $y=x$.
All the zeros of~$\zeta$ are double zeros~\cite[Corollary~1.3]{Fay}.
Therefore, interpreting~$\zeta$ as a holomorphic section of~$K$, it admits a square
root~$\xi\thchar{\delta'}{\delta''}$ which is a section of the line bundle~$L$.
Note that~$\xi\thchar{\delta'}{\delta''}$ depends on
the theta characteristic in two ways:
in the construction of~$\zeta$, and in the choice of the square root.

\begin{defi}
\label{def:prime}
The \emph{prime form} is the form~$E$ defined by
\begin{equation*}
  E(x,y)=\frac{\theta\thchar{\delta'}{\delta''}(y-x)}{\xi\thchar{\delta'}{\delta''}(x)\xi\thchar{\delta'}{\delta''}(y)}\,.
\end{equation*}
for~$x,y$ in the universal cover~$\wt{\Sigma}$ of~$\Sigma$.
\end{defi}

Let us give an explicit example.

\begin{exm}
\label{ex:prime}
If~$\Sigma$ has genus~$1$, it is isomorphic to a torus~$\TT(\tau)$ with~$\Im(\tau)>0$. There is
a unique odd theta characteristic, namely~$\binom{\frac{1}{2}}{\frac{1}{2}}$, which is always
non-degenerate. Hence, the numerator of the prime form is given
by~$\theta\thchar{\frac{1}{2}}{\frac{1}{2}}(y-x)=-\theta_1(\pi(y-x))$, see Example~\ref{ex:theta}.
As for the denominator, recall that the canonical line bundle of the torus is trivial.
Furthermore, the square root~$L$ of~$K=1$ corresponding to the odd theta characteristic is the
trivial bundle~$L=1$ (recall Example~\ref{ex:spin}).
The constant value of~$\xi\thchar{\frac{1}{2}}{\frac{1}{2}}$ is given by a square root
of~$-\pi\theta_1'(0)$, leading to the explicit formula~$E(x,y)=\frac{\theta_1(\pi(y-x))}{\pi\theta_1'(0)}$.
\end{exm}

The prime form is the key ingredient for constructing meromorphic functions on~$\Sigma$.
Indeed, if~$x_1,\dotsc,x_k,y_1,\dotsc,y_k$ are points on~$\Sigma$ with~$x_i\neq y_j$ for all~$i,j$
and such that~$\sum_{i=1}^k (x_i-y_i)$ is a principal divisor, then all meromorphic
functions with this divisor are of the form
\begin{equation*}
  g(x)= c\times \prod_{i=1}^k \frac{E(x_i,x)}{E(y_i,x)}\,,
\end{equation*}
with~$c$ a complex number.

We will make use of the following standard properties of the prime form, valid for any~$x,y\in\wt{\Sigma}$
(see e.g.~\cite[p.~3.210]{ThetaTata2}):
\begin{itemize}
\item $E(x,y)=0$ if and only if~$x$ and~$y$ project to the same point on~$\Sigma$;
\item these are first order zeros;
\item When $x$ and $y$ are close to each other,
\begin{equation*}
  E(x,y) =
  \frac{z(y)-z(x)}{\sqrt{dz(x)}\sqrt{dz(y)}}\left(1+O(z(x)-z(y))^2\right),
\end{equation*}
where $z$ is a local coordinate such that $\zeta=dz$ in a connected open set
containing $x$ and $y$.
\item $E(x,y)=-E(y,x)$.
\end{itemize}

The following result is also well-known, but most references only consider equalities up to signs.
Since signs do play an important role in our setting, we include the proof for completeness.

\begin{prop}
\label{prop:prime}
The prime form does not depend on the choice of the non-degenerate theta characteristic.
\end{prop}
\begin{proof}
Let us start by studying the quasi-periodicity of~$E(x,y)$.
If~$y'\in\wt\Sigma$ is obtained from~$y$ by adding the corresponding lift of the cycle~$A_j$, we have
\[
\theta\thchar{\delta'}{\delta''}(y'-x)=\theta\thchar{\delta'}{\delta''}((y-x)+\int_{A_j}\vec{\omega})=
\theta\thchar{\delta'}{\delta''}((y-x)+e_j)=e^{2i\pi\delta_j'}\theta\thchar{\delta'}{\delta''}(y-x)
\]
by Lemma~\ref{lem:ident_1}.
On the other hand and by definition, we
have~$\xi\thchar{\delta'}{\delta''}(y')=p_j\thchar{\delta'}{\delta''}(y)\xi\thchar{\delta'}{\delta''}(y)$,
with~$p_j\thchar{\delta'}{\delta''}(y)^2=:p_j^2(y)$ encoding the~$A_j$-quasi-periodicity of any holomorphic form.
Recall however that the sign of this square root does depend on the theta characteristic, hence the heavy notation.
Fixing an arbitrary square root~$p_j(y)$, we see that
\begin{equation}
  \label{eq:primeformA}
E(x,y')=\pm
(-1)^{2\delta'_j}
p_j^{-1}(y) E(x,y)\,,
\end{equation}
with~$p_j$ independent of the theta characteristic, while the sign might a priori depend on it.
(Note that even in the genus~$1$ case, we have~$E(x,y')=-E(x,y)$, recall Example~\ref{ex:prime};
therefore, and unlike claimed in many references, the prime form
is in general not invariant along~$A$-cycles.)
Similarly, if~$y''\in\wt\Sigma$ is obtained from~$y$ by adding the lift of the cycle~$B_j$, then we have
\begin{equation}
  \label{eq:primeformB}
E(x,y'')=\pm
(-1)^{2\delta''_j}
q_j^{-1}(y) e^{2i\pi(-\frac{\Omega_{jj}}{2}-\int_x^y\omega_j)} E(x,y)\,,
\end{equation}
with~$q_j^2(y)$ encodes the~$B_j$-quasi-periodicity of any holomorphic form.
The crucial point here is that, once again,
the factor appearing does not depend on the theta characteristic, at least up to sign.

Let us now consider prime forms~$E$ and~$E'$ obtained from two non-singular theta characteristics.
For any fixed~$x\in\wt{\Sigma}$, set~$f_x(y):=\frac{E(x,y)}{E'(x,y)}$.
By the fundamental properties of~$E$ listed above, both the numerator and denominator have the same zeros,
namely a simple zero at each element of~$\wt\Sigma$ with same image in~$\Sigma$ as~$x$.
Therefore, this quotient defines a (non-vanishing) holomorphic function on~$\wt{\Sigma}$.
Moreover, by the equalities displayed above, the~$A$ and~$B$ periods of~$f_x$ are trivial up to a sign.
Therefore, the function~$f_x$ induces a well-defined holomorphic function on the cover~$\Sigma'$
of~$\Sigma$ given by
the homomorphism~$\pi_1(\Sigma)\to(\ZZ/2\ZZ)^{2g}$ mapping the cycles~$A_j$ and~$B_j$ to distinct
elements of the canonical basis of~$(\ZZ/2\ZZ)^{2g}$.
This cover being finite, the Riemann surface~$\Sigma'$ is compact, and~$f_x$ constant.
The normalization of~$E$ and~$E'$ ensures that they have the same asymptotic behavior near the diagonal,
and hence that this constant is equal to~$1$. This completes the proof.
\end{proof}

\begin{rem}
    The statement above implies that the factor of automorphy for the prime form
    does not depend on $\binom{\delta'}{\delta''}$. It means that there is a
    bijection between the collection of signs
    $\{(-1)^{2\delta'_j},(-1)^{2\delta''_j}\}_{1\leq j\leq g}$ and the
    collection of $\pm$ coming from the square roots of $p_j^2(y)$ and
    $q_j^2(y)$ we need to consider for the line bundle $L$, and if we choose
    correctly the sign of ``reference" square roots $p_j(y)$ and $q_j(y)$, these
    collections of signs can be taken to be equal.
\end{rem}

We now come back to the setting of~M-curves.
Note that the anti-holomorphic involution~$\sigma\colon\Sigma\to\Sigma$ lifts to an anti-holomorphic
involution on the universal cover~$\wt{\Sigma}$, that we also denote by~$\sigma$.

\begin{lem}
\label{lem:primeform_conj}
If~$\Sigma$ is an~M-curve, then
the associated prime form satisfies
  \begin{equation*}
    \forall\,x,y\in\widetilde{\Sigma},\quad
  E(\sigma(x),\sigma(y))=\overline{E(x,y)}\,.
\end{equation*}
\end{lem}
\begin{proof}
First note that by Points~1 and~3 of Lemma~\ref{lem:imaginary}
together with Point~1 of Lemma~\ref{lem:ident_2}, we have
\[
\theta\thchar{\delta'}{\delta''}\Big(\int_{\sigma(x)}^{\sigma(y)}\vec{\omega}\Big)
=\theta\thchar{\delta'}{\delta''}\Big(\int_x^y\sigma^*\vec{\omega}\Big)
=\theta\thchar{\delta'}{\delta''}\Big(\overline{\int_x^y\vec{\omega}}\Big)
=\overline{\theta\thchar{\delta'}{\delta''}\Big(\int_x^y\vec{\omega}\Big)}\,.
\]
Furthermore, the theta characteristic being odd
and the period matrix purely imaginary, the
number~$\frac{\partial}{\partial z_j}\theta\thchar{\delta'}{\delta''}(0)$ is
easily seen to be real.
An additional use of the first point of Lemma~\ref{lem:imaginary} then leads to
the identity~$\sigma^*\zeta=\overline\zeta$, which implies
that~$\xi\thchar{\delta'}{\delta''}(\sigma(x))=\pm\overline{\xi\thchar{\delta'}{\delta''}(x)}$, with the global sign independent of~$x$.
Together with the
equality displayed above, this implies the statement.
\end{proof}

We now study the restriction of the prime form to~$\wt{A}_0\times\wt{A}_0$,
where~$\wt{A}_0$ denotes the universal covering of~$A_0$ given by an arbitrary connected component
of the preimage of~$A_0$ in~$\wt{\Sigma}$.
More precisely, we compute the phase of the~$(-1/2,-1/2)$-form~$E(x,y)$ for~$x,y\in\wt{A}_0$,
evaluated at the tangent vectors~$(v_x,v_y)$, where~$v$ denotes the velocity vector field of
(any regular parametrization of) the oriented curve~$\wt{A}_0$.

For lifts~$x,y\in\wt{A}_0$ of different elements of~$A_0$, we write~$\lfloor y-x\rfloor\in\ZZ$
for the unique integer such that the inequalities
\[
x+\lfloor y-x\rfloor\II<y<x+(\lfloor y-x\rfloor+1)\II
\]
hold in~$\wt{A}_0$, where~$x+\II$ denotes the image of~$x$ in~$\wt{A}_0$ via the generator of the
infinite cyclic covering group. (This slight abuse of notation is motivated by the third point
of Lemma~\ref{lemma:real}.)

\begin{lem}
\label{lem:prime}
There exists~$\varphi\in\{\pm 1,\pm i\}$ such that for any lifts~$x,y\in\wt{A}_0$ of different elements of~$A_0$, the phase of~$E(x,y)$ evaluated at the velocity vectors~$(v_x,v_y)$ is equal
to~$(-1)^{\lfloor y-x\rfloor}\,\varphi$.
\end{lem}
\begin{proof}
Recall that in our setting, the period matrix~$\Omega$ is purely imaginary.
Given an odd theta characteristic~$\binom{\delta'}{\delta''}$, this easily implies that the
number~$\frac{\partial}{\partial z_j}\theta\thchar{\new{\delta'}}{\new{\delta''}}(0)$ is real.
Furthermore, given any 
real element~$P\in\Sigma$,
the first two points of Lemma~\ref{lem:imaginary} imply
\[
\overline{\omega_j(v)}=\sigma^*\omega_j(v)=\omega_j(\sigma_*(v))=\omega_j(v)
\]
for all~$1\le j\le g$, where~$v$ denotes the velocity vector field of the real component containing~$P$.
As a consequence, the holomorphic
form~$\xi\thchar{\delta'}{\delta''}^2=\sum_j\frac{\partial}{\partial z_j}\theta\thchar{\delta'}{\delta''}(0)\,\omega_j$ takes real values on~$A_0$, when evaluated along the corresponding velocity vector field.
By the first point of Lemma~\ref{lem:ident_2},~$\theta\thchar{\delta'}{\delta''}(y-x)$ is also real,
so~$E(x,y)$ evaluated at~$(v_x,v_y)$ is either real or purely imaginary.
In other words, the corresponding phase~$f(x,y)$ is~$\pm 1$ or~$\pm i$.

We now use the crucial fact that the prime form vanishes only if both variables
are lifts of the same element of~$A_0$. This implies in particular that~$f(x,y)$ is constant for all~$x<y<x+\II$, say equal to~$\varphi(x)$ for some map~$\varphi\colon\wt{A}_0\to\{\pm 1,\pm i\}$. This map
being continuous, it is constant. Hence, for~$x,y\in\wt{A}_0$ with~$x<y<x+\II$, the fact that the
prime form is skew-symmetric now implies
\[
f(x,y)=\varphi(x)=\varphi(y)=f(y,x+\II)=-f(x+\II,y)\,.
\]
The equality~$f(x,y)=(-1)^{\lfloor y-x\rfloor}\,\varphi$ easily follows.
\end{proof}

\subsubsection{Fay's identity}
\label{subsub:Fay}

We will make use of the following three versions of Fay's identity~\cite{Fay}.
They are easy consequences of
the standard version formulated by Mumford~\cite[p.~3.214]{ThetaTata2}, that we now recall
without proof.

\begin{thm}[Fay's identity]
For any~$z\in\CC^g$ and $a,b,c,d\in\wt\Sigma$, we have
  \begin{multline}
  \label{eq:FayMumford}
    \theta(z+c-a)\theta(z+d-b)E(c,b)E(a,d)+
    \theta(z+c-b)\theta(z+d-a)E(c,a)E(d,b)\\=
    \theta(z+c+d-a-b)\theta(z)E(c,d)E(a,b)\,.
  \end{multline}
  \end{thm}

The first variation is the~$n=3$ case of~\cite[Lemma~1]{Fock}, and can be obtained as follows.
Divide Equation~\eqref{eq:FayMumford} by~$\theta(z)\theta(z+d-a)\theta(z+d-b)E(c,a)E(c,b)E(c,d)$
and set~$a=\alpha$,~$b=\beta$,~$c=u$,~$d=\gamma$ in~$\wt{\Sigma}$ and~$z=s-\gamma\in\wt{\Pic^0(\Sigma)}=\CC^g$.
This yields the equality
  \begin{multline}
    \label{eq:Fay}
    \frac{%
      \theta(s+u-\alpha-\beta)E(\alpha,\beta)
      }{%
      E(\alpha,u)E(\beta,u)\theta(s-\alpha)\theta(s-\beta)
    }
    +
    \frac{%
      \theta(s+u-\beta-\gamma)E(\beta,\gamma)
      }{%
      E(\beta,u)E(\gamma,u)\theta(s-\beta)\theta(s-\gamma)
    }
    \\ +
    \frac{%
      \theta(s+u-\gamma-\alpha)E(\gamma,\alpha)
      }{%
      E(\alpha,u)E(\gamma,u)\theta(s-\alpha)\theta(s-\new{\gamma})
    }
    = 0,
  \end{multline}
for all~$u,\alpha,\beta,\gamma$ in~$\wt\Sigma$, and all~$s$ in the universal cover~$\wt{\Pic^1(\Sigma)}$
of~$\Pic^1(\Sigma)$.

To obtain the second version, simply pass the second and third terms in Equation~\eqref{eq:Fay}
on the right-hand side. This yields the equation
\begin{equation}
  \label{eq:Fay_bis}
    \frac{%
      \theta(s+u-\alpha-\beta)E(\alpha,\beta)
      }{%
      E(\alpha,u)E(\beta,u)\theta(s-\alpha)\theta(s-\beta)
    }
    = F^{s,\gamma}(u;\beta)-F^{s,\gamma}(u;\alpha)\,,
\end{equation}
where
\begin{equation*}
  F^{s,\gamma}(u;\alpha)=\frac{\theta(s+u-\alpha-\gamma)E(\gamma,\alpha)}{E(\alpha,u)E(\gamma,u)\theta(s-\alpha)\theta(s-\gamma)}.
\end{equation*}
Note that $\gamma$ does not appear in the left-hand side of
Equation~\eqref{eq:Fay_bis} and can be chosen arbitrarily to define $F$.
By carefully letting $\gamma$ tend to $u$ in the definition of
$F^{s,\gamma}(u;\alpha)$, one obtains the following version of
Equation~\eqref{eq:Fay_bis}:
\begin{multline}
  \label{eq:Fay_ter}
\theta(u-s)
    \frac{%
      \theta(s+u-\alpha-\beta)E(\alpha,\beta)
      }{%
      E(\alpha,u)E(\beta,u)\theta(s-\alpha)\theta(s-\beta)
    }\\
    =\omega_{\beta-\alpha}(u)+
    \sum_{j=1}^g\left(%
      \frac{\partial\log\theta}{\partial z_j}(s-\alpha)-
      \frac{\partial\log\theta}{\partial z_j}(s-\beta)
    \right)\omega_j(u)\,,
\end{multline}
where~$\omega_{\beta-\alpha}(u)=\ud_u \log \frac{E(u,\beta)}{E(u,\alpha)}$ is
the unique meromorphic 1-form with 0 integral along $A$-cycles, and two \new{simple}
poles: at $\beta$ with residue 1, and at $\alpha$ with residue $-1$.
See~\cite[Proposition~2.10]{Fay} for a derivation of this variant.

The third version, which can be found at the very end of~\cite{Fock}, is simply obtained by
setting~$F_t(a,b):=\theta(a+b-t)E(a,b)$ in Equation~\eqref{eq:FayMumford} with~$z=a+b-t$,
yielding
\begin{equation}
\label{eq:FayFock}
F_t(a,b)F_t(c,d)+F_t(a,d)F_t(b,c)+F_t(a,c)F_t(d,b)=0
\end{equation}
for all~$a,b,c,d\in\wt\Sigma$ and~$t\in\wt{\Pic^2(\Sigma)}$.

\begin{rem}
\label{rem:Faychar}
It should be noted that in Equation~\eqref{eq:FayMumford}, the theta function~$\theta$ can be replaced with
any theta function~$\theta\thchar{\delta'}{\delta''}$ with theta characteristic. This is a
consequence of the first point of Lemma~\ref{lem:ident_1}.
The same holds true for all the versions of Fay's identity displayed above.
\end{rem}

\section{Fock's Kasteleyn operators and their inverses}
\label{sec:Fock}

In this section, we define our dimer models on an arbitrary minimal graph~$\Gs$, and initiate their study.
More precisely, we start in Section~\ref{sub:dimer} by briefly recalling the necessary combinatorial concepts,
namely those of train-tracks and minimal graphs, together with the definition of the dimer model.
Section~\ref{sub:mapd} also deals with background material, \emph{i.e.}, the definition of the \emph{discrete Abel map} of~\cite{Fock},
and of the parameter spaces~$X_\Gs$ of~\cite{BCdTimmersions}.
In Section~\ref{sub:Kast}, we finally give the definition of the models via the corresponding adjacency
operators of~\cite{Fock}, but restricting the parameters to \new{e}nsure that the resulting edge-weights are positive.
Section~\ref{sub:kernel} deals with explicit functions in the kernel of these operators inspired by~\cite{Fock}, functions that are
used in Section~\ref{sub:inverse} to define a two-parameter family of inverses of each of these operators.

\subsection{Dimer models on minimal graphs}
\label{sub:dimer}

\begin{figure}
  \centering
  \def\svgwidth{7cm}
  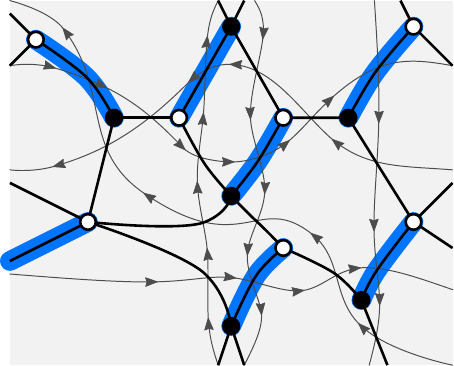
  \hfill
  \def\svgwidth{7cm}
  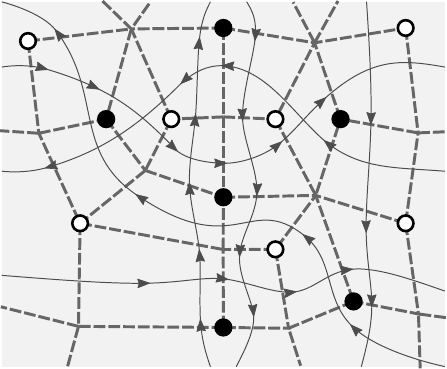
  \caption{Left: piece of a
minimal graph with its associated oriented
    train-tracks (in grey), and
 the corresponding portion of a dimer configuration (in blue).
  Right: the same train-tracks together with the associated graph $\GR$ (in
 dashed lines).
 }
  \label{fig:graph_traintrack_dimers}
\end{figure}

In an attempt to be reasonably self-contained, we now recall the fundamental concepts of train-tracks and minimal graphs, as well as the definition of the dimer model
and of the associated Kasteleyn matrices.
However, this article being a sequel to~\cite{BCdTelliptic}, we favor brevity over rigor and completeness, referring the reader to Sections~2.1 and~2.2 of~\cite{BCdTelliptic} for details.

Let~$\Gs=(\Vs,\Es)$ be a locally finite graph embedded in the plane with faces being bounded topological discs;
\new{in particular, the graph~$\Gs$ is infinite}.
If~$\Gs^*=(\Vs^*,\Es^*)$ stands for the dual embedded graph, then the associated \emph{quad-graph}~$\GR$ is defined from the vertex set~$\Vs\sqcup\Vs^*$ by
joining a primal vertex~$\vs\in\Vs$ and a dual vertex~$\fs\in\Vs^*$ with an edge each time~$\vs$ lies on the boundary of the face corresponding to~$\fs$.
Note that~$\GR$ embeds in the plane with (possibly degenerate) quadrilaterals faces (see Figure~\ref{fig:around_rhombus}).
Following ~\cite{Kenyon:crit,K-S}, we define a \emph{train-track} of~$\Gs$ as a maximal
chain of adjacent quadrilaterals of~$\GR$ such that when one enters a quadrilateral, one exits through the opposite edge,
see Figure~\ref{fig:graph_traintrack_dimers}.

Let us now assume that~$\Gs$ is \emph{bipartite}, \emph{i.e.}, that~$\Vs$ admits a partition~$\Bs\sqcup\Ws$ into black and white vertices such that no edge of~$\Es$ connects two vertices of the same color.
In this case, train-tracks can be consistently oriented, say, with black vertices on the right and white vertices on the left of the path, see again Figure~\ref{fig:around_rhombus}.
We let~$\T$ denote the set of consistently oriented train-tracks of the bipartite graph~$\Gs$.
A bipartite, planar graph~$\Gs$ is said to be \emph{minimal}~\cite{Thurston,GK} if its train-tracks do not self-intersect, and no pair of oriented train-tracks intersect twice in the same direction.
\new{This implies that train-tracks do not form loops, and that~$\Gs$ has neither multiple edges,
nor degree 1 vertices. In particular, a minimal graph is a simple graph.}

We now quickly recall basic facts on dimer models, referring to~\cite{KenyonIntro} for details.

A \emph{dimer configuration} of a graph~$\Gs$ is a collection~$\Ms\subset\Es$ such that every vertex is incident to exactly one edge of~$\Ms$,
see Figure~\ref{fig:graph_traintrack_dimers}.
If~$\Gs$ is finite and endowed with a positive edge-weight function~$\nu=(\nu_\es)_{\es\in\Es}$,
then the \emph{dimer Boltzmann measure}~$\PP$ on the set of dimer configurations of~$\Gs$ is defined by
\[
\PP(\Ms)=\frac{\prod_{\es\in\Ms}\nu_\es}{Z(\Gs,\nu)},
\]
where~$Z(\Gs,\nu)=\sum_{\Ms}\prod_{\es\in\Ms}\nu_\es$ is the \emph{dimer partition function}.
When the graph~$\Gs$ is infinite and planar, this notion is replaced by that of \emph{Gibbs measure}, see e.g.~\cite{KOS}.

Two dimer models on~$\Gs$ defined via edge-weights~$\nu$ and~$\nu'$
are called \emph{gauge equivalent} if there is a positive vertex-function~$\sigma$ such that~$\nu'_{\xs\ys}=\sigma_\xs\,\nu_{\xs\ys}\,\sigma_\ys$ holds for each~$\es=\xs\ys\in\Es$.
If~$\Gs$ is finite, then~$\nu$ and~$\nu'$ yield the same Boltzmann measure.
When $\Gs$ is planar and bipartite, two dimer models on~$\Gs$ are gauge equivalent if and only if
the corresponding edge-weights define equal \emph{face weights},
 where faces weights are the alternating product of edge-weights around each given bounded face.

One of the most fundamental tools for studying the dimer model is the \emph{Kasteleyn matrix}, named after~\cite{Kasteleyn1}, see also~\cite{TF}, and extended by Kuperberg~\cite{Kuperberg} as follows.
Let us fix a finite, planar and bipartite graph~$\Gs$. Consider a weighted adjacency matrix~$\Ks$ of~$\Gs$ twisted by a phase, \emph{i.e.}, a matrix~$\Ks$ with~$\Ks_{\ws,\bs}=\omega_{\ws\bs}\nu_{\ws\bs}$
and~$\omega_{\ws\bs}$ any modulus~$1$ complex number.
Let us assume that for any bounded face~$\fs$ of degree~$2m$ of~$\Gs$, the phase~$\omega$ satisfies the following \emph{Kasteleyn condition}:
\[
\prod_{j=1}^m \frac{\omega_{\ws_j\bs_j}}{\omega_{\ws_j\bs_{j+1}}}=(-1)^{m+1}\,,
\]
assuming the notation of Figure~\ref{fig:tt_around_face}.
Then, the dimer partition function and Boltzmann measure can be computed from~$\Ks$ and its inverse $\Ks^{-1}$, see~\cite{Kenyon1}.

\subsection{The discrete Abel map and the parameter
space~\texorpdfstring{$X_\Gs$}{XG}}
\label{sub:mapd}

Let~$\Sigma$ be an~M-curve. Recall that it admits an oriented real component denoted
by~$A_0$, which contains the base point~$x_0$.
We assign to each oriented train-track~$T\in\T$ of~$\Gs$ an element~$\alpha_T$ of~$A_0$,
referred to as its \emph{angle}.
(This terminology originates from the elliptic case, where~$A_0$ is naturally identified with~$\RR/\ZZ$.)

Following Fock~\cite{Fock}, we define a function~$\mapd$
from the set of vertices of~$\GR$ into~$\Pic(\Sigma)$, as follows.
Choose a face~$\fs_0$ and set~$\mapd(\fs_0)=0$.
Then, along an edge of~$\GR$ crossing a train-track~$T$ with
angle~$\alpha_T$, we formally add~$\alpha_T$ to the value of~$\mapd$ if we
arrive at a black vertex or leave a white vertex (see Figure~\ref{fig:around_rhombus}).
In this way, the degree of~$\mapd(\xs)\in\Div(\Sigma)$ is
equal to
\begin{equation*}
  \deg\mapd(\xs) =
  \begin{cases}
    1 & \text{if $\xs$ is a black vertex of $\Gs$,} \\
    0 & \text{if $\xs$ is a face of $\Gs$,} \\
    -1 & \text{if $\xs$ is a white vertex of $\Gs$.}
  \end{cases}
\end{equation*}
In particular, for any face~$\fs$ of~$\Gs$, the element~$\mapd(\fs)$ belongs
to~$\Pic^0(\Sigma)$.
By Lemma~\ref{lemma:real}, its image by the Abel-Jacobi map belongs to~$(\RR/\ZZ)^g\subset\Jac(\Sigma)$.

As it turns out, only a special class of angle assignments~$T\mapsto\alpha_T$ gives rise to probabilistic models.
It can be described as follows, see~\cite{BCdTimmersions} for more detail.

Let us call two non-closed oriented planar curves~\emph{parallel} (resp.~\emph{antiparallel}) if they intersect infinitely many times in
the same direction (resp. in opposite directions), or if they are disjoint and cross a topological disc in the same direction (resp. in opposite directions).
Consider a triple of oriented train-tracks of~$\Gs$, pairwise non-parallel.
Let~$B$ be a compact disk outside of which these train-tracks do not meet, apart from possible anti-parallel ones,
and order this triple of elements of~$\Tbip$ cyclically according to the outgoing points of the corresponding oriented curves in the circle~$\partial B$. This gives a well-defined partial cyclic order on~$\Tbip$, see~\cite[Section 2.3]{BCdTimmersions}.
Note that~$A_0$ is an oriented topological circle, and therefore endowed with a total cyclic order as well,
which allows for the following definition.

We define~$X_\Gs$ as the set of maps~$\mapalpha\colon\Tbip\to A_0$ \new{that are monotone, in the sense that they preserve the cyclic order, and} such that non-parallel train-tracks have distinct images.
One of the main results of~\cite{BCdTimmersions} is that if~$\Gs$ is minimal, then~$X_\Gs$ is included in the space of \emph{minimal immersions}
of~$\Gs$, and coincides with it if~$\Gs$ is minimal and periodic, see in particular \cite[Theorem 23, and Corollary 29]{BCdTimmersions}.

\subsection{Fock's Kasteleyn operators}
\label{sub:Kast}

To define a version of Fock's adjacency operator satisfying Kasteleyn's condition, let us fix a minimal graph~$\Gs$,
an M-curve~$\Sigma$ and an angle map~$\mapalpha\in X_\Gs$.

We now fix an arbitrary lift~$\widetilde\mapalpha\colon\Tbip\to\wt{A}_0$
of~$\mapalpha$, \emph{i.e.}, ~lifts~$\wt{\alpha}_T\in\wt{A}_0$ of the angles~$\alpha_T\in A_0$,
where~$\wt{A}_0\subset\wt{\Sigma}$ denotes the universal cover of~$A_0$.
Recall from Lemma~\ref{lemma:real} that the Abel-Jacobi map defines an embedding of~$A_0$
in~$(\RR/\ZZ)^g$, and therefore an embedding of~$\wt{A}_0$
in~$\RR^g\subset\wt{\Jac(\Sigma)}=\CC^g$.
We define a lift~$\wt{\mapd}\colon\Vs(\GR)\to\Div(\wt{\Sigma})$ of the discrete Abel map~$\mapd$
by setting~$\wt{\mapd}(\fs_0)=0$,
 and computing the values at every vertex iteratively by adding and subtracting the
  lifts $\wt{\alpha}_T$ of the crossed train-tracks, with the same local rule as
  $\mapd$.
  In particular, if $\bs$ (resp. $\ws$) and $\fs$ are separated by a train-track with angle
  $\alpha$ (resp.\@~$\beta$), one has $\wt\mapd(\bs)=\wt\mapd(\fs)+\wt\alpha$
  (resp.\@~$\wt\mapd(\ws)=\wt\mapd(\fs)-\wt\beta$).
Note that for any face~$\fs$ of~$\Gs$, the divisor~$\wt\mapd(\fs)$ has degree~$0$,
and its image by the Abel-Jacobi map~$\Div^0(\wt\Sigma)\to\CC^g$ belongs to~$\RR^g$.

\begin{defi}
\label{def:K}
  \emph{Fock's adjacency operator} $\Ks$ is the complex weighted adjacency operator
  of the graph $\Gs$, indexed by elements~$t\in(\RR/\ZZ)^g\subset\Jac(\Sigma)$,
  with non-zero coefficients given as follows:
  for every edge $\ws\bs$ crossed by train-tracks with angles $\alpha,\beta$ in~$A_0$ as in Figure~\ref{fig:around_rhombus}, we have
  \begin{equation}
    \Ks_{\ws,\bs}=\frac{E(\wt\alpha,\wt\beta)}{\theta(\wt{t}+\wt\mapd(\fs))\,
    \theta(\wt{t}+\wt\mapd(\fs'))}\,,
      \label{eq:def_Kast}
  \end{equation}
where~$\wt{t}\in\RR^g$
  is a lift of~$t\in(\RR/\ZZ)^g$.
\end{defi}

Here are several remarks on this definition.

\begin{rem}\leavevmode
\label{rem:Kast}
\begin{enumerate}
\item
Since we are working in the universal cover, the coefficient~$\Ks_{\ws,\bs}$  can be understood as an honest complex-valued
function of the lifted angles~$\widetilde{\alpha},\widetilde{\beta}\in\wt{A}_0$ (recall Section~\ref{subsub:linebundles}).
On the other hand, it does \emph{not} project to a well-defined function of~$\alpha,\beta\in A_0$.
Indeed, while the denominator gives a function on~$A_0\times A_0$, the numerator is only well-defined
on the universal cover~$\wt{A}_0\times\wt{A}_0$: replacing~$\wt\alpha\in \wt{A}_0$
by~$\wt\alpha+\II$, \emph{i.e.}, going once around~$A_0$, leads to
\[
E(\wt\alpha+\II,\wt\beta)=\frac{\xi\thchar{\gamma'}{\gamma''}(\wt\alpha)}{\xi\thchar{\gamma'}{\gamma''}(\wt\alpha+\II)}e^{2i\pi\gamma'\cdot\II}E(\wt\alpha,\wt\beta)=-\lambda E(\wt\alpha,\wt\beta)
\]
for some~$\lambda>0$, by Lemma~\ref{lem:prime}, and similarly for~$\beta$.
\new{(Here, we use the notation~$\binom{\gamma'}{\gamma''}$ for the odd theta characteristic
appearing in the prime form, to distinguish it from the theta characteristic~$\binom{\delta'}{\delta''}$
possibly appearing in the denominator, see points~3--5 below.)}
However, the corresponding face weights are well-defined,
\emph{i.e.}, these factors cancel up to gauge equivalence.

\item
By the second point of Lemma~\ref{lem:ident_1}, the \new{entry}~$\Ks_{\ws,\bs}$ does not depend on the choice of the lift~$\wt{t}\in\RR^g$ of~$t\in(\RR/\ZZ)^g$.

\item
By the third point of Lemma~\ref{lem:ident_1}, we have the
identity~$\theta\thchar{\delta'}{0}(z)=\theta\thchar{\delta'}{\delta''}(z-\delta'')$.
Hence, up to a translation of $\wt{t}$ by an element of $\frac{1}{2}\ZZ^g$, we can choose~$\delta''$
arbitrarily in the theta characteristic of the denominator of~\eqref{eq:def_Kast_intro}. We set~$\delta''=0$ in~\eqref{eq:def_Kast} for definiteness.

\item
As stated in Proposition~\ref{prop:Kasteley_intro}, it is possible to define Kasteleyn operators
indexed by an arbitrary real element of~$\Jac(\Sigma)$, \emph{i.e.}, by any element of the form~$t+\Omega\delta$
with~$t\in(\RR/\ZZ)^g$ and~$\delta\in(\frac{1}{2}\ZZ)^g$, as long as the theta characteristic~$\thchar{\delta'}{\delta''}$ of the denominator satisfies~$\delta'=\delta$.
Indeed, the gauge equivalence class of the resulting operator only depends on~$t\in(\RR/\ZZ)^g$, and we set~$\delta=\delta'=0$ for definiteness.

To check this claim, let us fix~$\delta\in(\frac{1}{2}\ZZ)^g$ and write~$\Ks'$ for  the Kasteleyn operator
defined as in~\eqref{eq:def_Kast}, but with~$\wt{t}$ replaced by~$\wt{t}+\Omega\delta$ and~$\theta$
replaced by~$\theta\thchar{\delta}{0}$.
By the third point of Lemma~\ref{lem:ident_1},
we get
\[
\Ks'_{\ws,\bs}=e^{2i\pi\delta\cdot\Omega\delta}\,
e^{2i\pi(2\delta)\cdot\wt{t}} \,
e^{2i\pi\delta\cdot(\new{\wt\mapd(\bs)+\wt\mapd(\ws)})}\,
\Ks_{\ws,\bs}
\,,
\]
yielding the statement.

This is no surprise, as we know from~\cite{KO:Harnack} that \new{for periodic graphs}, the gauge-equivalence classes are
parametrized by one point on each of the ovals, \emph{i.e.}, by a~$g$-dimensional torus,
see also the third point of Remark~\ref{rem:Harnack}.

\item
Note that another natural choice of theta characteristic is given by~$\delta'=\delta''=\frac{1}{2}\II$, leading to
\[
\Ks_{\ws,\bs}=\frac{E(\wt\alpha,\wt\beta)}{\theta\thchar{\frac{1}{2}\II}{\frac{1}{2}\II}(\wt{t}'+\wt\mapd(\fs))\,
\theta\thchar{\frac{1}{2}\II}{\frac{1}{2}\II}(\wt{t}'+\wt\mapd(\fs'))}
\]
indexed by~$t'\in(\RR/\ZZ)^g+\Omega\frac{1}{2}\II$. This generalizes the genus~$1$ case of~\cite{BCdTelliptic}.

\end{enumerate}
\end{rem}

\begin{figure}[ht]
  \centering
  \def\svgwidth{7cm}
  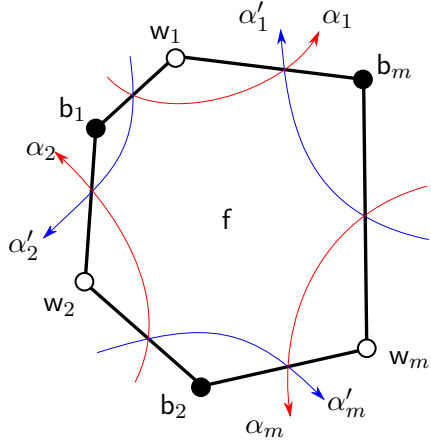
  \caption{Train-tracks around a typical face~$\fs$ of degree $2m$, with
    white vertices~$\ws_1,\dotsc,\ws_m$ \new{and black vertices~$\bs_1,\dots,\bs_m$} on its boundary.}
\label{fig:tt_around_face}
\end{figure}

The following result establishes the fact that we can harness Kasteleyn's theory.
Note that even though it extends the elliptic case treated in~\cite[Proposition~12]{BCdTelliptic},the proof given below is significantly simpler then the one appearing in~\cite{BCdTelliptic}.

\begin{prop}
\label{prop:Kasteleyn}
For any lift~$\wt{\mapalpha}$ of the map~$\mapalpha\in X_\Gs$ and any~$t\in(\RR/\ZZ)^g\subset\Jac(\Sigma)$,
Fock's adjacency operator~$\Ks$ is a Kasteleyn operator.
\end{prop}

\begin{proof}
Let us study the validity of the Kasteleyn condition around an arbitrary face~$\fs$
as in Figure~\ref{fig:tt_around_face}.
By definition, we have
\[
\Ks_{\ws,\bs}=
\frac{E(\wt\alpha,\wt\beta)}
{\theta(\wt{t}+\wt\mapd(\fs))\,\theta(\wt{t}+\wt\mapd(\fs'))}
\,,
\]
with~$\wt{t},\wt\mapd(\fs)$ and~$\wt\mapd(\fs')$ elements of~$\RR^g$.
Hence, by the third point of Lemma~\ref{lem:imaginary} and the second point of Lemma~\ref{lem:ident_2},
the denominator is strictly positive.
Since the prime form is skew-symmetric, the phase of the face weight around~$\fs$ is
therefore equal to the phase of
\[
(-1)^m\frac{E(\wt\alpha_1,\wt\alpha_2')E(\wt\alpha_2,\wt\alpha_3')\dotsb E(\wt\alpha_{m},\wt\alpha_1')}{E(\wt\alpha_1,\wt\alpha_1')E(\wt\alpha_2,\wt\alpha_2')\dotsb E(\wt\alpha_m,\wt\alpha_m')}\,.
\]
Note that this phase is well-defined for~$\alpha_j,\alpha_j'\in A_0$,
as replacing~$\wt\alpha_j$ by~$\wt\alpha_j+\II$ makes a sign appear at the
numerator and denominator, and similarly for~$\alpha'_j$.

By Lemma~\ref{lem:prime}, we know that the phase
of~$E(x,y)$ is equal
to~$(-1)^{\lfloor y-x\rfloor}\,\varphi$ for some fixed~$\varphi\in\{\pm 1,\pm i\}$.
Therefore, we are left with the proof that the integer
\[
\sum_{j=1}^m \lfloor\wt\alpha_{j+1}'-\wt\alpha_j\rfloor + \sum_{j=1}^m \lfloor\wt\alpha_{j}'-\wt\alpha_j\rfloor
\]
is odd, where~$\alpha'_{m+1}$ stands for~$\alpha'_1$.
Since~$\mapalpha$ belongs to~$X_\Gs$ and~$\Gs$ is minimal, Lemma~8 of~\cite{BCdTimmersions} states that~$\mapalpha$
is monotone and injective on the cyclically ordered set of oriented train-track strands adjacent to the vertex~$\ws_j$.
This implies the modulo~$2$ equality~$\lfloor\wt\alpha_{j+1}'-\wt\alpha_j\rfloor + \lfloor\wt\alpha_{j}'-\wt\alpha_j\rfloor=\lfloor\wt\alpha_{j+1}'-\wt\alpha'_j\rfloor$
for all~$j=1,\dotsc,m$;
indeed, observe that this equality does not depend on the lifts, and holds for any
choice of the form~$\wt\alpha_j<\wt\alpha'_j<\wt\alpha'_{j+1}<\wt\alpha_j+\II$.
Furthermore, since~$\mapalpha$ belongs to~$X_\Gs$ and~$\Gs$ is minimal, Lemma~9 of~\cite{BCdTimmersions} states that~$\mapalpha$
is also monotone and non-constant on the cyclically ordered set of oriented train-track strands  turning counterclockwise around~$\fs$
(appearing in blue in Figure~\ref{fig:tt_around_face}).
This implies that the
integer~$\sum_{j=1}^m \lfloor\wt\alpha_{j+1}'-\wt\alpha'_j\rfloor$
is odd; indeed, observe that its parity does not depend on the lifts, and that it is equal to~1 for any
choice of the form~$\wt\alpha'_1<\wt\alpha'_2<\dotsb<\wt\alpha'_m<\wt\alpha'_1+\II$.
This concludes the proof.
\end{proof}

\subsection{Kernel of the Kasteleyn operators}
\label{sub:kernel}

As before, we consider an~M-curve~$\Sigma$ with fixed parameter~$t\in (\RR/\ZZ)^g\subset\Jac(\Sigma)$,
a minimal graph~$\Gs$, and an angle map~$\mapalpha\in X_\Gs$.
The aim of this section is to define a meromorphic form~$g_{\xs,\ys}$ on~$\Sigma$,
with~$\xs,\ys$ arbitrary vertices of the quad-graph~$\GR$,
providing elements in the kernel of the operator~$\Ks$, see also~\cite{Fock}.
These forms play a crucial role in the definition and study of the \emph{divisor of a vertex},
see Proposition~\ref{prop:div(w)}, and of the inverses of the Kasteleyn operators, see Section~\ref{sub:inverse}.

To do so, and as in Section~\ref{sub:Kast},
we also fix a lift~$\widetilde{t}\in\RR^g$ of~$t$ and a lift~$\wt{\mapalpha}\colon\T\to\wt{A}_0$
of~$\mapalpha$,
which induces a lift~$\wt\mapd$ of the discrete Abel map.
We first define a function~$g_{\xs,\ys}$ on~$\wt\Sigma$,
starting with~$\xs,\ys$ being adjacent vertices in~$\GR$.
One of these vertices is a vertex~$\fs$ of~$\Gs^*$, while the other one is
a (white~$\ws$ o\new{r} black~$\bs$) vertex of~$\Gs$.
Depending on these two cases, and following the notation of Figure~\ref{fig:around_rhombus}, we set:
\begin{align*}
  g_{\fs,\ws}(\wt{u}) =
  g_{\ws,\fs}(\wt{u})^{-1} &=
    \frac{\theta(\wt{t}+(\wt{u}+\wt\mapd(\ws)))}{E(\wt\beta,\wt{u})}, \\
  g_{\bs,\fs}(\wt{u}) =
    g_{\fs,\bs}(\wt{u})^{-1} &=
    \frac{\theta(-\wt{t}+(\wt{u}-\wt\mapd(\bs)))}{E(\wt\alpha,\wt{u})}
\end{align*}
for~$\wt{u}\in\wt\Sigma$, noting that the divisors~$\wt{u}+\wt\mapd(\ws)$
and~$\wt{u}-\wt\mapd(\bs)$ both have
degree zero and hence can be considered naturally inside~$\CC^g$ via the Abel-Jacobi map~$\Div^0(\wt\Sigma)\to\CC^g$.
By the second point of Lemma~\ref{lem:ident_1}, the function~$g_{\xs,\ys}$ does not depend on the choice of the lift~$\wt{t}\in\RR^g$ of~$t\in(\RR/\ZZ)^g$.

When $\xs$ and $\ys$ are not necessarily neighbors, consider a path
$\xs=\xs_1,\dotsc,\xs_n=\ys$ in $\GR$ and set:
\begin{align*}
  g_{\xs,\ys}(\wt u) =
  \prod_{j=1}^{n-1} g_{\xs_j,\xs_{j+1}}(\wt u).
\end{align*}

The result is well defined, \emph{i.e.}, does not depend on the choice of the path,
because the product along any closed path is easily seen to be equal to 1.
For example, if~$\bs$ and~$\ws$ are neighbors in~$\Gs$ as in Figure~\ref{fig:around_rhombus},
we get
\begin{equation*}
  g_{\bs,\ws}(\wt u)=g_{\bs,\fs}(\wt u)g_{\fs,\ws}(\wt u)=
  \frac{%
    \theta(\wt t + (\wt u+\wt\mapd(\ws)))
    \theta(-\wt t + (\wt u-\wt\mapd(\bs)))
  }{%
    E(\wt\alpha,\wt u)E(\wt\beta,\wt u)
  }\,.
\end{equation*}

Due to its quasi-periodicity properties, the divisor of~$g_{\xs,\ys}$ is well-defined
on~$\Sigma$ for any vertices~$\xs,\ys$ of~$\GR$.
Furthermore, when~$\xs$ and~$\ys$ are vertices of~$\Gs$, the function~$g_{\xs,\ys}$
defined on~$\wt\Sigma$ projects to a well-defined form on~$\Sigma$,
as follows.

\begin{lem}\label{lem:gmeromfuncform}
  Let $\bs,\bs'$ (resp. $\ws,\ws'$) be two black (resp.\@ white) vertices. Then:
  \begin{itemize}
    \item
  the meromorphic function $g_{\bs,\ws}$
  on $\wt\Sigma$ projects to a meromorphic 1-form on $\Sigma$ (\emph{i.e.}, a section of
  the canonical bundle);
\item the meromorphic functions $g_{\bs,\bs'}$ and $g_{\ws,\ws'}$ on
  $\wt\Sigma$ project to meromorphic functions on $\Sigma$.
\end{itemize}
Moreover, these 1-forms and functions are real, in the sense that they
  satisfy:
  \begin{equation*}
    \sigma^* g_{\bs,\ws} = \overline{g_{\bs,\ws}},\quad
    \sigma^* g_{\bs,\bs'} = \overline{g_{\bs,\bs'}},\quad
    \sigma^* g_{\ws,\ws'} = \overline{g_{\ws,\ws'}}.
  \end{equation*}
\end{lem}

\begin{proof}
To show the first point, fix two vertices~$\bs,\ws$ and
 consider a path $\xs_0,\xs_1,\dotsc,\xs_{2n}$ in $\GR$ from~$\bs$ to
  $\ws$.
  The vertices $\xs_1$ and $\xs_{2n-1}$ represent faces $\fs$ and $\fs'$
  of $\Gs$, which may coincide if $\bs$ and $\ws$ are neighbors.
  Denote by~$\alpha$ (resp.~$\beta$) the angle associated to the train-track separating~$\bs$
from~$\fs$ (resp.\@ $\ws$ from~$\fs'$). For~$1\leq k \leq n-1$, denote
  by~$\alpha_k$ and~$\beta_k$ the angles of the two train-tracks associated to
  the two edges of this path incident with~$\xs_{2k}$, so that~$g_{\bs,\ws}$ can be written as
  \begin{equation}
  \label{eq:g}
    g_{\bs,\ws}(\wt u) =
    \frac{%
      \theta(\wt{t}+\wt\mapd(\fs')+(\wt u-\wt\beta))\,
      \theta(-\wt{t}-\wt\mapd(\fs)+(\wt u-\wt\alpha))
    }{E(\wt{\alpha},\wt u) E(\wt{\beta},\wt u)}
    \prod_{k=1}^{n-1} \frac{E(\wt\alpha_k,\wt u)}{E(\wt\beta_k, \wt u)}\,.
  \end{equation}
  The divisor of $g$ on $\wt\Sigma$ is~$\pi_1(\Sigma)$-periodic,
  and projects to a divisor on~$\Sigma$.
  We now compute this divisor from the product form above,
  making use of Riemann's theorem, see Section~\ref{sub:theta}, to understand the zeros of the first two factors.

  By
Riemann's theorem
applied to~$f_e(u)=\theta(e+\int_{x_0}^u\vec{\omega})$
  with
  $e$ equal to~$t+\mapd(\fs')+(x_0-\beta)$ and
  $-t-\mapd(\fs)+(x_0-\alpha)$ respectively, there exist
  elements~$x_1,\dotsc, x_g$ and~$y_1,\dotsc, y_g$ of~$\Sigma$
  such that for all $1\leq j \leq g$,
\begin{equation*}
  \theta(t+\mapd(\fs') + (x_j-\beta))=
  \theta(-t-\mapd(\fs) + (y_j -\beta))=0\,.
\end{equation*}
  Therefore, the divisor of $g_{\bs,\ws}$, representing its zeros and poles, is well
  defined as an algebraic sum of points in~$\Sigma$ (and not only in~$\wt\Sigma$),
  and is given by
  \begin{equation*}
    D=\sum_{j=1}^g (x_j + y_j) - \alpha - \beta
    +\sum_{k=1}^{n-1}(\alpha_k-\beta_k)\,.
  \end{equation*}
  Now we use Relation~\eqref{eq:RiemannDelta} to get the following equalitites in~$\Jac(\Sigma)$:
  \begin{equation*}
    \sum_{j=1}^g (x_j-x_0)= \Delta - t - \mapd(\fs') +(\beta-x_0),\quad
    \sum_{j=1}^g (y_j-x_0)= \Delta + t + \mapd(\fs) +(\alpha-x_0)\,.
  \end{equation*}

  Moreover, the definition of~$\mapd$ yields the equality~$\mapd(\fs)-\mapd(\fs')=\sum_{k=1}^{n-1} (\beta_k-\alpha_k)$.
 Thus, the divisor~$D$ of~$g_{\bs,\ws}$ satisfies the equality
  \begin{equation*}
  D-2(g-1)x_0=2\Delta
  \end{equation*}
   in~$\Jac(\Sigma)$, which by~\cite[Corollary~3.11, p.~166]{ThetaTata1} implies that~$D$ is linearly equivalent to the canonical divisor.
  By standard arguments, $g_{\bs,\ws}$ is a meromorphic 1-form, and the first point is proved.

We now briefly sketch an alternative proof of this first statement, using the viewpoint of Section~\ref{subsub:linebundles}
and the notation of the proof of Proposition~\ref{prop:prime}.
  When~$u$ is moving along the cycle~$A_j$, we get a
  factor of automorphy~$p_j^2(u)$, and when~$u$ is moving along~$B_j$, the ``extra''
  factors coming from the theta functions and the prime forms cancel exactly,
  and only remains the factor of automorphy~$q_j^2(u)$.
 Hence, the map~$g_{\bs,\ws}$ on~$\widetilde\Sigma$ transforms exactly like meromorphic
  1-forms lifted to the universal cover, and therefore projects onto~$\Sigma$ as a meromorphic~$1$-form.

The second point can be proved in the same way: use the product
form for~$g_{\bs,\bs'}$ and~$g_{\ws,\ws'}$ together with Riemann's theorem
to check that the corresponding divisors are principal.

Finally, the last point is a direct consequence of the fact that the
Riemann matrix~$\Omega$ is purely imaginary (Lemma~\ref{lem:imaginary}),
together with Point~1 of Lemma~\ref{lem:ident_2} and Lemma~\ref{lem:primeform_conj}.
\end{proof}

These forms provide non-zero vectors in the kernel of~$\Ks$, as follows.

\begin{lem}[\cite{Fock}]\label{lem:gkernel}Fix~$\widetilde{u}\in\wt\Sigma$.\leavevmode
  \begin{itemize}
    \item For any vertex $\xs$ of $\GR$, $(g_{\bs,\xs}(\widetilde{u}))_{\bs\in\Bs}$ is in
      the right kernel of $\Ks$. Equivalently, for any white vertex $\ws\in\Ws$,
we have~$\sum_{\bs\sim\ws}\K_{\ws,\bs}\,g_{\bs,\xs}(\widetilde{u})=0$.
    \item For any vertex $\xs$ of $\GR$, $(g_{\xs,\ws}(\widetilde{u}))_{\ws\in\Ws}$ is in
      the left kernel of $\Ks$. Equivalently, for any black vertex $\bs\in\Bs$,
we have~$\sum_{\ws\sim\bs}g_{\xs,\ws}(\widetilde{u})\K_{\ws,\bs}=0$.
  \end{itemize}
\end{lem}

\begin{proof}
  This is a consequence of Fay's trisecant identity in its telescopic
  form~\eqref{eq:Fay_bis}, as noted in the periodic context by~Fock~\cite{Fock},
see also~\cite[Proposition~16]{BCdTelliptic} for details in the elliptic case.
\end{proof}

We now come to the study of poles and zeros of~$g_{\xs,\ys}$.
Recall that for any vertices~$\xs,\ys$ of~$\GR$, these poles and zeros give a well-defined divisor of~$g_{\xs,\ys}$
on~$\Sigma$. In~\cite[Section~3.4]{BCdTelliptic}, this divisor was studied in the elliptic case.
We now adapt and generalize this discussion to the present case.

\begin{lem}
\label{lem:sector}
Suppose that the graph~$\Gs$ is minimal and that the angle map~$\mapalpha$ belongs to~$X_\Gs$. Then,
  for any vertices~$\xs,\ys$ of $\GR$, there exists a partition of~$A_0$ into two intervals, such that one contains
  no poles of~$g_{\xs,\ys}$ and the other no zeros.
\end{lem}
\begin{proof}
  Recall that~$\wt{t}$ belongs to~$\mathbb{R}^g$. Moreover, for~$u\in A_0$, the divisors~$\wt{u}-\wt{\mapd}(\bs)$ and~$\wt{u}+\wt{\mapd}(\ws)$ are mapped by the Abel-Jacobi to~$\mathbb{R}^g \new{\oplus} \Omega\mathbb{Z}^g$ (recall Lemma~\ref{lemma:real}).
  As a consequence, the arguments of the~$\theta$ functions appearing in the factors of~$g_{\xs,\ys}$ belong \new{to}~$\mathbb{R}^g\new{\oplus}\Omega\mathbb{Z}^g$.
  By Point~2 of Lemma~\ref{lem:ident_1} and of Lemma~\ref{lem:ident_2},
  this function does not vanish on~$\mathbb{R}^g\new{\oplus}\Omega\mathbb{Z}^g$.
  Therefore, the two ``theta'' factors appearing in the expression of~$g_{\bs,\ws}$ do not contribute to the zeros or poles of that 1-form on~$A_0$: those come from the remaining factors
  expressed with the prime form in the numerator and denominator, respectively.
The statement now follows from Section~3.4 of~\cite{BCdTelliptic}:
indeed, the proof of Lemma~19 of~\cite{BCdTelliptic} can now be applied verbatim,
as it only relies on the partial cyclic order of train-tracks of the minimal graph~$\Gs$ and the properties of
the angle map~$\mapalpha\in X_\Gs$.
\end{proof}

As a consequence, we can extend the terminology of~\cite[Definition~10]{BCdTelliptic}:
we define the \emph{sector  associated to~$g_{\bs,\ws}$}, denoted by~$s_{\bs,\ws}$,  to be the part of the
partition of~$A_0$ containing the poles of~$g_{\bs,\ws}$. If it has no zeros on~$A_0$ (which happens when~$\bs$
and~$\ws$ are neighbors), then~$s_{\bs,\ws}$ is defined to be the arc from~$\alpha$ to~$\beta$ in the
positive direction of~$A_0$, with the convention of Figure~\ref{fig:around_rhombus}.

\medskip

The~$1$-forms~$g_{\bs,\ws}$ also allow us to define the \emph{divisor of a vertex},
thus extending this notion due to Kenyon-Okounkov~\cite{KO:Harnack} from the periodic
to the general case, see Section~\ref{sub:curve}.

To do so, we fix a white vertex~$\ws$ and assume that the parameter~$t\in(\RR/\ZZ)^g$ is \emph{generic},
in the sense that the function~$\wt{u}\mapsto\theta(\wt{t}+(\wt u+\wt\mapd(\ws)))$ is not identically zero on~$\wt{\Sigma}$.

\begin{defi}
\label{def:divisor}
Let~$\ws$ be a white vertex of~$\Gs$, and let~$t$ be generic. Then, the \emph{divisor of~$\ws$},
denoted by~$\mathrm{div}(\ws)$, is the element of~$\Div(\Sigma)$ given by the divisor of the
function~$\wt{u}\mapsto\theta(\wt{t}+(\wt u+\wt\mapd(\ws)))$ on~$\wt{\Sigma}$.
\end{defi}

This divisor can be given a more concrete description, as follows.
Let us fix a point~$x_0\in A_0\subset\Sigma$, and denote by~$\Delta\in\CC^g$
the associated constant given by Riemann's theorem~\eqref{eq:RiemannDelta}.

\begin{prop}
\label{prop:div(w)}
For any white vertex~$\ws$ of~$\Gs$,~$\mathrm{div}(\ws)$ is a divisor of degree~$g$
whose class in~$\Pic(\Sigma)$ is determined by the following equality in~$\Pic^0(\Sigma)=\Jac(\Sigma)$:
\begin{equation}
\label{eq:divisor}
(\mathrm{div}(\ws)-g x_0) + (\mapd(\ws)+x_0) =  \Delta - t\,.
\end{equation}
Moreover,~$\mathrm{div}(\ws)$ is given by the common zeros of the~$1$-forms~$\{g_{\bs,\ws}\}_{\bs\in\Bs}$,
which consist of one point on each of the real components~$A_1,\dotsc,A_g$ of~$\Sigma$.
Finally, the assignment~$t\mapsto\mathrm{div}(\ws)$ defines a bijection
from~$(\RR/\ZZ)^g\subset\Jac(\Sigma)$ to the product~$A_1\times\dotsb\times A_g$.
\end{prop}
\begin{proof}
Let us fix a white vertex~$\ws$ of~$\Gs$.
The zeros of~$\theta(\wt{t}+(\wt{u}+\wt\mapd(\ws)))$ can be computed using Riemann's theorem~\eqref{eq:RiemannDelta}  as in the proof of Lemma~\ref{lem:gmeromfuncform}, easily leading to Equation~\eqref{eq:divisor}.
Now, recall that~$t$ belongs to~$(\RR/\ZZ)^g$ by definition, and so does~$\mapd(\ws)+x_0$
by the third point of Lemma~\ref{lemma:real}.
On the other hand, we know by Lemma~\ref{lem:Delta} that the constant~$\Delta$, which only depends on~$x_0$,
belongs to~$(\RR/\ZZ)^g+\Omega\frac{1}{2}\mathbf{1}$.
Finally, the fourth point of Lemma~\ref{lemma:real} ensures that the Abel-Jacobi map defines a homeomorphism from
the product~$A_1\times\dotsb\times A_g$ onto that particular real torus,
thus showing that~$\mathrm{div}(\ws)$ consists \new{of} one simple zero on each of the real
components~$A_1,\dotsc,A_g$ of~$\Sigma$.
Moreover, Equation~\eqref{eq:divisor} defines a bijection between~$A_1\times\dotsb\times A_g$ and the elements~$t$ of~$(\RR/\ZZ)^g\subset\Jac(\Sigma)$.

By definition, the form~$g_{\bs,\ws}$ contains the
factor~$\theta(\wt{t}+(\wt u+\wt\mapd(\ws)))$ for any~$\bs\in\Bs$.
Hence, we are left with the proof that the family~$\{g_{\bs,\ws}\}_{\bs\in\Bs}$ does not contain any additional
common zero or pole.
Focusing on the black vertices adjacent to~$\ws$, and following the notation of Figure~\ref{fig:around_rhombus}, we have
\[
g_{\bs,\ws}(\wt{u})=\frac{\theta(\wt{t}+(\wt{u}+\wt\mapd(\ws)))\theta(-\wt{t}+(\wt{u}-\wt\mapd(\ws)-\wt\alpha-\wt\beta))}{E(\wt\alpha,\wt{u})E(\wt{\beta},\wt{u})}\,.
\]
Since~$\Gs$ is minimal and~$\mapalpha$ belongs to~\new{$X_\Gs$}, the train-track angles are cyclically ordered around~$\ws$, \new{see~\cite[Lemma~8]{BCdTimmersions}}.
In particular, the functions~$\{g_{\bs,\ws}\}_{\bs\sim\ws}$ do not have common poles unless~$\ws$ is of degree~$2$.
In this latter case, a similar
argument shows that~$\{g_{\bs,\ws}\}_{\bs\in\Bs_1}$ do not have common poles.
Furthermore, a common zero of the second ``theta factor'' above would contradict Riemann's theorem.
Indeed, imagine that $u$ is a common zero to the second theta factor in
$g_{\bs,\ws}$ and $g_{\bs',\ws}$ for two consecutive black vertices $\bs$ and
$\bs'$ around $\ws$, which we first assume to have degree at least 3.
In other words, $u$ is a common zero of $f_e$ and $f_{e'}$,
with~$e,e'\in\RR^g$ given by
\begin{equation*}
  e=e_0
  -\int_{\wt{x}_0}^{\wt{\alpha}}\vec{\omega}
  -\int_{\wt{x}_0}^{\wt{\beta}}\vec{\omega},
  \quad
  e'=e_0
  -\int_{\wt{x}_0}^{\wt{\beta}}\vec{\omega}
  -\int_{\wt{x}_0}^{\wt{\gamma}}\vec{\omega},
  \quad
  \text{where }
  e_0=-\wt{t}-\wt{\mapd}(\ws)-\wt{x}_0\,.
\end{equation*}
If~$e$ and~$e'$ are not degenerate, then~$f_e$ and~$f_{e'}$ are not identically
zero and, by Lemma~\ref{lem:Delta}, they both have one zero on~$A_j$ for each~$1\leq
j\leq g$. If we denote these zeros by~$x_j$ and~$x'_j$, respectively, then they
satisfy the relation~\eqref{eq:RiemannDelta} in~$\Jac(\Sigma)$, which takes the form
\begin{equation*}
  \Delta
  =e   +\sum_{j=1}^g \int_{x_0}^{x_j}\vec{\omega}
  =e'  +\sum_{j=1}^g\int_{x_0}^{x'_j}\vec{\omega}\,.
\end{equation*}
Let us assume without loss of generality that the common zero of~$f_e$ and~$f_{e'}$ is~$u=x_1=x'_1$.
Then, if we add to the previous equality the vector
\begin{equation*}
  -e_1:=-e_0+\int_{x_0}^\alpha\vec{\omega}+\int_{x_0}^\beta\vec{\omega}+\int_{x_0}^\gamma\vec{\omega}-\int_{x_0}^u\vec{\omega}\,,
\end{equation*}
we get
\begin{equation*}
  -e_1+\Delta=
  \int_{x_0}^\gamma\vec{\omega}+\sum_{j=2}^g\int_{x_0}^{x_j}\vec{\omega} =
  \int_{x_0}^\alpha\vec{\omega}+\sum_{j=2}^g\int_{x_0}^{x'_j}\vec{\omega}\,.
\end{equation*}
Since Equation~\eqref{eq:RiemannDelta} uniquely determines the theta divisor (if~$e_1$ is not degenerate),
this means that the divisor of~$f_{e_1}$ is described by the two~$g$-tuples of
points:~$\gamma,x_2,\dotsc,x_g$
on one hand, and~$\alpha,x'_2,\dotsc,x'_g$ on the other hand.
But this is impossible: indeed,~$\alpha$ is not equal to~$\gamma$ since~$\ws$ has degree at least~$3$,
and~$\alpha\in A_0$ cannot be equal to~$x_j\in A_j$ since they belong to different real components of~$\Sigma$.
If~$\ws$ has degree~$2$, then a similar argument shows that the second ``theta'' factor\new{s} in~$\{g_{\bs,\ws}\}_{\bs\in\Bs_1}$
do not have a common zero.
This concludes the proof.
\end{proof}

\subsection{Inverses of the Kasteleyn operators}
\label{sub:inverse}

Once again, we fix an~M-curve~$\Sigma$ with an element~$t\in (\RR/\ZZ)^g\subset\Jac(\Sigma)$,
a minimal graph~$\Gs$, and an angle map~$\mapalpha\colon\Tbip\to A_0$ in the parameter space~$X_\Gs$.
The aim of this section is to define a two-parameter family of inverses for the associated operator~$\Ks$.

To do so, we need some preliminary definitions.
Recall from Section~\ref{sub:period} that~$\Sigma$ cut along the cycles~$A_0,\dotsc,A_g$ consists
in two compact oriented surfaces with boundary; we denote by~$\Sigma^+$ the one whose oriented boundary
contains~$A_0$ endowed with the fixed orientation, see Figure~\ref{fig:surface}.
Define~$\D$ as the subset of~$\Sigma$ given by
\begin{equation*}
\D=\Sigma^+\setminus\mapalpha(\Tbip)\,.
\end{equation*}

 Given~$\bs,\ws$ and~$u_0$ in the interior of~$\Sigma^+$, we define a
 path~$\Cs_{\bs,\ws}^{u_0}$ in~$\Sigma$
 as an oriented simple path from~$\sigma(u_0)$ to~$u_0$, intersecting~$A_0$ once in
 the complement of the sector~$s_{\bs,\ws}$ (recall Section~\ref{sub:kernel}),
 disjoint from~$A_1\cup\dotsb\cup A_g$
 and such that~$\sigma(\Cs_{\bs,\ws}^{u_0})=-\Cs_{\bs,\ws}^{u_0}$.

When~$u_0$ lies on the boundary of~$\Sigma^+$
(\emph{i.e.}, when~$\sigma(u_0)=u_0$), we define~$\Cs_{\bs,\ws}^{u_0}$ as the natural limit of~$\Cs_{\bs,\ws}^u$ when~$u$
approaches~$u_0$ in~$\Sigma^+$, namely as a closed loop \new{based at~$u_0$} with the following properties:
\begin{itemize}
  \item if~$u_0$ belongs to~$A_0$, then~$\Cs_{\bs,\ws}^{u_0}$ is a homotopically trivial closed contour on~$\Sigma$,
  maybe enclosing some poles of $g_{\bs,\ws}$;
  \item if~$u_0$ belongs to~$A_j$ for some~$1\leq j\leq g$, then~$\Cs_{\bs,\ws}^{u_0}$ is
    homologous to~$B_j$.
\end{itemize}

  Note that these properties do not determine the path~$\Cs_{\bs,\ws}^{u_0}$ uniquely,
  even up to continuous deformation
  in~$\Sigma\setminus\mapalpha(\Tbip)$.
  However, the resulting operator turns out not to depend on this choice.
\medskip

We now define a family of operators~$(\A^{u_0})_{u_0\in\D}$, and prove in
Theorem~\ref{thm:Kinv} that they are indeed inverses of the Kasteleyn operator~$\Ks$.

\begin{defi}
\label{defi:inv}
For every~$u_0\in\D$, we define the linear operator~$\As^{u_0}$ mapping functions
on white vertices (with finite support for definiteness) to functions on black vertices by
its entries: for every pair~$(\bs,\ws)$ of black and white vertices of~$\Gs$, we set
  \begin{equation}
    \As_{\bs,\ws}^{u_0}  = \frac{1}{2i\pi}\int_{\Cs_{\bs,\ws}^{u_0}}g_{\bs,\ws}\,,
      \label{eq:Kinv}
  \end{equation}
\new{where} the path of integration~$\Cs_{\bs,\ws}^{u_0}$ \new{is} as given above,
and the meromorphic~$1$-form~$g_{\bs,\ws}$
as defined in Section~\ref{sub:kernel}.
\end{defi}

We recall from Lemma~\ref{lem:gmeromfuncform} that~$g_{\bs,\ws}$ is a meromorphic 1-form,
with poles in~$s_{\bs,\ws}$ by construction, so the integral is invariant if we deform the path~$\Cs_{\bs,\ws}^{u_0}$
in~$\Sigma\setminus s_{\bs,\ws}$.
\new{Note also that if~$u_0$ and~$u_1$ belong to the same connected component of~$A_0\setminus\mapalpha(\T)$, the operators~$\As^{u_0}$ and~$\As^{u_1}$ coincide.}

Using the same argument of contour deformation as
in~\cite[Lemma~24]{BCdTelliptic}, we obtain the following alternative expression
for the coefficents of~$\A^{u_0}$.

\begin{lem}
\label{lem:H}
 Let~$H^{u_0}$ be a meromorphic function on~$\Sigma\setminus\Cs_{\bs,\ws}^{u_0}$
with a discontinuity jump of~$+1$ when crossing~$\Cs_{\bs,\ws}^{u_0}$ from right to left,
and let~$\gamma_{\bs,\ws}^{u_0}$ be a collection of contours, homologically trivial in~$\Sigma$,
surrounding all the poles of~$g_{\bs,\ws}H^{u_0}$ exactly once counterclockwise.
  Then, we have the equality
  \begin{equation}
    \A^{u_0}_{\bs,\ws}=\frac{1}{2i\pi}\oint_{\gamma_{\bs,\ws}^{u_0}} H^{u_0}\,g_{\bs,\ws}\,.
      \label{eq:Kinv_H}
  \end{equation}
\end{lem}

\begin{rem}\leavevmode
\label{rem:H}
\begin{enumerate}
\item The function~$H^{u_0}$ is well defined up to addition of a meromorphic function on~$\Sigma$.
By a careful choice of that meromorphic function, it might be assumed that~$H^{u_0}$ has no pole on~$A_0$.
It is also possible if needed to ensure that all poles of~$H^{u_0}$ are simple.
\item If~$u_0$ belongs to~$A_0$, then~$\Cs_{\bs,\ws}^{u_0}$ bounds a disk, and~$H^{u_0}$ can simply be chosen to be the indicator
function of this disk.
If~$u_0$ belongs to~$A_j$ for $1\leq j\leq g$, then~$H^{u_0}$ can be understood as a determination, which depends on~$\bs$ and~$\ws$,
of the multivalued function on~$\Sigma$ given by the projection of a meromorphic function on the infinite cyclic cover of~$\Sigma$
determined by the loop~$A_j$.
Finally, if~$u_0\neq\sigma(u_0)$, then~$H^{u_0}$ can be understood as a determination of the multivalued function
on~$\Sigma\setminus{\{u_0,\sigma(u_0)\}}$ given by the projection of a meromorphic function on the infinite cyclic cover
determined by the loops around~$u_0$ and~$\sigma(u_0)$.
In any case, even though the function~$H^{u_0}$ depends on~$\bs$ and~$\ws$, it can be chosen so that its poles
(and residues against a 1-form) do not depend on these vertices, hence their absence in the notation.
\item An explicit form for~$H^{u_0}$ is given in Remark~\ref{rem:Hexplicit} below.
\end{enumerate}
\end{rem}

We are finally ready to state and prove the main result of this section.

\begin{thm}\label{thm:Kinv}
  For any~$u_0\in\D$, the operator~$\A^{u_0}$ is an inverse of the operator~$\Ks$.
\end{thm}

\begin{proof}
 The proof follows the same lines as the one of~\cite[Theorem~26]{BCdTelliptic}, which in turn is inspired from~\cite{Kenyon:crit}.
We need to check that we have~$(\K\A^{u_0})_{\ws,\ws'}=\delta_{\ws,\ws'}$ for every pair of white
vertices~$\ws,\ws'$ and~$(\A^{u_0}\K)_{\bs,\bs'}=\delta_{\bs,\bs'}$ for any pair of black vertices~$\bs,\bs'$.
We write the proof of the first equality in detail; the second can be checked in a similar way.

Let us first assume that~$\ws$ and~$\ws'$ are distinct, and use Definition~\ref{defi:inv}
together with Lemma~\ref{lem:sector}.
\new{By~\cite[Lemma~23]{BCdTelliptic},} the intersection of the complements in~$A_0$ of the sectors~$\{s_{\bs,\ws'}\}_{\bs\sim\ws}$
is non-empty. 
Therefore, the paths~$\{\Cs_{\bs,\ws'}^{u_0}\}_{\bs\sim\ws}$ can be chosen to coincide with
a single path~$\Cs_{\ws,\ws'}^{u_0}$.
This leads to the equality
  \begin{equation*}
    (\Ks\As^{u_0})_{\ws,\ws'}=\sum_{\bs\sim\ws}
    \K_{\ws,\bs}\,\frac{1}{2i\pi}\int_{\Cs_{\ws,\ws'}^{u_0}}g_{\bs,\ws'}
=\frac{1}{2i\pi}\int_{\Cs_{\ws,\ws'}^{u_0}}\left(\sum_{\bs\sim\ws}\K_{\ws,\bs}\,g_{\bs,\ws'}\right)\,,
  \end{equation*}
  which vanishes by Lemma~\ref{lem:gkernel}.

  We now turn to the case where~$\ws$ and~$\ws'$ coincide and use Lemma~\ref{lem:H}, which yields
 \[
    (\Ks\As^{u_0})_{\ws,\ws}=\frac{1}{2i\pi}\sum_{\bs\sim\ws}\oint_{\gamma_{\bs,\ws}^{u_0}} H^{u_0}\,\K_{\ws,\bs}\,g_{\bs,\ws}\,.
 \]
Let us fix a black vertex~$\bs$ adjacent to~$\ws$, and denote by~$\alpha$ and~$\beta$ the angles of the incident
  train-tracks and by~$\fs,\fs'$ the adjacent faces, as in Figure~\ref{fig:around_rhombus}.
  We compute explicitly the corresponding integral
by  residues.
 The residues contributing to
  the integral come in two classes: on the one hand, those coming from poles of~$H^{u_0}$,
  and on the other hand, those from singularities
  of~$\K_{\ws,\bs}g_{\bs,\ws}$.

Recall from Remark~\ref{rem:H} that, even though~$H^{u_0}$ depends on~$\bs,\ws$, its poles do not.
  The residue computation for the contribution of the poles of~$H^{u_0}$
  \begin{equation*}
    \sum_{v\ \text{pole of}\ H^{u_0}} \res_v\left(H^{u_0}\,\Ks_{\ws,\bs}g_{\bs,\ws}\right)
  \end{equation*}
can be carried to the universal cover~$\wt\Sigma$. In particular, if all
the poles of~$H^{u_0}$ are simple, it will result in a linear combination of
evaluations of~$g_{\bs,\ws}$ at lifts~$\wt{v}$ of the poles of~$H^{u_0}$.
When summing over all black vertices adjacent to the white vertex~$\ws$, this
linear combination of functions~$g$ vanishes as~$\bs\mapsto
g_{\bs,\ws}(\wt{v})$ is in the kernel of~$\Ks$ by~Lemma~\ref{lem:gkernel}.

  Let us now turn to the remaining residues at poles of
  $\K_{\ws,\bs}g_{\bs,\ws}$.
From Fay's identity in the form of Equation~\eqref{eq:Fay_ter}, we see that
the meromorphic 1-form $\Ks_{\ws,\bs}g_{\bs,\ws}$ has the following
decomposition
  \begin{equation*}
    \K_{\ws,\bs}\,g_{\bs,\ws} = \omega_{\beta-\alpha}+\sum_{j=1}^g
    c^{\,j}_{\ws,\bs} \omega_j,
  \end{equation*}
  with
  \begin{equation*}
  c^{\,j}_{\ws,\bs}=\frac{\partial\log\theta}{\partial
  z_j}(\wt{t}+\wt\mapd(\fs))-
  \frac{\partial\log\theta}{\partial z_j}(\wt{t}+\wt\mapd(\fs'))\,,
  \end{equation*}
on which we read that it has two simple poles: at~$\beta$ with residue~$+1$, and
at~$\alpha$ with residue~$-1$.
  Therefore,
  since we assumed that~$H^{u_0}$ has no pole on~$A_0$,
  we have:
  \begin{equation*}
    \res_{u=\alpha} H^{u_0}\omega_{\beta-\alpha} = -H^{u_0}(\alpha),\quad
    \res_{u=\beta} H^{u_0}\omega_{\beta-\alpha} = H^{u_0}(\beta).
  \end{equation*}
Since the angle map~$\mapalpha$ belongs to the space~$X_\Gs$ and~$\Gs$ is minimal, Lemma~8 of~\cite{BCdTimmersions}
states that its restriction is monotone and injective on the (cyclically ordered) set of oriented train-track strands adjacent to the vertex~$\ws$.
In other words, the angles $\alpha,\beta,\dotsc$ of the train-track strands~$\ws$ wind once around~$A_0$.
By construction, the corresponding increments
$H^{u_0}(\beta)-H^{u_0}(\alpha),\dotsc$ sum to $+1$, yielding
\begin{equation*}
  (\K \A^{u_0})_{\ws,\ws}=\sum_{\bs\sim\ws}\K_{\ws,\bs}\A^{u_0}_{\bs,\ws}=1\,.
\end{equation*}
This concludes the proof.
\end{proof}

  From the mere existence of at least an inverse of~$\Ks$, we get the following simple but useful lemma.
  \begin{lem}
    \label{lem:maxpp}
    Let~$f$ a function on black vertices of~$\Gs$ with finite support, such that~$\Ks f$ is
    identically zero on white vertices. Then~$f$ is identically zero.
  \end{lem}
  \begin{proof}
    Let~$\As^{u_0}$ be one of the inverses of~$\Ks$ given above. Since all the
     occurring sums have a finite number of nonzero terms, we can
     write~$f=\left(\A^{u_0} \K\right)f = \A^{u_0}\left(\K f\right) =\A^{u_0} 0  = 0$.
  \end{proof}

\begin{rem}
\label{rem:Hexplicit}
  A first step towards the construction of~$H^{u_0}$ is to define a function
  with a jump along a path joining~$u_0$ to~$\sigma(u_0)$. Instead of working
  directly on the surface~$\Sigma$, we pass to its universal cover~$\wt\Sigma$.
  Let~$\wt{u}_0\in\wt\Sigma$ be an arbitrary lift of~$u_0\in\Sigma$, and
  let~$\sigma(\wt{u}_0)\in\wt\Sigma$ denote a lift of~$\sigma(u_0)\in\Sigma$
  such that~$u_0$ and~$\sigma(\wt{u}_0)$ belong to one fundamental domain.

  The expected discontinuity can be obtained by taking the logarithm of an
  expression having a zero at~$\wt{u}_0$ and a pole at~$\sigma(\wt{u}_0)$.
Hence, a natural first candidate is given by
  \[
    H_{\text{pre}}^{u_0}(x)=\frac{1}{2i\pi}\log\frac{E(\wt{u}_0,x)}{E(\sigma(\wt{u}_0),x)}
  \]
  for every~$x\in\wt\Sigma$ whose orbit under the action of the fundamental group does not meet
  a path~$\Cs$ connecting~$\sigma(\wt{u}_0)$ to~$\wt{u}_0$.
  (We can take~$\Cs$ to be a lift to $\wt{\Sigma}$ of $\Cs_{\bs,\ws}^{u_0}$ for a given pair~$(\bs,\ws)$ but any continuous deformation will do.)
  This ensures that we work with a consistent determination of the logarithm.
  This function has the desired behaviour of jumping by +1 when crossing
  the path~$\Cs$.
Moreover,~$H_{\text{pre}}^{u_0}$ is quasi-periodic: if~${x}'\in\wt{\Sigma}$ (resp.~${x}''$) is obtained from~$x$ by the action
of a loop in~$\pi_1(\Sigma)$ corresponding to~$A_j$ (resp.~$B_j$), then we have by Equations~\eqref{eq:primeformA} and~\eqref{eq:primeformB}
  \[
    H_{\text{pre}}^{u_0}({x}')= H_{\text{pre}}^{u_0}(x),\quad
    H_{\text{pre}}^{u_0}({x}'')=H_{\text{pre}}^{u_0}(x)+\int_{\Cs}\omega_j.
  \]

  We can compensate this defect of periodicity along the~$B$-cycles by noting that for~$y\in\wt\Sigma$, we have
  \[
    \frac{\partial\log\theta}{\partial z_k}(x'-y)=
    \frac{\partial\log\theta}{\partial z_k}(x-y),\quad
    \frac{\partial\log\theta}{\partial z_k}(x''-y)=
   \frac{\partial\log\theta}{\partial z_k}(x-y) -2i\pi\delta_{j,k}\,.
  \]
  Hence, we fix~$y\in\wt\Sigma$ and set

  \begin{equation}
    \label{eq:formulaH}
    H^{u_0}(x)=\frac{1}{2i\pi}\log\frac{E(\wt{u}_0,x)}{E(\sigma(\wt{u}_0),x)}
    +\frac{1}{2i\pi}\sum_{j=1}^g\frac{\partial\log\theta}{\partial
    z_j}(x-y)\times\int_{\Cs}\omega_j\,,
  \end{equation}
  which projects to a well-defined function on~$\Sigma$ deprived from the projection of~$\Cs$,
and satisfies the desired conditions.

In the genus~1 case (recall Examples~\ref{ex:period-elliptic} and~\ref{ex:prime}), the choice~$y=\frac{1}{2}\in\CC=\wt{\TT}$ gives
the function~$H^{u_0}$ of~\cite[Section~4.3]{BCdTelliptic}.
\end{rem}

\section{The periodic case}
\label{sec:periodic}

This section deals with the special case where the bipartite planar graph~$\Gs$ is~$\ZZ^2$-periodic.
We start in Section~\ref{sub:prelim} by recalling the properties of train-tracks in this case; we also introduce the space~$X^\mathit{per}_\Gs\subset X_\Gs$ of periodic angle maps,
and the Newton polygon~$N(\Gs)$.
In Section~\ref{sec:Kast-per}, we show that~$\mapalpha\in X^\mathit{per}_\Gs$ induces a periodic operator~$\Ks$ if and only if its image by some natural map~$\varphi\colon X^\mathit{per}_\Gs\to N(\Gs)$
lies in~$\ZZ^2$.
In Section~\ref{sub:curve}, we assume that~$\Ks$ is periodic and use the functions~$g_{\xs,\ys}$ of Section~\ref{sub:kernel} to give an explicit parametrization of the spectral curve for the
corresponding periodic dimer model;
we also identify the divisor of a vertex,
and show that dimer models with Fock's weights can realize any such ``spectral
data''.
We then describe the set of ergodic Gibbs measures of this model in Section~\ref{sub:Gibbs}, and use the map~$\varphi$ to express the
corresponding slopes in Section~\ref{sub:slope}.
Finally, in Section~\ref{sub:free}, we give an explicit local formula for the
free energy and the surface tension of this model.

\subsection{Preliminaries}
\label{sub:prelim}

The aim of this preliminary section is to quickly recall the specificities of oriented train-tracks and angle maps in the periodic case, referring to~\cite[Section~4.1]{BCdTelliptic} for details.

In the whole of this section, we assume that the bipartite planar graph~$\Gs$ is \emph{~$\ZZ^2$-periodic}, \emph{i.e.}, that the group~$\ZZ^2$ acts freely by translation on colored vertices, edges and faces.
We fix a basis of~$\ZZ^2$, allowing to identify a \emph{horizontal} direction (along the vector~$(1,0)$) and a \emph{vertical} one (along the vector~$(0,1)$).
The graph~$\Gs$ has a natural toroidal exhaustion~$(\Gs_n)_{n\geq 1}$, where~$\Gs_n:=\Gs/n\ZZ^2$.
We use similar notation for the toroidal exhaustions of the dual graph~$\Gs^*$, of the quad-graph~$\GR$, and of the train-tracks~$\Tbip$.

Fix a face~$\fs$ of~$\Gs$ and draw two simple dual paths in the plane, denoted by~$\gamma_x$ and~$\gamma_y$, joining~$\fs$ to~$\fs+(1,0)$ and~$\fs+(0,1)$ respectively,
and intersecting only at~$\fs$. They project onto the torus to two simple closed loops in~$\Gs^*_1$, also denoted by~$\gamma_x$ and~$\gamma_y$, winding around the torus and intersecting only at~$\fs$.
Their homology classes~$[\gamma_x]$ and~$[\gamma_y]$ form a basis of the first homology group of the torus~$H_1(\TT,\ZZ)$, and allow for its identification with~$\ZZ^2$.
Every oriented train-track~$T\in\Tbip$ projects to an oriented closed curve on the torus, so the corresponding
homology class~$[T]\in H_1(\TT,\ZZ)$ can be written as~$[T]=h_T[\gamma_x]+v_T[\gamma_y]$
with~$h_T$ and~$v_T$ coprime integers.
This allows to define a partial cyclic order on~$\Tbip$ by using the natural cyclic order
of coprime elements of~$\ZZ^2$ around the origin, an order which coincides with the partial cyclic order on~$\Tbip$ defined in Section~\ref{sub:mapd}.
By construction, this cyclic order induces a cyclic order on~$\Tbip_1=\Tbip/\ZZ^2$.
Note also that two oriented train-tracks~$T,T'\in\Tbip$ are parallel (resp.\@ anti-parallel) as defined in Section~\ref{sub:mapd} if and only if~$[T]=[T']$ (resp.\@~$[T]=-[T']$) in~$H_1(\TT,\ZZ)$.

Recall that~$X_\Gs$ denotes the set of monotone maps~$\mapalpha\colon\Tbip\to A_0$ assigning different images to non-parallel train-tracks.
Following~\cite{BCdTelliptic}, we denote by~$X^{\mathit{per}}_\Gs$ the set of~$\ZZ^2$-periodic elements of~$X_\Gs$:
\[
X^\mathit{per}_\Gs=\{\mapalpha\in X_\Gs\,|\,\alpha_{T+(m,n)}=\alpha_T \text{ for all }T\in\T\text{ and }
(m,n)\in\ZZ^2\}\,.
\]
Since disjoint curves on the torus have either identical or opposite homology classes, this space can be described as
\[
X^\mathit{per}_\Gs=\{\mapalpha\colon\Tbip_1\to A_0\,|\,\mapalpha\text{ is monotone and } \alpha_T\neq\alpha_{T'}\text{ for }[T]\neq[T']\}\,.
\]

By construction, the sum of all oriented closed curves~$T\in\T_1$ bounds a~$2$-chain in the torus, so its homology class vanishes and we have~$\sum_{T\in\T_1}[T]=0$.
As a consequence, the collection of vectors~$([T])_{T\in\T_1}$ in~$\ZZ^2$, ordered cyclically, and drawn so that the initial point of a vector~$[T]$ is the end point of the previous vector,
give\new{s} a convex polygon well-defined up to translations.
This polygon is referred to as the \emph{geometric Newton polygon} of~$\Gs$~\cite{GK} and denoted by~$N(\Gs)$.
The space~$X^\mathit{per}_\Gs$ can now be described combinatorially as the set of order-preserving
maps from oriented boundary edges of~$N(\Gs)$ to~$A_0$ mapping distinct vectors to distinct images.

\subsection{Periodicity of the Kasteleyn operator}
\label{sec:Kast-per}

From now on, we assume that the graph~$\Gs$ is minimal and~$\ZZ^2$-periodic.
We further suppose that~$\Gs$ is \emph{non-degenerate}, in the sense that its geometric Newton polygon~$N(\Gs)$ has positive area.
The aim of this section is to understand for which maps~$\mapalpha\in X^\mathit{per}_\Gs$ the corresponding Kasteleyn operator~$\K$ defined in Equation~\eqref{eq:def_Kast} is periodic.
This criterion in expressed in terms of a natural map~$\varphi\colon X^\mathit{per}_\Gs\to N(\Gs)$ that also proves useful in Section~\ref{sub:slope}.

Note that the periodicity of~$\Gs$ and of~$\mapalpha$ is not
sufficient to ensure the periodicity of the operator~$\K$.
Indeed, this operator makes use of the~$\Pic(\Sigma)$-valued discrete Abel map~$\mapd$ defined in Section~\ref{sub:mapd},
which might have horizontal and vertical \emph{periods}.
More precisely, we have that for every vertex~$\xs$ of~$\GR$ and~$(m,n)\in\ZZ^2$, the equality
\begin{equation}
\label{eq:per-eta}
\mapd(\xs+(m,n))=\mapd(\xs)+\sum_{T\in\T_1}(mv_T-nh_T)\alpha_T
\end{equation}
holds in~$\Pic(\Sigma)$, where recall that~$[T]=(h_T,v_T)\in\ZZ^2$ denotes
the homology class of~$T$.

Consider the map
\[
\varphi\colon X_\Gs^\mathit{per}\longrightarrow\CC^g
\]
defined as follows.
Let us enumerate by~$T_1,\dotsc,T_r$ the elements of~$\T_1$ respecting the
cyclic order, and let~$P_1,\dotsc,P_r\in\CC$ denote the integer points on the boundary of~$N(\Gs)$
numbered so that~$P_{j+1}-P_j=[T_j]$ (where~$P_{r+1}$ stands for~$P_1$).
Given a map~$\mapalpha\in X^\mathit{per}_\Gs$, set
\begin{equation}
\varphi(\mapalpha)=\sum_{j=1}^r P_j \int_{\alpha_{j-1}}^{\alpha_j}\vec{\omega}\in\CC^g\,,
  \label{equ:map_vphi}
\end{equation}
where~$\alpha_j$ stands for~$\alpha_{T_j}$, and the integration path follows the orientation of~$A_0$.

\begin{prop}
\label{prop:angles_perio}
For any~$1\le i\le g$, the image of the coordinate~$\varphi_i\colon X_\Gs^\mathit{per}\to\CC$
of~$\varphi$ is equal to the interior of~$N(\Gs)$.
Moreover, a periodic
map~$\mapalpha\in X_\Gs^\mathit{per}$ induces a periodic Kasteleyn operator~$\K$ if and
only if~$\varphi(\mapalpha)$ lies in~$(\ZZ^2)^g$.
In such a case, the~$g$ integer points in the interior of~$N(\Gs)$ given by~$\varphi(\mapalpha)$ are distinct.
\end{prop}

\begin{proof}
Knowing Lemma~\ref{lem:non-zero}, the beginning of the proof follows quite closely the analogous result in the genus 1 case, see~\cite[Proposition 39]{BCdTelliptic}. Proving that, when $\vphi(\mapalpha)$ lies in $(\ZZ^2)^g$, its $g$ coordinates correspond to distinct integer points of the interior of the Newton polygon $N(\Gs)$ is new (this is not relevant when $g=1$) and non-trivial. In particular this shows that, when $\vphi(\mapalpha)$ lies in $(\ZZ^2)^g$, the interior of $N(\Gs)$ contains at least $g$ integer points.

Let us fix~$\mapalpha\in X_\Gs^\mathit{per}$ and consider its image by~$\varphi_i$ for an
arbitrary~$1\le i\le g$.
First observe that since~$\mapalpha$ belongs to~$X_\Gs$, we have
\[
\sum_{j=1}^r \int_{\alpha_{j-1}}^{\alpha_j}\omega_i=\int_{A_0}\omega_i
=\sum_{k=1}^g\int_{A_k}\omega_i=\sum_{k=1}^g\delta_{k,i}=1\,.
\]
By Lemma~\ref{lem:non-zero}, we also have that~$\int_{\alpha_{j-1}}^{\alpha_j}\omega_i\ge 0$.
Therefore,~$\varphi_i(\mapalpha)$ is a convex combination of the
vertices~$P_1,\dotsc,P_r$, and hence an element of the convex hull~$N(\Gs)$ of these vertices,
so we have the inclusion of the image of~$\varphi_i$ into~$N(\Gs)$.

Now, let us
write~$\overline{X}_\Gs^\mathit{per}$ for
the set of non-constant monotone maps~$\mapalpha\colon\T_1\to A_0$ ($\overline{X}_\Gs^\mathit{per}$ is the set~$X_\Gs^{\mathit{per}}$ without the condition that train-tracks with different homology classes need to have distinct images), and denote by~$\Delta=\{\mapbeta=(\beta_j)_j\in[0,1]^r\,|\,\sum_{j=1}^r\beta_j=1\}$ the standard simplex of dimension~$r-1$.
Observe that~$\varphi_i$
can be described as the restriction to~$X_\Gs^\mathit{per}$ of the composition
\[
\overline{X}_\Gs^\mathit{per}\stackrel{\delta_i}{\longrightarrow}\Delta
\stackrel{p}{\longrightarrow}N(\Gs)\,,
\]
with~$\delta_i(\mapalpha)=(\int_{\alpha_{j-1}}^{\alpha_j}\omega_i)_j$
and~$p(\mapbeta)=\sum_j\beta_j P_j$.
Since~$p$ is an affine surjective map, any point in the interior of~$N(\Gs)$
is the image under~$p$ of an element of the interior of~$\Delta$, \emph{i.e.},\@ an element~$\mapbeta\in\Delta$ with no
vanishing coordinate. Therefore, we have
\[
\delta_i^{-1}(p^{-1}(\mathrm{int}\,N(\Gs)))\subset \delta_i^{-1}(\mathrm{int}\,\Delta)\subset\{\mapalpha\in \overline{X}_\Gs^\mathit{per}\,|\,\mapalpha\text{ injective}\}\subset X_\Gs^\mathit{per}\,,
\]
thus checking the inclusion of the interior of~$N(\Gs)$ into~$\varphi_i(X_\Gs^\mathit{per})$.

To prove the opposite inclusion, consider an arbitrary element~$x$ of~$N(\Gs)\setminus\mathrm{int}\,N(\Gs)$,
and let us write~$F$ for the biggest face of~$N(\Gs)$ containing~$x$ in its interior: concretely,~$F=x$ if~$x$ is
a vertex of~$N(\Gs)$, and~$F$ is the boundary edge of~$N(\Gs)$ containing~$x$ otherwise.
By definition, we have~$p^{-1}(x)=\{\mapbeta\in\Delta\,|\,\sum_j\beta_jP_j=x\}$.
Fix a reference frame for~$\RR^2$ with origin at~$x$ and first coordinate axis orthogonal to~$F$.
Then, the first coordinate of the equation~$\sum_j\beta_jP_j=x$ leads to~$\beta_j=0$ for all~$j$ such that~$P_j$ does
not belong to~$F$. Since~$N(\Gs)$ has positive area, we have~$\beta_j=0$ for some vertex~$P_j$ of~$N(\Gs)$.
By Lemma~\ref{lem:non-zero}, such an element of~$\Delta$
can only be realized as~$\delta_i(\mapalpha)$ with~$\alpha_j=\alpha_{j-1}$. Since~$P_j$ is a vertex of~$N(\Gs)$,
we have~$[T_j]\neq[T_{j-1}]$, so~$\mapalpha$ does not belong to~$X_\Gs^\mathit{per}$.
This shows the inclusion of~$\varphi_i(X_\Gs^\mathit{per})$ into the interior of~$N(\Gs)$, and thus the equality of these
two sets.

Since~$\mapalpha$ is assumed to be periodic,
the operator~$\K$ itself is periodic if and only if the~$\Pic^0(\Sigma)$-valued
discrete Abel map~$\mapd$ on faces is periodic
by Lemma~\ref{lemma:real}.
By Equation~\eqref{eq:per-eta}, this
holds if and only if
\[
\sum_{T\in\T_1} [T]\alpha_T =
\sum_{T\in\T_1} \left(\begin{smallmatrix} h_T \\ v_T \end{smallmatrix}\right)\alpha_T= \left(\begin{smallmatrix} 0 \\ 0 \end{smallmatrix}\right)\in(\RR/\ZZ)^{2g}\,.
\]
This is equivalent to requiring that
the following element of~$\RR^{2g}$ belongs to~$\ZZ^{2g}$:
\begin{equation}
\label{equ:diff}
\sum_{j=1}^r [T_j]\int_{x_0}^{\alpha_j}\vec{\omega}=
\sum_{j=1}^r (P_{j+1}-P_j)\int_{x_0}^{\alpha_j}\vec{\omega}
=\sum_{j=1}^rn_j P_j - \varphi(\mapalpha)\,,
\end{equation}
with~$n_j=\int_{\alpha_{j-1}}^{\alpha_{j}}\vec{\omega}-(\int_{x_0}^{\alpha_j}\vec{\omega}-\int_{x_0}^{\alpha_{j-1}}\vec{\omega})$, and $x_0$ the reference point of $A_0$ chosen in Section~\ref{sub:Abel-Jacobi}. Since~$n_j$ belongs to~$\ZZ^g$ and~$P_j$ to~$\ZZ^2$ for all~$j$, this is equivalent to requiring that~$\varphi(\mapalpha)$ belongs to~$\ZZ^{2g}$. This concludes the proof of the second statement.

To show the last statement,
let us fix~$k\neq l$ and consider the holomorphic 1-form~$\omega=\omega_l-\omega_k$. Since for all~$i\neq k,l$, the integral along~$A_i$ of~$\omega_k$ and~$\omega_l$ is zero,
  then so is the integral of~$\omega$.
  By the same argument as in the proof of Lemma~\ref{lem:non-zero},
  the form~$\omega$ has at least 2 zeros on each
  such~$A_i$, that is at least~$2(g-2)$ zeros (counted with multiplicity).

  We now turn to the behavior of~$\omega$ on~$A_0$. As this form is real and
  has vanishing integral along~$A_0$, it can be written in a tubular
  neighborhood of~$A_0$ as~$\omega=df$
  with~$f$ a non-constant real-valued function. Therefore,~$\omega$ has at least two (distinct) zeros \new{on~$A_0$}.
  On the other hand, as
  the divisor of~$\omega$ has degree~$2g-2$, it cannot have more.
  Let us call
  these two zeros~$\beta$ and~$\gamma$,
  corresponding respectively to the minimum and maximum of~$f$ along~$A_0$.
  Let~$1\le j_0\le j_1\le r$ be
  the indices such that
  \begin{equation*}
    \alpha_{j_0-1} < \beta \leq \alpha_{j_0},
    \quad
    \alpha_{j_1-1} < \gamma \leq \alpha_{j_1}.
  \end{equation*}
  We suppose for the moment that we are in the generic situation where~$\beta$ and~$\gamma$ are distinct from the~$\alpha_j$'s.

  By means of contradiction, let us now assume that the~$k^\mathit{th}$
  and~$l^\mathit{th}$ coordinates of~$\varphi(\mapalpha)$ coincide, \emph{i.e.}, that we have
  \begin{equation}
    \sum_{j=1}^r P_j \int_{\alpha_{j-1}}^{\alpha_{j}}\omega_k =
    \sum_{j=1}^r P_j \int_{\alpha_{j-1}}^{\alpha_{j}}\omega_l\,.
    \label{eq:lincombpoints_omegakl}
  \end{equation}
  This can be written as
  \begin{equation}
    \sideset{}{'}\sum_{j=0}^r P_j\int_{\alpha_{j-1}}^{\alpha_j}\omega = 0\,,
    \label{eq:lincombpoints}
  \end{equation}
  where the prime in the sum means that we drop indices~$j$ for which~$\alpha_{j-1}=\alpha_j$.
  Since the Newton polygon has positive area, it has at least
  three corners (or extremal points), and since~$\mapalpha$ belongs to~$X_\Gs^\mathit{per}$, the angles associated to the two
  train-tracks with homology given by the two edges attached to such a corner are
  different.
  Therefore, the sum with a prime has at least a number of terms
  equal to the number of corners, which is at least three.

  \begin{figure}
    \centering
    \hfill
    \begin{tikzpicture}[scale=0.7]
      \coordinate (p1) at (0,0);
      \coordinate (pj0) at (-2,2);
      \coordinate (p4) at (-3,3);
      \coordinate (p5) at (-5,2);
      \coordinate (p6) at (-6,0);
      \coordinate (pj1) at (-6,-2);
      \coordinate (p9) at (-3,-4);
      \coordinate (pr) at (-1,-3);
\node[label=right:$P_1$] at (p1) {};
      \node[label=right:$P_{j_0}$] at (pj0) {};
      \node[label=left:$P_{j_1}$] at (pj1) {};
      \node[label=right:$P_r$] at (pr) {};
      \draw[clip] (p1) -- (p4) -- (p5) -- (p6) -- (pj1) -- (p9) -- (pr) -- cycle;
      \draw[style=help lines,color=lightgray] (-6,-5) grid (0,5) ;
      \fill[fill=blue!40, fill opacity=0.4] (pj0) -- (p4) -- (p5) -- (p6) -- (pj1)-- cycle;
      \fill[fill=red!40, fill opacity=0.4] (pj1) -- (p9) -- (pr) -- (p1) -- (pj0)-- cycle;
    \end{tikzpicture}
    \hfill
    \begin{tikzpicture}[scale=2]
      \coordinate (alpha0) at (-15:1);
      \coordinate (alpha1) at (10:1);
      \coordinate (alphaj0m1) at (70:1);
      \coordinate (beta) at (90:1);
      \coordinate (alphaj0) at (105:1);
      \coordinate (alphaj1m1) at (-120:1);
      \coordinate (gamma) at (-105:1);
      \coordinate (alphaj1) at (-85:1);
      \draw (0,0) circle [radius=1];
      \node[label=right:$\alpha_1$] at (alpha1) {};
      \node[label=above:$\alpha_{j_0-1}$] at (alphaj0m1) {};
      \node[label=above:$\alpha_{j_0}$] at (alphaj0) {};
      \node[label=above:$\beta$] at (beta) {};
      \node[label=left:$\alpha_{j_1-1}$] at (alphaj1m1) {};
      \node[label=below:$\alpha_{j_1}$] at (alphaj1) {};
      \node[label=below:$\gamma$] at (gamma) {};
      \node[label=right:{$\alpha_0=\alpha_r$}] at (alpha0) {};
      \fill (alpha1) circle [radius=0.05];
      \fill (alphaj0m1) circle [radius=0.05];
      \fill (alphaj0) circle [radius=0.05];
      \fill (alphaj1m1) circle [radius=0.05];
      \fill (alphaj1) circle [radius=0.05];
      \fill (165:1) circle [radius=0.05];
      \fill (170:1) circle [radius=0.05];
      \fill (200:1) circle [radius=0.05];
      \fill (-45:1) circle [radius=0.05];
      \fill (alpha0) circle [radius=0.05];
      \draw (90:0.95) -- (90:1.05); 
      \draw (-105:0.95) -- (-105:1.05); 
      \draw [line width=4mm, opacity=0.4,color=blue!40]
        (beta) arc [start angle=90,end angle=255, radius=1] (gamma);
      \draw [line width=4mm, opacity=0.4,color=red!40]
        (gamma) arc [start angle=-105,end angle=90, radius=1] (beta);
    \end{tikzpicture}
    \hfill\mbox{}
    \caption{Left: the Newton polygon~$N(\Gs)$ and its subdivision induced by the
      convex hulls of~$\{P_{j_0},P_{j_0+1},\dotsc,P_{j_1}\}$ (in blue) and of~$\{P_{j_1}, P_{j_1+1},\dotsc,P_{j_0}\}$ (in red).
      Right: a schematic
      representation of~$A_0$ and the angles associated to train-tracks with
      homology classes given by the boundary of~$N(\Gs)$. Along the blue (resp.\@
      red) arc from~$\beta$ to~$\gamma$ (resp.\@ from~$\gamma$ to~$\beta$),~$\omega$ is positive (resp.\@ negative).}
      \label{fig:sub}
  \end{figure}
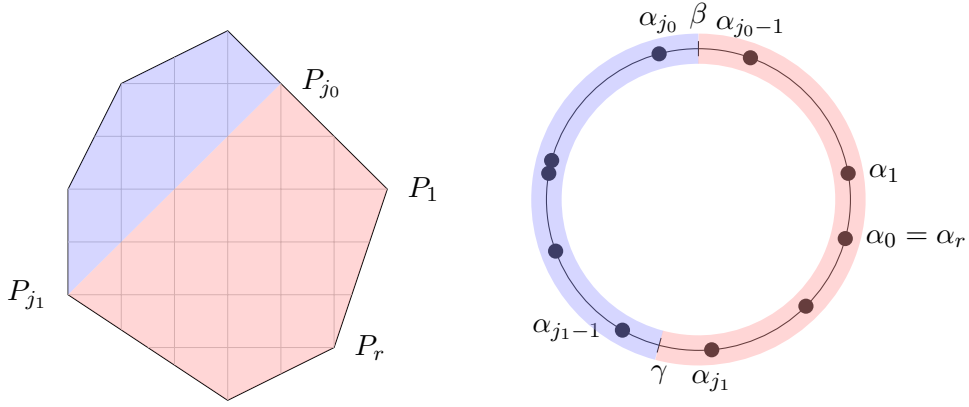

  When~$\alpha_{j-1}\neq\alpha_j$, then
  \begin{equation}
    \int_{\alpha_{j-1}}^{\alpha_j} \omega \text{\ is\ }
    \begin{cases}
      \text{strictly positive}& \text{if $j_0 < j < j_1$ in cyclic order},\\
      \text{strictly negative} & \text{if $j_1 < j < j_0$}.
    \end{cases}
    \label{eq:signomega}
  \end{equation}
  We also split the integral
  from $\alpha_{j_0-1}$ to $\alpha_{j_0}$
  (resp.\@ from $\alpha_{j_1-1}$ to $\alpha_{j_1}$)
  into a negative part from $\alpha_{j_0-1}$ to $\beta$
  (resp.\@ from $\gamma$ to $\alpha_{j_1}$),
  and a positive part from $\beta$ to $\alpha_{j_0}$
  (resp.\@ from $\alpha_{j_1-1}$ to $\gamma$).

  Passing terms with negative coefficients on the right-hand side of the equal
  sign in~\eqref{eq:lincombpoints}, one gets:
  \begin{multline}
    P_{j_0} \int_{\beta}^{\alpha_{j_0}}\omega +
      \sideset{}{'}\sum_{j_0<j<j_1}P_j\int_{\alpha_{j-1}}^{\alpha_j}\omega
      +P_{j_1}\int_{\alpha_{j_1-1}}^\gamma\omega \\
      =
    P_{j_1} \left(-\int_{\gamma}^{\alpha_{j_1}}\omega\right) +
    \sideset{}{'}\sum_{j_1<j<j_0}P_j\left(\int_{\alpha_{j-1}}^{\alpha_j}\omega\right)
    +P_{j_0}\left(\int_{\alpha_{j_0-1}}^\beta\omega\right)\,.
    \label{eq:lincombpoints_contrad}
  \end{multline}
  Note that at least one of the two sums is not empty, because of the number of
  corners being at least 3.
  The sums of the scalar coefficients on both sides are non-zero and equal, as the total
  integral of $\omega$ along $A_0$ is zero. If we divide both sides by this sum, this equation can
  be interpreted as defining a point~$Q$ of the Newton polygon, written once
  as a convex combination of~$P_{j_0}, P_{j_0+1},\dotsc, P_{j_1}$, and once as
  a convex combination of~$P_{j_1},P_{j_1+1},\dotsc,P_{j_0}$. As these two
  collections of vertices of~$N(\Gs)$ are
  vertices of two subpolygons forming a subdivision of $N(\Gs)$, it means that
  this point $Q$ is along the segment $[P_{j_0},P_{j_1}]$ shared by the two
  subpolygons, \new{see Figure~\ref{fig:sub}}.

  This segment being part of the boundary of both subpolygons,
  all coefficients in front of corners different from
  $P_{j_0}$ and $P_{j_1}$ are equal to 0. Since there is at least one such
  corner, there exists a~$j\neq j_0,j_1$ such that
  $\int_{\alpha_{j-1}}^{\alpha_j}\omega=0$, contradicting~\eqref{eq:signomega}.
  Therefore Equation~\eqref{eq:lincombpoints_omegakl} cannot hold, and the
  conclusion of the lemma follows.

  The non-generic case where either~$\beta$ or~$\gamma$ is equal to one of
  the~$\alpha_j$'s can be treated similarly: the point $P_{j_0}$ or $P_{j_1}$ appears only on one side of
  Equation~\eqref{eq:lincombpoints_contrad}. (The discussion changes a little bit
  depending on whether these points are corners or not.)
\end{proof}

\subsection{The spectral data}
\label{sub:curve}

The aim of this section is to understand the spectral data associated to periodic dimer models with Fock's weights.  After recalling the necessary prerequisites, we give an explicit parametrization of the spectral curve,
following and extending the discussion of~\cite[Section~5.4]{BCdTelliptic}, see Proposition~\ref{prop:param_curve}.
Then, we identify the divisor of a vertex as defined by Kenyon-Okounkov~\cite{KO:Harnack}
with its counterpart from Definition~\ref{def:divisor}, see Proposition~\ref{prop:divisor}.
Finally, in Theorem~\ref{thm:Harnack}, we
show that any Harnack curve endowed with a standard divisor can be explicitly realized as the spectral data of a dimer model of this class.

We start by recalling what is meant by the \emph{spectral data} of a dimer model~\cite{KO:Harnack}.
For this part of the discussion, let us suppose that $\Gs$ is any planar, periodic, bipartite weighted graph (not necessarily minimal) and that $\Ks$ is the corresponding Kasteleyn operator.
Following~\cite{KOS}, we define the finite matrix~$\K(z,w)$ for any~$(z,w)\in(\mathbb{C}^*)^2$ as the action in a natural basis of~$\Ks$
on~$(z,w)$-quasiperiodic functions on $\Gs$, \emph{i.e.}, functions~$f$ satisfying
\begin{equation*}
  f(\xs+(m,n))= z^m w^n f(\xs)
\end{equation*}
for any (black or white) vertex~$\xs$ and any~$(m,n)\in\mathbb{Z}^2$.
The \emph{characteristic polynomial}~$P(z,w)$ is the determinant of~$\K(z,w)$.
The \emph{Newton polygon} of~$P$, denoted by~$N(P)$, is the convex hull of lattice points~$(i,j)\in\ZZ^2$ such that $z^iw^j$ appears as a monomial in~$P$.
It actually coincides (up to translations) with its geometric counterpart~$N(\Gs)$ defined in Section~\ref{sub:prelim}, see~\cite[Theorem~3.12]{GK}.

The \emph{spectral curve}~$\Cscr$ is the zero locus of the
characteristic polynomial:
\begin{equation*}
  \Cscr = \{(z,w)\in(\mathbb{C}^*)^2\ :\ P(z,w)=0\}.
\end{equation*}

Following the convention of~\cite{BCdTelliptic}, we define the \emph{amoeba}~$\Ascr$ of the curve~$\Cscr$ as the image
of~$\Cscr$ through the map~$(z,w)\mapsto (-\log|w|,\log|z|)$.
By~\cite{KO:Harnack,KOS},~$\Cscr$ is a \emph{Harnack curve}~\cite{Mikhalkin1}
(also known as a \emph{simple Harnack curve}, see e.g.~\cite{Erwan}),
which is
equivalent to saying that the amoeba map~$\C\to\Ascr$ is at most 2-to-1~\cite{M-R}.
The \emph{real locus} of the curve~$\Cscr$ is the set of points that are invariant under complex conjugation
deprived from its isolated nodes~$\Cscr_\mathrm{sing}$, which are the only singularities a Harnack curve admits~\cite{Mikhalkin1,Erwan}.

Being Harnack is a very strong condition on~$\C$, and on the associated amoeba~$\Ascr$.
For example, if~$\C$ is a genus~$g$ Harnack curve with Newton polygon~$N(P)$, then the boundary
of~$\Ascr$ (which is the image of the real locus of~$\C$ under the amoeba map) consists of~$g$ ovals together with one unbounded component
producing a tentacle for each boundary edge of~$N(P)$. Moreover, each interior point of~$N(P)$ with integer coordinates
corresponds to an oval or to an isolated node, and each tentacle's asymptotic direction coincides with the
vector in~$\ZZ^2$ given by the corresponding boundary edge of~$N(P)$.

\begin{rem}
\label{rem:ovals}
As stated above, for a Harnack curve~$\C$ with given Newton polygon~$N(P)$, each  interior point of~$N(P)$
with integer coordinates corresponds either to an oval of~$\C$ or to an element of~$\Cscr_\mathrm{sing}$.
In our setting, the situation is very explicit: as explained in Proposition~\ref{prop:angles_perio}, any~$\mapalpha\in X_\Gs^\mathit{per}$ inducing
a periodic Kasteleyn operator defines~$g$ distinct integer points in the interior of~$N(\Gs)$ via~$\varphi(\mapalpha)\in(\ZZ^2)^g$.
These are precisely the points that give rise to the ovals, see Corollary~\ref{cor:slope-gas} and the proof of Proposition~\ref{prop:param_curve} below.
\end{rem}

The spectral curve is only the first part of the \emph{spectral data} introduced by Kenyon and Okounkov in~\cite{KO:Harnack}, see also~\cite{GK}: the second consists \new{of} a divisor on~$\C$, whose definition we now recall.
To do so, let us fix an element~$\ws$ in the set~$\Ws_1$ of white vertices of~$\Gs_1=\Gs/\ZZ^2$.
The kernel and cokernel of
\[
\Ks(z,w)\colon\CC^{\Bs_1}\longrightarrow\CC^{\Ws_1}
\]
are~$1$-dimensional for every smooth~$(z,w)\in\C$, see e.g. the proof of~\cite[Theorem~2.2]{Cook}, thus defining line bundles over~$\C\setminus\C_\mathrm{sing}$.
The class in~$\mathrm{Coker}\,\Ks(z,w)$ of the indicator function~$\mathbf{1}_\ws\in\CC^{\Ws_1}$
defines a section of the cokernel line bundle, whose divisor~$(\ws)\in\mathrm{Div}(\C)$ is called the \emph{divisor of the vertex~$\ws$}.
In concrete terms, the equations of~$(\ws)$ are the~$(\ws,\bs)$-cofactors of~$\Ks(z,w)$ for~$\bs\in\Bs_1$.  In the case of periodic graphs, this yields a second definition of the divisor of a vertex, see Definition~\ref{def:divisor}. We prove that the two coincide in Proposition~\ref{prop:divisor}.
Remarkably, this divisor is a so-called \emph{standard divisor}, \emph{i.e.}, it consists \new{of} the sum of one point on each of the~$g$ ovals of~$\C$, see~\cite[Theorem~1]{KO:Harnack}.

\medskip

We now come back to our setting: the graph~$\Gs$ is minimal and periodic,~$\Sigma$ is an M-curve,~$t$ is a real element of~$\Jac(\Sigma)$, and the angle map~$\mapalpha\in X_\Gs^\mathit{per}$ is such that the corresponding Kasteleyn operator~$\Ks$ is periodic. Recall that by Proposition~\ref{prop:angles_perio}, this implies in particular that the Newton polygon~$N(\Gs)$ has at least $g$ interior integer points.

We progress towards an explicit parametrization of the associated spectral curve~$\C$.
Let us fix an arbitrary vertex~$\xs_0$ of~$\Gs$.
For any~$\wt{u}\in\widetilde\Sigma$, the function $\xs\mapsto g_{\xs,\xs_0}(\wt{u})$ is $(z(\wt{u}),w(\wt{u}))$-quasiperiodic with
\begin{equation*}
  z(\wt{u}) = g_{\xs_0+(1,0),\xs_0}(\wt{u}),\quad w(\wt{u}) = g_{\xs_0+(0,1),\xs_0}(\wt{u})\,.
\end{equation*}
These quantities are easily seen not to depend on~$\xs_0$.
By Lemma~\ref{lem:gkernel}, and as already observed in~\cite{Fock}, the pair~$(z(\wt{u}),w(\wt{u}))$ belongs to the
spectral curve~$\C$ for all~$\wt{u}\in\widetilde\Sigma$ not corresponding to a~$u\in\mapalpha(\T)$.
Using the definition of~$g_{\xs,\ys}$ together with Equation~\eqref{eq:per-eta} and
the second point of Lemma~\ref{lem:ident_1}, we get the following explicit
expressions in terms of train-track angles and homology classes:
\begin{equation}
  z(\wt{u})=\prod_{T\in \Tbip_1} E(\wt\alpha_T,\wt{u})^{-v_T},\quad
  w(\wt{u})=\prod_{T\in\Tbip_1}  E(\wt\alpha_T,\wt{u})^{h_T}\,.
  \label{eq:expr_zu_wu}
\end{equation}

By Lemma~\ref{lem:gmeromfuncform}, the maps~$z$ and~$w$ project to meromorphic functions on~$\Sigma$, thus defining a
holomorphic map~$\psi\colon\Sigma\setminus\mapalpha(\Tbip)\to\C$.
This map is not injective in general, as it may send two conjugated elements of~$\Sigma$ to an isolated node in~$\C_\mathrm{sing}$.
However, we have the following result.

\begin{prop}~\label{prop:param_curve}
The map~$\psi\colon\Sigma\to\C$ given by~$\psi(u)=(z(u),w(u))$ is an explicit birational parametrization of the spectral curve~$\C$,
mapping~$A_1,\dotsc,A_g$ to the ovals of~$\C$ and~$A_0$ to the unbounded real component of~$\C$,
implying in particular that~$\C$ has geometric genus~$g$.
More precisely, its restriction
\[
\psi\colon\Sigma\setminus\{\mapalpha(\Tbip)\cup\psi^{-1}(\C_\mathrm{sing})\}\to\C\setminus\C_\mathrm{sing}
\]
is a biholomorphic parametrization of the spectral curve deprived from its singularities.
\end{prop}

\begin{proof}
The map~$\psi$ being meromorphic, it parametrizes an open set of an irreducible component of the spectral curve~$\Cscr$.
Since this curve is Harnack~\cite{KO:Harnack}, it is irreducible. Therefore, the map~$\psi$ is a parametrization of the whole spectral curve.

By Equation~\eqref{eq:expr_zu_wu} and Lemma~\ref{lem:primeform_conj}, we have~$\psi(\sigma(u))=\overline{\psi(u)}$ for all~$u\in\Sigma$.
This implies that~$\psi$ maps the real locus~$A_0\cup A_1\cup\dotsb\cup A_g$ of~$\Sigma$ to the real locus of~$\C$.
Since~$z(u)$ and~$w(u)$ have zeros and poles on~$A_0$, this real component of~$\Sigma$ is mapped to the unbounded real component of~$\Cscr$,
so the remaining real components~$A_1,\dotsc,A_g$ are mapped to the ovals of~$\C$ or to its isolated real nodes (recall Remark~\ref{rem:ovals}).
This latter case is excluded, as it would imply that the holomorphic map~$z$ is constant along some~$A_j$, and hence constant.
By Corollary~\ref{cor:slope-gas}, each of these ovals~$A_j$ gives rise to a different slope, and is therefore mapped to a distinct oval of~$\C$.
Finally, note that since~$\mapalpha$ belongs to~$X^\mathit{per}_\Gs$, the cyclic ordering of~$\mapalpha(\Tbip)\subset A_0$ coincides with
the cyclic ordering of the tentacles of~$\C$.
We are now in the setting of~\cite[Theorem~10]{Erwan}, and can therefore conclude: if~$\psi$ was not birational, then the
curve~$\psi(\Sigma)=\C$ would not be reduced, which is impossible as the isolated nodes are its only possible singularities.
\end{proof}

Before turning to the divisor, we need a preliminary lemma whose proof is postponed to Section~\ref{sub:Gibbs}.

\begin{lem}
\label{lem:Q}
Let~$Q(z,w)$ be the adjugate matrix of~$\Ks(z,w)$.  For every~$\bs$ and~$\ws$ on~$\Gs_1$
and every~$u\in\Sigma\setminus\mapalpha(\T)$, we have
  \[
Q(z(u),w(u))_{\bs,\ws}\lambda(u) = g_{\bs,\ws}(u)\,,
\]
where~$\lambda$ is the  meromorphic 1-form on~$\Sigma$ given
by~$\lambda=\frac{dz}{zw\partial_wP(z,w)}=-\frac{dw}{zw\partial_zP(z,w)}$.
\end{lem}

We now turn to the identification \new{of} the divisor of a vertex of Definition~\ref{def:divisor} with the
homonymous notion of Kenyon-Okounkov.

\begin{prop}
\label{prop:divisor}
For any white vertex~$\ws$, the divisors~$\psi^{-1}((\ws))$ and~$\mathrm{div}(\ws)$ coincide.
\end{prop}
\begin{proof}
Let~$\ws$ be an arbitrary white vertex, let~$(\ws)$ be the associated divisor on~$\C$, and let~$\psi^{-1}((\ws))$
be the corresponding divisor on~$\Sigma$.
By definition, it is given by the common zeros in~$\Sigma\setminus\mapalpha(\T)$
of~$Q(z(u),w(u))_{\bs,\ws}$ for all black vertices~$\bs\in\Bs_1$.
By Lemma~\ref{lem:Q}, this adjugate matrix satisfies
\[
Q(z(u),w(u))_{\bs,\ws}\lambda(u)=g_{\bs,\ws}(u)
\]
for all~$\bs$ and~$u\in\Sigma\setminus\mapalpha(\T)$,
with~$\lambda=\frac{dz}{zw\partial_wP(z,w)}=-\frac{dw}{zw\partial_zP(z,w)}$.
Note that the poles of~$\lambda$ coming from~$\frac{dz}{z}$ or~$\frac{dw}{w}$ lie in~$\mapalpha(\T)$, and therefore do not contribute to the divisor. Therefore, the possible poles of~$\lambda$
outside~$\mapalpha(\T)$ come
from the remaining factors~$\frac{1}{w\partial_wP(z,w)}$ and~$\frac{1}{z\partial_zP(z,w)}$, and hence correspond
to singular points of~$\C$ where the divisor is not defined.
As a consequence, the divisor~$\psi^{-1}((\ws))$ is determined by the common zeros
in~$\Sigma\setminus\mapalpha(\T)$ of~$g_{\bs,\ws}$ for all~$\bs\in\Bs_1$.
The conclusion now follows from Proposition~\ref{prop:div(w)}.
\end{proof}

\begin{rem}
\label{rem:div}
It should be noted that the arguments of Propositions~\ref{prop:div(w)} and~\ref{prop:divisor}
provide an independent proof of the following fact:
given a dimer model on a minimal periodic graph with Fock's weights, the divisor of a vertex is standard.
We hope that this discussion, in particular Equation~\eqref{eq:divisor}, helps clarifying the discrepancy between the viewpoints of
Kenyon-Okounkov~\cite{KO:Harnack} and Goncharov-Kenyon~\cite{GK}, who consider standard divisors of degree~$g$, and of Fock~\cite{Fock},
who deals with holomorphic line bundles of degree~$g-1$.
\end{rem}

In their seminal work~\cite{KO:Harnack}, Kenyon and Okounkov not only prove that the spectral curve of a dimer model
is Harnack and comes equipped with a standard divisor; they also show that every Harnack curve with triangular Newton polygon
endowed with such a divisor can be realized by a dimer model.
This result is extended to arbitrary Newton polygons by Goncharov and Kenyon in~\cite{GK}.
The main aim of Fock in~\cite{Fock}  is to give an explicit form to this inverse map: given a smooth curve~$\Cscr$ (non-necessarily Harnack),
he constructs an explicit ``dimer model'' (not-necessarily with real edge-weights) whose spectral curve is~$\Cscr$.

We are in the position to give a modified version of his result, now restricted to Harnack curves,
following and completing the discussion of~\cite[Section~5.4]{BCdTelliptic}.

\begin{thm}\label{thm:Harnack}
Fix a Harnack curve~$\C$ endowed with a standard divisor~$D$.
Then, there exists an abstract~M-curve~$\Sigma$, a periodic minimal graph~$\Gs$, a map $\mapalpha\in{X}^\mathit{per}_\Gs$
and an element~$t$ of~$(\RR/\ZZ)^g\subset\Jac(\Sigma)$ such that
the associated Fock operator~$\K$ is periodic, and such that the spectral data of
the corresponding dimer model coincides with~$(\Cscr,D)$
(up to a scale change~$(z,w)\mapsto (\lambda z, \mu w)$ with~$\lambda,\mu\in\RR^*$, and fixing a vertex~$\ws$).
Moreover, the assignment~$t\mapsto D$ defines a bijection from~$(\RR/\ZZ)^g$ to the
set of standard divisors on~$\C$.
\end{thm}

\begin{proof}
 The curve~$\Cscr$ being Harnack, it has the maximal number of real components, \emph{i.e.}, it is an~M-curve.
 Hence, there exists an abstract~M-curve~$\Sigma$ and a birational map~$\psi\colon\Sigma\to\Cscr$ such that~$\psi(\sigma(u))=\overline{\psi(u)}$
 for~$u\in\Sigma$.
By definition of a Harnack curve, it has a single unbounded real component, and we
denote by~$A_0$ the corresponding real component of~$\Sigma$.
Let us write~$\psi(u)=(z(u),w(u))$ for the coordinates of the parametrization~$\psi$ of~$\Sigma$.
The maps~$z,w$ being meromorphic functions on~$\Sigma$, their respective divisors are of the form~$-\sum_j v_j \alpha_j$
and~$\sum_j h_j \alpha_j$ for some finite set of elements~$\{\alpha_j\}_j$ of~$\Sigma$
and integers~$\{v_j\}_j,\{h_j\}_j$ so that~$\sum_jv_j=\sum_jh_j=0$.
Since an element~$\alpha_j\in\Sigma$ with~$v_j\neq 0$ or~$h_j\neq 0$ corresponds via~$\psi$ to an element of the complement of~$\C$
in its (toric) closure, and since such elements lie in the closure of the unbounded real component of~$\C$, we have~$\alpha_j\in A_0$ for all~$j$.
By the discussion in Section~\ref{subsub:prime}, there exist constants~$\lambda,\mu\in\CC^*$ such that
 \begin{equation*}
    z(u)=\lambda \prod_j E(u,\alpha_j)^{-v_j},
    \quad
    w(u)=\mu \prod_j E(u,\alpha_j)^{h_j}
  \end{equation*}
for all~$u\in\Sigma$.
The equality~$\psi(\sigma(u))=\overline{\psi(u)}$ together with Lemma~\ref{lem:primeform_conj} imply that~$\lambda$
and~$\mu$ belong to~$\RR^*$, and can therefore be assumed to be~$\pm 1$ via a global scaling.

Allowing for the same element~$\alpha_j$ to appear multiple times, it can be assumed that~$v_j$ and~$h_j$ are coprime for all~$j$.
Then, the equalities~$\sum_jv_j=\sum_jh_j=0$ ensure that there exists a minimal graph~$\Gs_1\subset\TT$
with oriented train-tracks~$\Tbip=\{T_j\}_j$ satisfying~$[T_j]=(h_j,v_j)\in H_1(\TT;\ZZ)$ for all~$j$, see e.g.~\cite[Theorem~2.5]{GK} for the proof.

Let~$\mapalpha\colon\Tbip\to A_0$ be defined by~$\mapalpha(T_j)=\alpha_j$.
By construction, the cyclic order on the boundary edges of~$N(\Gs)$ agrees with the partial cyclic order on~$\Tbip$
given by  homology classes, which in turn coincides with the cyclic order of the tentacles of the amoeba of~$\C$.
By definition of a Harnack curve (recall the discussion before Remark~\ref{rem:ovals}), this cyclic order
coincides with the cyclic order on~$\{\alpha_j\}_j\subset A_0$.
Moreover, two train-tracks with different homology classes correspond to distinct elements in~$A_0$.
In conclusion, the map~$\mapalpha$ belongs to the space~${X}^\mathit{per}_\Gs$, as reinterpreted in Section~\ref{sub:prelim}.

Let~$\K$ be Fock's adjacency operator corresponding to the~M-curve~$\Sigma$, to the universal cover~$\Gs\subset\RR^2$ of~$\Gs_1\subset\TT$,
to~$\mapalpha$, and to any~$t\in(\RR/\ZZ)^g\subset\Jac(\Sigma)$. This operator is Kasteleyn by Proposition~\ref{prop:Kasteleyn},
so it defines a dimer model on~$\Gs$ with positive edge-weights. By construction, the curve~$\C$ is given by the
elements~$(z,w)$ of~$\new{(\mathbb{C}^*)^2}$ such that the kernel of~$\Ks(z,w)$ is non-trivial, so it is the spectral curve of this dimer model.
Also, the divisors~$\sum_j v_j \alpha_j$ and~$\sum_j h_j \alpha_j$ are principal by construction.
By Equation~\eqref{eq:per-eta}, this ensures that the discrete Abel map~$\mapd\colon\{\text{faces of~$\Gs$}\}\to\Pic^0(\Sigma)\simeq\Jac(\Sigma)$
is~$\ZZ^2$-periodic, and so is the operator~$\K$.

The last statement is now a direct consequence of Propositions~\ref{prop:div(w)} and~\ref{prop:divisor}.
\end{proof}

We close this section with several remarks.
\begin{rem}\leavevmode
\label{rem:Harnack}
  \begin{enumerate}
    \item It is natural to wonder to which extent the minimal graph~$\Gs$ is completely determined by the curve~$\C$.
This question is answered by~\cite[Theorem~2.5]{GK}: two minimal graphs defining the same spectral curve~$\C$ are not necessarily the same,
but they define the same associated Newton polygon and therefore, they are related by a sequence of explicit local transformations.
In~\cite[Section~7]{BCdTelliptic}, see also Section~\ref{sub:inv} below, we check that the models defined by Fock's operators are invariant
under these local transformations.
   \item Another natural question is whether every periodic Kasteleyn operator on a minimal graph~$\Gs$
with spectral curve~$\Cscr$ is gauge-equivalent
to Fock's Kasteleyn operators for some~$\mapalpha\in{X}^\mathit{per}_\Gs$ and some~$t\in(\RR/\ZZ)^g$.
The answer is positive, and is a consequence of Theorem~\ref{thm:Harnack} together with~\cite[Theorem~7.3]{GK}.
    \item Two dimer models on the same graph but coming from different  M-curves, or different angle maps,
define different spectral curves, and are therefore not gauge-equivalent.
Moreover, by Theorem~\ref{thm:Harnack} and~\cite[Theorem~7.3]{GK}, we have the following result:
two periodic dimer models on the same minimal graph~$\Gs$ arising from the same M-curve,
the same angle map, and elements~$t,t'\in(\RR/\ZZ)^{g}$
are gauge equivalent if and only if~$t=t'$.

\item Theorem~\ref{thm:Harnack} shows that any spectral data can be realized by a periodic minimal graph.
Does this statement hold for a smaller class of (periodic, bipartite) graphs?
A natural candidate is given by the set of bipartite \emph{isoradial} graphs~\cite{Kenyon:crit},
which can be described as the bipartite graphs whose \new{train-tracks} do not self-intersect and meet at most once~\cite{K-S}.
However, it is unknown whether every convex polygon can be realized as the Newton polygon of such an isoradial graph.

\item \new{As defined by Goncharov-Kenyon~\cite{GK}, the spectral data of a dimer model consists not only of the couple~$(\C,D)$, but also of a {\em parametrization\/}~$\nu$ of the divisor at infinity of the spectral curve.
More concretely,~$\nu$ should be thought of as a total cyclic ordering of the boundary points of~$\C$, see~\cite{KO:Harnack}.
Via the spectral transform, such a parametrization is realized by the set~$\T_1$ of train-tracks of~$\Gs_1$.}

\new{
In our context, recall that~${X}^\mathit{per}_\Gs$ consists of maps~$\mapalpha\colon\T_1\to A_0$ which
preserve the partial cyclic ordering on~$\T_1$.
Therefore, one can very well fix a full spectral data~$(\C,D,\nu)$ in the statement of Theorem~\ref{thm:Harnack}: indeed, the specification of~$\nu$ simply fixes a total cyclic ordering on~$\T_1$ compatible with its intrinsic partial ordering, and hence
restricts the possible angle maps~$\mapalpha\in{X}^\mathit{per}_\Gs$ to those which preserve this total ordering.}
  \end{enumerate}
\end{rem}

\subsection{Ergodic Gibbs measures}
\label{sub:Gibbs}

For any planar, bipartite, periodic weighted graph~$\Gs$ (not necessarily minimal),
the set of ergodic Gibbs measures on dimer configurations was completely characterized in the
work of Kenyon, Okounkov and Sheffield~\cite{KOS}: they form a two-parameter
family~$(\mathbb{P}^{B})$ indexed by~$B=(B_x,B_y)\in\mathbb{R}^2$.
All these measures are determinantal and have an explicit expression in terms of
the periodic Kasteleyn operator~$\K$ and its companion~$\K(z,w)$: for any~$B$,
the probability of occurrence of~$k$ distinct edges~$\es_1=\ws_1\bs_1,\dotsc,\es_k=\ws_k\bs_k$ is given by
\begin{equation}
\label{eq:GibbsKOS}
  \mathbb{P}^B(\es_1,\dotsc,\es_k)=\left(\prod_{j=1}^k \K_{\ws_j,\bs_j}\right)
  \det_{1\leq i,j\leq k} A^B_{\bs_i,\ws_j}\,.
\end{equation}
Here, the operator~$A^B$ has entries given by the following formula: if~$\ws$ and~$\bs$ are in the same fundamental domain
and~$(m,n)$ belongs to~$\mathbb{Z}^2$, then
\[
  A^B_{\bs+(m,n),\ws}=\iint_{\TT_B} \K(z,w)^{-1}_{\bs,\ws}\,z^m w^n
  \frac{dz}{2i\pi z}\frac{dw}{2i\pi w}
  =\iint_{\TT_B} \frac{Q(z,w)_{\bs,\ws}}{P(z,w)}z^m w^n
\frac{dz}{2i\pi z}\frac{dw}{2i\pi w}\,,
\]
with~$Q(z,w)$ the adjugate matrix of~$\K(z,w)$ and~$\TT_B=\{(z,w)\in{(\mathbb{C}^*)}^2\ ;\ |z|=e^{B_y}, |w|=e^{-B_x}\}$.

The phase diagram of this family is described by the amoeba of the characteristic
polynomial~$P(z,w)$, see~\cite[Theorem~4.1]{KOS}:
\begin{itemize}
  \item If~$U$ is a connected component of the complement of the interior of the amoeba,
then all the values of~$B$ inside~$U$ give the same measure:
    \begin{itemize}
      \item if~$U$ is unbounded, the measure is called \emph{solid} or
        \emph{frozen}: \new{every edge has} a deterministic state;
      \item if~$U$ is bounded, the measure is called \emph{gaseous} (or \emph{smooth} in the more general terminology of random surfaces):
        correlation between edges decay exponentially fast.
    \end{itemize}
  \item Any $B$ in the interior of the amoeba gives a different measure
    $\mathbb{P}^B$. These measures are called \emph{liquid} (or \emph{rough} in
    the terminology of random surfaces): the covariance of two edges at distance
    $n$ decays like $n^{-2}$.
\end{itemize}

We now directly relate these operators $A^B$ with our operators $\A^{u_0}$, indexed by the
subset~$\D=\Sigma^+\setminus\mapalpha(\Tbip)$ of~$\Sigma$.

\begin{thm}
  \label{prop:ABAu0}
  For any~$B=(B_x,B_y)$ in the amoeba of~$\C$, let~$u_0$ be the unique element of~$\D$
such that~$\log|z(u_0)| = B_y$ and~$\log|w(u_0)| = -B_x$.
  Then, the operators~$A^{B}$ and~$\A^{u_0}$ coincide.
\end{thm}
This result together with Proposition~\ref{prop:param_curve} and~\cite[Theorem~4.1]{KOS} immediately yield the following
alternative presentation of the Gibbs measures and of the associated phase diagram.

\begin{cor}
\label{cor:Gibbs}
Fix a periodic minimal graph~$\Gs$, an M-curve~$\Sigma$, an element~$\mapalpha\in{X}^\mathit{per}_\Gs$
and a real element~$t\in\Jac(\Sigma)$, and consider the dimer model on~$\Gs$ with corresponding Kasteleyn operator~$\Ks$ \new{which we assume to be periodic}.
Then, the set of ergodic Gibbs measures is given by the measures~$(\PP^{u_0})_{u_0\in\D}$ whose expression on cylinder sets is given
as follows: for any set~$\{\es_1=\ws_1 \bs_1,\dotsc,\es_k=\ws_k \bs_k\}$ of distinct edges of~$\Gs$,
\begin{equation}\label{equ:Gibbs}
\PP^{u_0}(\es_1,\dotsc,\es_k)=
\Bigl(\prod_{j=1}^k
\Ks_{\ws_j,\bs_j}\Bigr)\times
\det_{1\leq i,j\leq k} \Bigl(\A^{u_0}_{\bs_i,\ws_j}\Bigr)\,.
\end{equation}
Furthermore, the model is solid (resp. gaseous, liquid) if~$u_0$ belongs
to~$A_0$ (resp. to~$A_1\cup\dotsc\cup A_g$, to the interior of~$\D$).
\qed
\end{cor}

Before turning to the proof of Theorem~\ref{prop:ABAu0}, let us emphasize the key property of Corollary~\ref{cor:Gibbs}.

\begin{rem}\label{rem:proba_edge}
Although it is far from obvious from their
  expression~\eqref{eq:GibbsKOS},
  the local statistics~\eqref{equ:Gibbs}
for the Gibbs measures $(\PP^{u_0})_{u_0\in\D}$ are \emph{local},
in the sense that~$\PP^{u_0}(\es_1,\dotsc,\es_k)$ only depends on the
graph~$\Gs$ in a neighborhood of~$\es_1\cup\dotsb\cup\es_k$, see however Remark~\ref{rem:weight_Fock_intro}. As an example of application, let us compute the probability of occurrence of a single edge $\es=\ws\bs$ using the notation of Figure~\ref{fig:around_rhombus}.

Using Corollary~\ref{cor:Gibbs} and the definition of $\As^{u_0}$, see Equation~\eqref{eq:Kinv}, we have:
 \[
 \PP^{u_0}(\es)=\frac{1}{2i\pi}\int_{\Cs^{u_0}_{\bs,\ws}} \Ks_{\ws,\bs}g_{\bs,\ws}.
 \]
Now, using Fay's identity in the form of~\eqref{eq:Fay_ter} to compute the product
$\Ks_{\ws,\bs}g_{\bs,\ws}$, we obtain
 \[
 \PP^{u_0}(\es)=\frac{1}{2i\pi}\int_{\Cs^{u_0}_{\bs,\ws}}\omega_{\beta-\alpha}+
 \frac{1}{2i\pi}
 \sum_{j=1}^g \Bigl(\frac{\partial \log\theta}{\partial z_j}(\wt{t}+\wt{\mapd}(\fs))-
 \frac{\partial \log\theta}{\partial z_j}(\wt{t}+\wt{\mapd}(\fs'))\Bigr)\int_{\Cs^{u_0}_{\bs,\ws}} \omega_j,
 \]
where recall that
$\omega_{\beta-\alpha}=\ud_{u}\log\frac{E(u,\beta)}{E(u,\alpha)}$ is the unique
meromorphic 1-form with $0$ integral along $A$-cycles, and two simple poles: at $\beta$ with residue 1, and $\alpha$ with residue $-1$.
Using the definition of the contours in the different phases, see Section~\ref{sub:inverse}, we can push the computation further.

$\bullet$ Gaseous phases. Then $\Cs^{u_0}_{\bs,\ws}$ is homologous to $B_k$
  for some $1\leq k\leq g$. Using Riemann's bilinear
  relation~\eqref{eq:RiemannBilin} with the differential of the third kind
  $\omega_{\beta-\alpha}$, and the fact that
  $\int_{B_k}\omega_j=\Omega_{j,k}$,
  we obtain
\[
\PP^{u_0}(\es)=\int_{\alpha}^{\beta}\omega_k + \frac{1}{2i\pi}
\sum_{j=1}^g \Omega_{j,k}
\Bigl(\frac{\partial \log\theta}{\partial z_j}(\wt{t}+\wt{\mapd}(\fs))-
\frac{\partial \log\theta}{\partial z_j}(\wt{t}+\wt{\mapd}(\fs'))\Bigr),
\]
where the path of integration from $\alpha$ to $\beta$ lies in the surface $\Sigma$ cut along
$\{A_j,B_j:\, 1\leq j\leq g\}$, see Figure~\ref{fig:cutopen}.

  Because in the surgery to cut open $\Sigma$, we take $\Cs_{\bs,\ws}^{u_0}$ as
  a realization of $B_k$, and since by definition, the contour
  $\Cs_{\bs,\ws}^{u_0}$ does not intersect $A_0$ in
  the oriented arc from $\alpha$ to $\beta$, the integration from $\alpha$ to
  $\beta$ of $\omega_k$ is really an integral along $A_0$ in the positive
  direction.

$\bullet$ Solid phases. Then, since we are integrating holomorphic
  1-forms on closed contours bounding disks in $\Sigma$, the integrals $\int_{\Cs^{u_0}_{\bs,\ws}} \omega_j$ are all equal to 0. The first integral is non zero if and only if $u_0$ is such that the cyclic order $[\alpha,u_0,\beta]$ on $A_0$ is preserved. In this case, it is equal to the residue at $\beta$ which is 1 by definition of $\omega_{\beta-\alpha}$, so we find
 \[
\PP^{u_0}(\ws\bs)=\II_{\{\text{the cyclic relation $[\alpha,u_0,\beta]$ holds in $A_0$}\}}.
 \]
$\bullet$ Liquid phase. We use the explicit form of $\omega_{\beta-\alpha}$, and obtain
   \begin{multline*}
\PP^{u_0}(\ws\bs)=\frac{1}{2i\pi}\log \frac{E(\sigma(u_0),\alpha)}{E(u_0,\alpha)}
\frac{E(u_0,\beta)}{E(\sigma(u_0),\beta)}\\
+\frac{1}{2i\pi}
 \sum_{j=1}^g \Bigl(\frac{\partial \log\theta}{\partial z_j}(\wt{t}+\wt{\mapd}(\fs))-
 \frac{\partial \log\theta}{\partial z_j}(\wt{t}+\wt{\mapd}(\fs'))\Bigr) \int_{\Cs_{\ws,\bs}^{u_0}} \omega_j.
\end{multline*}
Note that \new{specializing} to the case $g=1$, we recover the computation of Proposition 43 of~\cite{BCdTelliptic}.
\end{rem}

Theorem~\ref{prop:ABAu0} is a generalization of~\cite[Theorem~34]{BCdTelliptic} which deals with the elliptic case.
We follow the same strategy for the proof: we do not use a uniqueness argument for inverses of~$\Ks$ with some growth property,
but perform a direct computation to partially evaluate the double integral defining~$A^B_{\bs+(m,n),\ws}$ by taking a residue,
then make a change of variable to
transform the remaining integral as an integral on the surface~$\Sigma$.

Before doing that, let us associate to any closed, oriented, dual path~$\gamma$ on~$\Gs_1$
the following function : for~$\wt{u}\in\wt\Sigma$, set
\begin{equation}
  J_\gamma(\wt{u}) = \sum_{\es=\ws\bs}
  (\es\wedge\gamma)\K_{\ws,\bs}g_{\bs,\ws}(\wt{u}),
 \label{eq:defJgamma}
\end{equation}
where~$\es\wedge\gamma\in\ZZ$ denotes the algebraic intersection number of the oriented edge~$\es=\ws\bs$ with
the oriented curve~$\gamma$.
By Lemma~\ref{lem:gmeromfuncform}, the function~$J_\gamma$ on~$\wt{\Sigma}$
projects to a meromorphic 1-form on~$\Sigma$.
We now relate~$J_\gamma$ with~$z$ and~$w$ given by Equation~\eqref{eq:expr_zu_wu}.

\begin{prop}
  \label{prop:J_dlog}
  For any closed oriented dual path~$\gamma$ on~$\Gs_1$, we have the equality
 \[
J_\gamma=-d\log z_\gamma\,,
\]
where~$z_\gamma$ stands for~$z^{h_\gamma}w^{v_\gamma}$ if~$[\gamma]=(h_\gamma,v_\gamma)\in\ZZ^2=H_1(\TT;\ZZ)$.
\end{prop}

\begin{proof}
  These two quantities define meromorphic 1-forms on~$\Sigma$. To
  prove they are equal, it is enough to show that they have the same singular
  parts, and the same periods along the cycles~$A_j$,~$1\leq j\leq g$.
More precisely, we prove that these two differential forms
  \begin{itemize}
    \item have~0 integral along~$A$-cycles;
    \item have no pole outside of $\{\alpha_T\ ;\ T\in\Tbip_1,\ T\wedge\gamma\neq 0\}$;
      \item admit $\alpha_T$ as a simple pole
        with residue $-T\wedge\gamma$ if $T\wedge\gamma$ does not vanish.
    \end{itemize}

  Let us start with $J_\gamma$.
  First, we rewrite Fay's identity in the form of Equation~\eqref{eq:Fay_ter}
in terms of~$\K$ and~$\g$ for two neighboring vertices~$\bs$ and~$\ws$ with
  the convention of Figure~\ref{fig:around_rhombus}.
 It yields the following unique decomposition of the meromorphic 1-form~$\Ks_{\ws,\bs}g_{\bs,\ws}$
as a sum of a holomorphic part (\emph{i.e.}, a linear combination of the $\omega_j$'s) and a
  meromorphic part with zero integral along the $A$-cycles:
  \begin{equation}
    \Ks_{\ws,\bs} g_{\bs,\ws}
    =
    \omega_{\beta-\alpha}+\sum_{j=1}^g\left(%
      \frac{\partial\log\theta}{\partial z_j}
      (\wt{t}+\wt\mapd(\fs))
      -
        \frac{\partial\log\theta}{\partial z_j}
        (\wt{t}+\wt\mapd(\fs'))
    \right)\omega_j\,.
    \label{eq:Faydiff}
  \end{equation}
  Note that the coefficient of~$\omega_j$ is the difference of the same function
  evaluated at~$\fs$ and~$\fs'$. By definition, the 1-form~$J_\gamma$ is the  weighted sum of these
  contributions, with weights given by the algebraic  intersections of~$\gamma$ with edges $\es=\ws\bs$.

  The first point now follows from the fact that for any face~$\fs$, the
  \new{term}~$\frac{\partial\log\theta}{\partial z_j}(\wt{t}+\wt\mapd(\fs))$ appears 
  in the sum defining~$J_\gamma$ \new{with coefficient~$\sum_{\es\subset\partial\fs}(\es\wedge\gamma)=\partial\fs\wedge\gamma=0$ since~$\gamma$ is a closed oriented path and~$\partial\fs$ a
  closed oriented path that bounds; hence,} the total coefficient of~$\omega_j$ in that sum is 0.
To check the second and third points, note that the only poles come from train-tracks
intersecting~$\gamma$. Fix a train-track~$T_\alpha$ of $\T_1$ with angle~$\alpha$. Every
intersection of~$T_\alpha$ with~$\gamma$ corresponds to an edge. The contribution
of this intersection to~$T_\alpha\wedge\gamma$ can be positive (as on
Figure~\ref{fig:intersec_tt_edge_gammax}), and the residue of~$\omega_{\beta-\alpha}$ is~$-1$,
or it can be negative, and the residue is~$+1$. Summing all these contributions,
we see that the total residue of~$J_\gamma$ at~$\alpha$ is~$-T_\alpha\wedge\gamma$.

\begin{figure}[htb]
  \centering
  \def\svgwidth{6cm}
\begingroup%
  \makeatletter%
  \providecommand\color[2][]{%
    \errmessage{(Inkscape) Color is used for the text in Inkscape, but the package 'color.sty' is not loaded}%
    \renewcommand\color[2][]{}%
  }%
  \providecommand\transparent[1]{%
    \errmessage{(Inkscape) Transparency is used (non-zero) for the text in Inkscape, but the package 'transparent.sty' is not loaded}%
    \renewcommand\transparent[1]{}%
  }%
  \providecommand\rotatebox[2]{#2}%
  \newcommand*\fsize{\dimexpr\f@size pt\relax}%
  \newcommand*\lineheight[1]{\fontsize{\fsize}{#1\fsize}\selectfont}%
  \ifx\svgwidth\undefined%
    \setlength{\unitlength}{582.2603951bp}%
    \ifx\svgscale\undefined%
      \relax%
    \else%
      \setlength{\unitlength}{\unitlength * \real{\svgscale}}%
    \fi%
  \else%
    \setlength{\unitlength}{\svgwidth}%
  \fi%
  \global\let\svgwidth\undefined%
  \global\let\svgscale\undefined%
  \makeatother%
  \begin{picture}(1,0.50221329)%
    \lineheight{1}%
    \setlength\tabcolsep{0pt}%
    \put(0,0){\includegraphics[width=\unitlength,page=1]{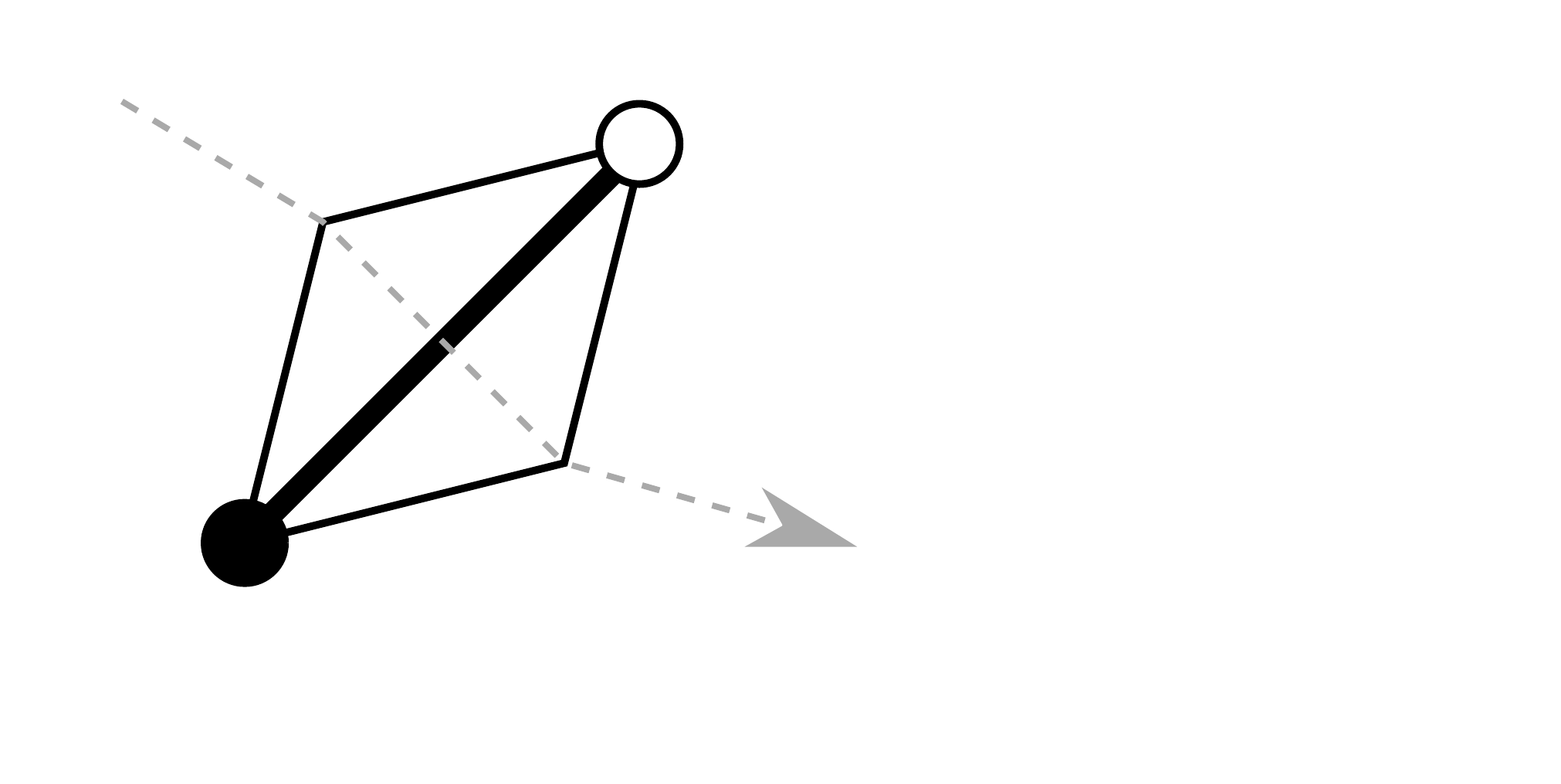}}%
    \put(0.59337397,0.15384039){\color[rgb]{0,0,0}\makebox(0,0)[lt]{\lineheight{1.25}\smash{\begin{tabular}[t]{l}$\gamma_x$\end{tabular}}}}%
    \put(0.4130423,0.46298042){\color[rgb]{0,0,0}\makebox(0,0)[lt]{\lineheight{1.25}\smash{\begin{tabular}[t]{l}$\ws$\end{tabular}}}}%
    \put(0.07814057,0.10231704){\color[rgb]{0,0,0}\makebox(0,0)[lt]{\lineheight{1.25}\smash{\begin{tabular}[t]{l}$\bs$\end{tabular}}}}%
    \put(0,0){\includegraphics[width=\unitlength,page=2]{fig_intersec_tt_edge_gammax.pdf}}%
    \put(0.56761228,0.01215122){\color[rgb]{0,0,0}\makebox(0,0)[lt]{\lineheight{1.25}\smash{\begin{tabular}[t]{l}$T_\alpha$\end{tabular}}}}%
    \put(0.74794399,0.34705289){\color[rgb]{0,0,0}\makebox(0,0)[lt]{\lineheight{1.25}\smash{\begin{tabular}[t]{l}$T_\beta$\end{tabular}}}}%
  \end{picture}%
\endgroup%

  \caption{Intersection of the two train-tracks corresponding to an edge and
  $\gamma_x$ when the white vertex is on the left of $\gamma_x$.}
  \label{fig:intersec_tt_edge_gammax}
\end{figure}

We now turn to~$d\log z_\gamma$.
To check the first point, note that~$z_\gamma$ is real along the~$A$-cycles, and does not vanish as all its zeros and poles lie on~$A_0$.
Therefore, the argument of~$z_\gamma$ is constant along any~$A_j$, so the integral of~$d\log z_\gamma$ is equal to 0.
To prove the second and third points, we use Equation~\eqref{eq:expr_zu_wu} to express~$z_\gamma$ in terms of angles associated to train-tracks.
It follows that the 1-form~$d\log z_\gamma$ has a
  simple pole at~$\alpha_T$ if the train-track~$T$ intersects~$\gamma$,
  and that the associated residue is the degree of that point in the divisor of~$z_\gamma$,
which is given by~$h_Tv_\gamma - v_Th_\gamma = T\wedge\gamma$. This concludes the proof.
\end{proof}

\begin{rem}\label{rem:J_gamma}
  For $\gamma=\gamma_y$, we can automatically keep track of the edges of $\Gs_1$
  intersecting $\gamma_y$ with the correct signs and multiplicities, by using
  $\Ks(z,w)$ as a kind of generating function for those. Because of the definition
  of the variable $w$, selecting these intersections with the correct signs
  boils down to applying the $w\partial_w$ differential operator to $\Ks(z,w)$. Thus we have, see also Equation (27) of~\cite{BCdTelliptic},
  \begin{equation*}
    J_{\gamma_y}(u)=-\sum_{\bs,\ws}w(u)\partial_w
    \Ks(z(u),w(u))_{\ws,\bs}g_{\bs,\ws}(u).
  \end{equation*}
  Likewise, for $\gamma=\gamma_x$, we have
  \begin{equation*}
    J_{\gamma_x}(u)=+\sum_{\bs,\ws}z(u)\partial_z
    \Ks(z(u),w(u))_{\ws,\bs}g_{\bs,\ws}(u).
  \end{equation*}
\end{rem}

\medskip

We are finally ready to give the proofs of Lemma~\ref{lem:Q}
relating~$Q(z,w)_{\bs,\ws}$ and~$g_{\bs,\ws}$, and of Theorem~\ref{prop:ABAu0}
on the equality of the operators~$A^B$ and~$A^{u_0}$.

\begin{proof}[Proof of Lemma~\ref{lem:Q}]
Let~$Q(z,w)$ be the adjugate matrix of~$\Ks(z,w)$.  We need to show the equality
 \begin{equation}
\label{eq:Q}
Q(z(u),w(u))_{\bs,\ws}\lambda(u) = g_{\bs,\ws}(u)
\end{equation}
for every~$\bs$ and~$\ws$ on~$\Gs_1$ and every~$u\in\Sigma\setminus\mapalpha(\T)$,
with~$\lambda$ the  meromorphic 1-form on~$\Sigma$ given
by~$\lambda=\frac{dz}{zw\partial_wP(z,w)}=-\frac{dw}{zw\partial_zP(z,w)}$.
 From the fact that the adjugate matrix~$Q(z,w)$ satisfies
  \begin{equation}
    Q(z,w)\K(z,w)=P(z,w) \operatorname{Id},
    \label{eq:adjugate}
  \end{equation}
  we know that it has rank at most~1 on the spectral curve. It is a product of a
  vector in the right kernel of~$\K(z,w)$ times one in the left kernel of~$\K(z,w)$.
Hence, there exists a non-zero meromorphic~1-form~$\lambda$ on~$\Sigma$ such that
Equation~\eqref{eq:Q} is satisfied
for every~$\bs$ and~$\ws$ on~$\Gs_1$.
  Moreover, differentiating~\eqref{eq:adjugate} with respect to~$w$, evaluating on the spectral curve and taking the trace yields, see also~\cite[Lemma 37]{BCdTelliptic}:
  \begin{align*}
  \lambda(u)\partial_w P(z(u),w(u))&=\sum_{\bs,\ws} \partial_w
  \Ks(z(u),w(u))_{\ws,\bs}g_{\bs,\ws}(u)\\
  &=-\frac{1}{w(u)} J_{\gamma_x}(u)= \frac{dz}{z w}(u),
  \end{align*}
  where in the last line we use Remark~\ref{rem:J_gamma} and Proposition~\ref{prop:J_dlog}.
This implies the first formula for~$\lambda$, while the second can be obtained in the same way by
exchanging the roles of~$z$ and~$w$ and dealing carefully with signs, or
noticing that on the spectral curve, the following equality holds:
\[
  \partial_z P(z(u),w(u)) dz(u)+
  \partial_w P(z(u),w(u)) dw(u)
  =d(P(z(u),w(u))=0.\qedhere
\]
\end{proof}

\begin{proof}[Proof of Theorem~\ref{prop:ABAu0}]
  We just sketch the proof, as it is analogous to that
  of~\cite[Theorem~34]{BCdTelliptic}.
  Consider a  white vertex~$\ws$ of~$\Gs_1$.
  By Lemma~\ref{lem:maxpp} applied to~$f=A^B_{\cdot,\ws}-A^{u_0}_{\cdot,\ws}\in\CC^{\Bs}$,
it is enough to show that for any black vertex~$\bs$ of $\Gs_1$, the equality
  \begin{equation*}
    A^B_{\bs+(m,n),\ws}=A^{u_0}_{\bs+(m,n),\ws}
  \end{equation*}
  holds for all~$(m,n)\in\mathbb{Z}^2$, except possibly for a finite number of values.
  Therefore, we now assume that~$n$ is large enough. The other situations are treated similarly
  by exchanging the roles of~$z$ and~$w$, and/or replacing~$z$ and~$w$ by their
  inverses.

  We now evaluate by the residue theorem one of the two integrals defining
  $A^{B}_{\bs+(m,n)}$: for every $z$ on the circle of radius $e^{B_y}$,
  the only poles of
  \begin{equation*}
    w\mapsto\frac{Q(z,w)_{\bs,\ws}}{P(z,w)} w^{n-1}
  \end{equation*}
  inside the disk $\{|w|<e^{-B_x}\}$ come from zeros of $P(z,\cdot)$,
  and we write
  \[
    \int_{|w|=e^{-B_x}} \frac{Q(z,w)_{\bs,\ws}}{P(z,w)}w^{n-1}\frac{dw}{2i\pi}=
    \sum_{j=1}^{d_{z,B_x}}\frac{Q(z,w_j(z))_{\bs,\ws}}{\partial_w
    P(z,w_j(z))}w_j(z)^{n-1},
    \]
    where $w_1(z),\dotsc,w_{d_{z,B_x}}(z)$ represent the poles of $w\mapsto P(z,w)$
  inside the circle of radius $e^{-B_x}$. When $z$ varies along the circle of
  radius $e^{B_y}$, the points $(z,w_j(z))$ describe paths which can be pulled
  back by $\psi$ as paths on the surface $\Sigma$. This collection of paths
  consists of
  \begin{itemize}
    \item a certain number of closed loops with trivial homology along the
      $A$-cycles
    \item a path connecting $\sigma(u_0)$ to $u_0$ if $B$ is in the interior of
      the amoeba. If $B$ is on the boundary, then this path is also a closed
      loop.
  \end{itemize}
  This family of paths and loops on $\Sigma$ can be deformed without crossing
  any $\{\alpha_T\}_{T\in\T_1}$, to become the path
  $\Cs_{\bs+(m,n),\ws}^{u_0}$.

  Performing the change of variable from~$z$ to~$u$ in the remaining
  integral defining~$A^B_{\bs+(m,n),\ws}$ and using Lemma~\ref{lem:Q}, we have
  \begin{align*}
    A^B_{\bs+(m,n),\ws} &= \int_{|z|=e^{B_y}}
    \sum_{j=1}^{d_{z,B_x}}\frac{Q(z,w_j(z))_{\bs,\ws}}{\partial_w
    P(z,w_j(z))}w_j(z)^{n-1}z^{m-1}\frac{dz}{2i\pi}\\
                        &=\frac{1}{2i\pi}
    \int_{\Cs_{\bs+(m,n),\ws}^{u_0}}\frac{Q(z(u),w(u))_{\bs,\ws}}{z(u)w(u)\partial_w
    P(z(u),w(u))} z^m(u) w^n(u)dz(u)\\
	&=\frac{1}{2i\pi}
\int_{\Cs^{u_0}_{\bs+(m,n),\ws}}g_{\bs,\ws}(u)z(u)^m w(u)^n \\
 &=\frac{1}{2i\pi}
\int_{\Cs^{u_0}_{\bs+(m,n),\ws}}g_{\bs+(m,n),\ws}(u)
= A^{u_0}_{\bs+(m,n),\ws}\,.
  \end{align*}
This concludes the proof.
\end{proof}

\subsection{Slope of the Gibbs measures \texorpdfstring{$\mathbb{P}^{u_0}$}{Pu₀}}
\label{sub:slope}

The discussion of this section follows~\cite[Section~5.6]{BCdTelliptic}.
Therefore, we only develop the new aspects coming from higher genus and often refer the reader to this article for details.

Let us fix~$u_1\in A_0\setminus\mapalpha(\T_1)$ and denote by~$\Ms_1$ the dimer
configuration on which the solid Gibbs measure~$\mathbb{P}^{u_1}$ is
concentrated. Also, let~$P_1$ be the integer point on the boundary of~$N(\Gs)$
corresponding to the interval of~$A_0\setminus\mapalpha(\T_1)$ containing~$u_1$.

\new{For any dimer configuration~$\Ms$}, the \new{corresponding} \emph{height difference} (relative to~$\Ms_1$) between two faces~$\fs$ and~$\fs'$ of $\Gs$
is defined as
\begin{equation}
  h(\fs')-h(\fs)=\sum_{\es=\ws\bs} (\es\wedge\gamma)\left(\II_{\{\es\in
  \Ms\}}-\II_{\{\es\in \Ms_1\}}\right)\,,
  \label{eq:heightchange}
\end{equation}
where~$\gamma$ is an oriented dual path connecting~$\fs$ to~$\fs'$, and~$\es\wedge\gamma$ its algebraic intersection
number with the oriented edge~$\es=\ws\bs$.
This quantity is
well defined and does not depend on the choice of~$\gamma$ because~$\Ms$ and~$\Ms_1$,
viewed as 1-forms on~$\Gs$, have the same divergence at any vertex.

The \emph{slope}~$(s^{u_0},t^{u_0})$ of the Gibbs measure~$\mathbb{P}^{u_0}$ is
the expected horizontal and vertical height change~\cite{KOS}, \emph{i.e.}, the
expectation of Expression~\eqref{eq:heightchange} for~$\fs'$ equal to~$\fs+(1,0)$ and~$\fs+(0,1)$,
or in other words, for~$\gamma=\gamma_x$ and~$\gamma=\gamma_y$, respectively.

Applying Corollary~\ref{cor:Gibbs} in the case~$k=1$ to Expression~\eqref{eq:heightchange} as in~\cite[Theorem~38]{BCdTelliptic}, we get
the equalities
\begin{equation*}
  s^{u_0}=\frac{1}{2i\pi}\int_{\Cs_{u_1}^{u_0}} J_{\gamma_x},\quad
  t^{u_0}=\frac{1}{2i\pi}\int_{\Cs_{u_1}^{u_0}} J_{\gamma_y}\,,
\end{equation*}
where~$\Cs_{u_1}^{u_0}$ is an oriented path in~$\Sigma$ connecting~$\sigma(u_0)$ to~$u_0$,
crossing~$A_0$ once at~$u_1$, disjoint from~$A_1\cup\dotsb\cup A_g$,
such that $\sigma(\Cs_{u_1}^{u_0})=-\Cs_{u_1}^{u_0}$.
By Proposition~\ref{prop:J_dlog}, this can be rewritten as
\begin{equation}
  s^{u_0}=-\frac{1}{2i\pi}\int_{\Cs_{u_1}^{u_0}} \ud \log z,\quad
  t^{u_0}=-\frac{1}{2i\pi}\int_{\Cs_{u_1}^{u_0}} \ud \log w\,.
  \label{eq:slope}
\end{equation}
In other words,~$s^{u_0}$ and~$t^{u_0}$ are (up to a multiplication by $-\pi$)
continuous determinations of the arguments of~$z(u_0)$ and~$w(u_0)$ respectively.
Up to a proper normalization, they correspond to a unique point in the
\emph{coamoeba} of $\C$.
Note that these formulas can be seen as a refinement of~\cite[Theorem~5.6]{KOS},
where the equalities are only valid up to a sign and modulo~$\pi$.
\new{They} are also related to~\cite[Proposition~3.2]{KOS}, since the part of the coamoeba of~$\C$ parametrized by~$\Sigma^+$ is in 1-to-1
correspondence with the Newton polygon
of $P$~\cite{Passare}.

Since the magnetic field is given by the log of the modul\new{us} of $z(u_0)$ and
$w(u_0)$, the pair slope/magnetic field is realizing (half of) the
\emph{amoeba-to-coamoeba mapping} for the Harnack curve $\C$, which has been
described and computed explicitly by
Passare~\cite{Passare}. In terms of dimer models, this translates to the fact
that the slope and the magnetic field are dual variables when performing
Legendre transform between the \emph{free energy}, represented by the Ronkin
function of the characteristic polynomial, and the \emph{surface
tension}~\cite{KOS}.

We now give explicit formulas for the slopes of the solid and gaseous phases.
For solid phases, Corollary~39 from~\cite{BCdTelliptic} is valid in the current more general
context without modification, as it only relies on the connection between the divisor
for~$z(u)$ and the homology class of~$T\in\T_1$.

\begin{cor}[slopes of solid phases, \cite{BCdTelliptic}, Corollary~39]
  Suppose that~$u_0$ belongs to one of the connected components of~$A_0\setminus\mapalpha(\T_1)$.
Then, we have
  \begin{equation*}
    (s^{u_0},t^{u_0}) = \sum_{T\in\T_1: [u_0,\alpha_T,u_1]} (v_T,-h_T)\,,
  \end{equation*}
where the sum is over all~$T\in\T_1$ such that the cyclic order relation~$[u_0,\alpha_T,u_1]$ holds in~$A_0$.
In particular, the points~$P_1+(t^{u_0},-s^{u_0})$ indexed by the connected
  components of~$A_0\setminus\mapalpha(\T_1)$ are the integer boundary vertices of~$N(\Gs)$.\qed
\end{cor}

For gaseous phases of the model, we have the following correspondence with the~$g$ marked interior lattice points of~$N(\Gs)$,
recall Proposition~\ref{prop:angles_perio}.
Note that the main ingredient in the proof is the Riemann bilinear relation~\eqref{eq:RiemannBilin}.

\begin{cor}[slopes of gaseous phases]
\label{cor:slope-gas}
Suppose that~$u_0$ belongs to~$A_k$ for some~$1\leq k\leq g$.
Then, the slope~$(s^{u_0}, t^{u_0})$ of the corresponding Gibbs measure is related to the~$k^\text{th}$~component of $\varphi(\mapalpha)$ by
  \begin{equation*}
    P_1+(t^{u_0},-s^{u_0}) = \varphi_k(\mapalpha)\,.
  \end{equation*}
\end{cor}

\begin{proof}
We give the details for the imaginary part of the identity,
which corresponds to the horizontal slope~$s^{u_0}$.
The computation of the real part, corresponding to the vertical slope~$t^{u_0}$,
is similar, and therefore left to the reader.

 Since~$u_0$ belongs to~$A_k$, the contour of integration~$\Cs_{u_1}^{u_0}$ is a loop homologous to~$B_k$ in~$\Sigma$.
However, because of the possible presence of singularities of~$z$,
it can a priori not be moved outside the interval of~$A_0\setminus\mapalpha(\T_1)$ containing~$u_1$.
 For this reason, we consider realizations of the cycles~$B_1,\dotsc,B_g$
crossing~$A_0$ in that same fixed interval.

  According to Proposition~\ref{prop:J_dlog} and its proof, the differential form $J_{\gamma_x}=-d\log z$ is
  the unique differential of the third kind with a simple pole at
every~$\alpha_T$ for $T\in\T_1$ such that~$v_T\neq 0$, and residue $v_T$.
  Therefore, we can take this form for~$\omega_D$ in the Riemann bilinear relation~\eqref{eq:RiemannBilin},
with corresponding divisor~$D=\sum_{T\in\T_1}v_T\alpha_T$.
This yields
\begin{equation}
\label{eq:s}
s^{u_0}=\frac{1}{2i\pi}\int_{\Cs_{u_1}^{u_0}} J_{\gamma_x}=\frac{1}{2i\pi}\int_{B_k} \omega_D=\int_{D^-}^{D^+} \omega_k\,,
\end{equation}
where the integration paths from~$D^-=\sum_{T:v_T<0}(-v_T)\alpha_T$ to~$D^+=\sum_{T:v_T>0}v_T\alpha_T$
lie in the surface~$\Sigma$ cut along the cycles $A_\ell, B_\ell$.
Note that in this surface with boundary, the cycle~$A_0$ is represented by~$g$ oriented segments.
  Because of the assumption that all the~$B_\ell$'s intersect~$A_0$ in the same interval
of~$A_0\setminus\mapalpha(\T_1)$, all the $\alpha_T$'s are in the same segment.
We label them~$\alpha_1,\dotsc,\alpha_r$ in the increasing order along this oriented segment of~$A_0$
(see Figure~\ref{fig:cutopen}), and write~$T_1,\dotsc,T_r$ for the corresponding train-tracks.
In conclusion, the integral of~$\omega_k$ from~$D^-$ to~$D^+$ is given by integrals along a single segment of~$A_0$,
with orientation from~$\alpha_j$ with~$v_{T_j}<0$ to~$\alpha_\ell$ with~$v_{T_\ell}>0$.

\begin{figure}
  \centering
  \def\svgwidth{7cm}
  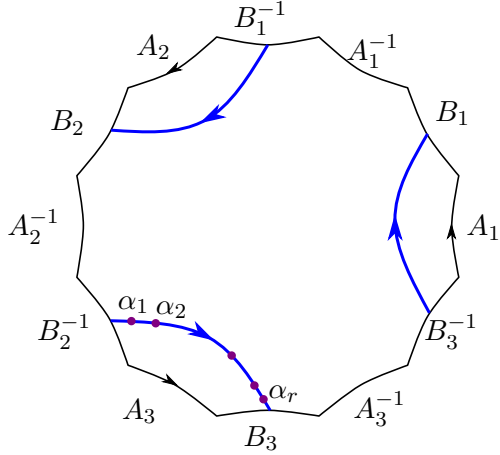
  \caption{The surface~$\Sigma$
cut along cycles representing a basis of its homology, chosen in such a way that all cycles~$B_\ell$ cut~$A_0$ in
 the same interval of~$A_0\setminus\mapalpha(\T)$.
The arcs in blue are the segments of~$A_0$ after the cuts, one of which contains all of~$\mapalpha(\T)$.
}
\label{fig:cutopen}
\end{figure}

Let us now study Equation~\eqref{equ:diff},
where we take~$x_0\in A_0$ just before~$\alpha_1$. The~$k^{th}$~component of~$n_j\in\ZZ^g$  is easily seen to be equal to~$\delta_{j1}$,
so the \new{vertical coordinate} of the~$k^{th}$~component of Equation~\eqref{equ:diff} yields
\[
\sum_{j=1}^rv_{T_j}\int_{x_0}^{\alpha_j}\omega_k=\Im\left(P_1-\varphi_k(\mapalpha)\right)\,,
\]
where the integrals are along the positive orientation of~$A_0$.
Comparing this equation with the right-hand side of~\eqref{eq:s}, and carefully taking into account the different orientation constraints
for these integrals, we obtain the equality~$s^{u_0}=\Im\left(P_1-\varphi_k(\mapalpha)\right)$.
\end{proof}

\subsection{Surface tension and free energy}
\label{sub:free}

For $(s,t)\in N(G)$, the \emph{surface tension} $\tau(s,t)$ of the dimer model
is defined as the exponential
growth rate of the partition function $Z_n^{(s,t)}$ of dimer configurations on
$\Gs_n=\Gs\setminus n\mathbb{Z}^2$ whose horizontal and vertical height change
around the torus are conditioned to be $\lfloor n s\rfloor$ and $\lfloor n t\rfloor$:
\begin{equation*}
  \tau(s,t) = -\lim_{n\to\infty} \frac{1}{n^2} \log Z_n^{(s,t)}\,,
\end{equation*}
see~\cite[Section 3.2.4]{KOS}.
Sheffield~\cite{Sheffield} proved that $\tau$ is a strictly convex function over
$N(G)$.
Using the correspondence between $(s,t)$ and $u_0$ given by
Equation~\eqref{eq:slope}, one can now see~$\tau$ as a function
$\underline{\tau}(u_0)=\tau(s^{u_0},t^{u_0})$ on
$\D=\Sigma^+\setminus\mapalpha(\T)$.

The \emph{free energy}~$F(B_x, B_y)$ of the dimer model is the Legendre dual of~$\tau$
\begin{equation*}
  F(B_x,B_y)=\max_{(s,t)}\bigl(sB_x+t B_y-\tau(s,t)\bigr)\,.
\end{equation*}
It is well defined up to an anchoring in~$\mathbb{Z}^2$ of the Newton polygon~$N(G)$ to fix the additive
constant in the definition of the height, which
corresponds to a change in the linear part in~$(B_x, B_y)$. This is done for
example by choosing a ``frozen'' point~$u_1$ in~$\D\cap A_0$ as the
reference configuration to measure the height function.

As the surface tension, the free energy can be seen as a function on $\D$:
\begin{equation*}
  \underline{F}(u)=F(-\log|w(u)|,\log|z(u)|)\,.
\end{equation*}
Because of the duality relation between the height and the magnetic field, the
differential of $\underline{\tau}$ is given by
\begin{align*}
  d\underline{\tau}(u) &= \frac{\partial\tau}{\partial s} ds^{u} +
  \frac{\partial\tau}{\partial t} dt^{u}
                  = B_x(u)ds^{u} + B_y(u) dt^{u}\\
                  &= \frac{1}{\pi}\bigl(\log|w(u)| d\arg z(u) -\log|z(u)|d\arg
                  w(u)\bigr)\,.
\end{align*}
Here and later in this section, the arguments are measured along continuous
paths in~$\D$ starting from the fixed point~$u_1$.

When~$u_0=u_1$, the measure concentrates on a single
periodic configuration~$\Ms_1$, and therefore, the surface tension is explicitly
obtained as
\begin{equation*}
  \underline{\tau}(u_1) =
-\sum_{\ws\bs\in \Ms_1}\log\bigl| \Ks_{\ws,\bs}\bigr|\,.
\end{equation*}
Then one can obtain the expression for $\underline{\tau}(u_0)$ for any other~$u_0\in\D$ by integration:
\begin{equation*}
  \underline{\tau}(u_0)=\underline{\tau}(u_1) +\frac{1}{\pi}\int_{u_1}^{u_0}
  \log|w(u)| d\arg z(u) -\log|z(u)|d\arg w(u)\,.
\end{equation*}

Taking advantage of the fact that~$z(u)$ and~$w(u)$ are expressed as products
over train-tracks of~$\Gs_1$, recall Equation~\eqref{eq:expr_zu_wu}, one can rewrite the~1-form
\begin{equation*}
  \log|w(u)| d\arg z(u) -\log|z(u)|d\arg w(u)=-\sum_{S}\sum_T (h_S v_T-v_S h_T)
  k_{\alpha_S}(u)d\ell_{\alpha_T}(u)\,,
\end{equation*}
where~$k_\alpha(u)=\log|E(\wt\alpha,\wt{u})|$ and~$\ell_\alpha(u) = \arg E(\wt\alpha,\wt{u})$.
The integer~$h_S v_T-v_S h_T$ is nothing but the algebraic intersection number
of the train-tracks~$S$ and~$T$.
Therefore, each edge~$\es\in\Es_1$ contributes exactly twice to this double sum: if~$\es$ is at
the intersection of the train-tracks~$T_\alpha$ and~$T_\beta$, then it contributes
once when~$(S,T)=(T_\alpha,T_\beta)$, and once when~$(S,T)=(T_\beta,T_\alpha)$.
Following the convention of~Figure~\ref{fig:around_rhombus}, this leads to
\begin{equation*}
  \log|w(u)| d\arg z(u) -\log|z(u)|d\arg w(u)=
  \sum_{\es\in\Es_1} \bigl(k_\beta(u)d\ell_\alpha(u)-k_\alpha(u)d\ell_\beta(u)\bigr)\,.
\end{equation*}
This yields the following \emph{local formula} for the surface tension, in the sense
that it consists of a sum
of terms associated to edges of the fundamental domain.
\begin{prop}
  \label{prop:surface_tension}
  Let~$u_1$ be a point on~$\D\cap A_0$ describing a
  frozen phase, and let~$\Ms_1$ be the corresponding dimer configuration on~$\Gs_1$,
  repeated in a periodic way on~$\Gs$.
  For any~$u_0\in\D$, the surface tension~$\underline{\tau}(u_0)$ is given by
\begin{equation*}
  \underline{\tau}(u_0)=-\sum_{\ws\bs\in\Ms_1}\log\bigl|\Ks_{\ws,\bs}\bigr|
  +\frac{1}{\pi}\sum_{\es\in\Es_1}
\int_{u_1}^{u_0}k_\beta(u)d\ell_\alpha(u)-k_\alpha(u)d\ell_\beta(u)\,.\qed
\end{equation*}
\end{prop}

\begin{rem}\leavevmode
  \begin{enumerate}
    \item
  Note that the term associated to a given edge~$\es$ in the sum over~$\Es_1$
  is genuinely local, as it only depends on the two parameters~$\alpha$ and~$\beta$
  of the train-tracks crossing that edge. All the dependency on~$t$ and the
  non-locality associated to~$\mapd$ is contained in the constant~$\underline{\tau}(u_1)$.
\item \emph{Per se}, the functions~$k_\alpha$ and~$\ell_\beta$ are not
  well-defined
on~$\Sigma$, so one needs to be slightly cautious when
  manipulating them. However, changing the lift~$\wt{u}$ in their definition
  would add to the integrand~$k_\beta(u)d\ell_\alpha(u)-k_\alpha(u)d\ell_\beta(u)$ a term which is of the
form~$F(\wt{u};\alpha)-F(\wt{u};\beta)$, and hence contributes~$0$ when summing over~$\es\in\Es_1$.
\item One can take advantage of telescopic contributions to add (or subtract)
  $k_\alpha(u)d\ell_\alpha(u) - k_\beta(u)d\ell_\beta(u)$ to the integrand, and so
  replace it by $(k_\alpha(u)\mp k_\beta(u))(d\ell_\alpha(u)\pm
  d\ell_\beta(u))$.
  \end{enumerate}
\end{rem}

We now derive a formula in the same spirit for~$\underline{F}$.
\begin{cor}
  \label{cor:free_energy}
  For any~$u_0$ in~$\D$, the free energy~$\underline{F}(u_0)$ is given by
  \begin{equation*}
    \underline{F}(u_0)=
    \sum_{\ws\bs\in\Ms_1}\log\bigl|\Ks_{\ws,\bs}\bigr|
    +\sum_{\es\in\Es_1}
    \frac{1}{\pi}\int_{u_1}^{u_0}\ell_\beta(u)dk_\alpha(u)-\ell_\alpha(u)dk_\beta(u)\,.
  \end{equation*}
\end{cor}

\begin{proof}
  For any~$u_0$ in~$\D$, Legendre duality yields the following relation
  between~$\underline{F}(u_0)$ and~$\underline{\tau}(u_0)$:
  \begin{equation*}
    \underline{F}(u_0)= s^{u_0} B_x(u_0) +t^{u_0} B_y(u_0) - \underline{\tau}(u_0)\,.
  \end{equation*}
  As above, we now rewrite the quantity~$s^{u_0} B_x(u_0) +t^{u_0} B_y(u_0)$ as a
  sum over edges of~$\Es_1$:
  \begin{align*}
    s^{u_0} B_x(u_0) +t^{u_0} B_y(u_0)
    &= \frac{1}{\pi}\bigl(\arg z(u_0) \log|w(u_0)| -\arg
    w(u_0)\log|z(u_0)|\bigr)\\
    &=\frac{1}{\pi}\sum_{S}\sum_{T}(h_S v_T-v_S h_T)k_{\alpha_S}(u_0)
    l_{\alpha_T}(u_0)\\
    &=\frac{1}{\pi}\sum_{\es\in\Es_1}
    k_{\alpha}(u_0) \ell_{\beta}(u_0)-k_{\beta}(u_0) \ell_{\alpha}(u_0)\,.
  \end{align*}
  Finally, for any $\alpha$ and $\beta$, integration by parts yields
  \begin{equation*}
  k_{\alpha}(u_0) \ell_{\beta}(u_0)-\int_{u_1}^{u_0} k_{\alpha}(u) d\ell_{\beta}(u)
= \int_{u_1}^{u_0}\ell_{\beta}(u) dk_{\alpha}(u).\qedhere
  \end{equation*}
\end{proof}

\begin{rem}\leavevmode
  \begin{enumerate}
    \item The choice of the reference frozen phase appears in~$\ell_\alpha$,
which is defined as a continuous argument computed along a path from
      that reference point~$u_1$.
    \item Using the fact that~$k_\alpha, d\ell_\alpha$ (resp.\ $dk_{\alpha},\ell_{\alpha}$)
are symmetric (resp.\ antisymmetric) with respect
      to~$\sigma$, the integration from~$u_1$ to~$u_0$ can also be expressed as~$\frac{1}{2}$ times an
integral from~$\sigma(u_0)$ to~$u_0$ along the path~$\mathsf{C}_{u_1}^{u_0}$,
      symmetric with respect to $\sigma$ and
      passing through the sector of~$A_0\setminus\mapalpha(\T)$ containing~$u_1$.
One can then replace~$\ell_{\alpha}(u)=\arg
        E(\wt\alpha,\wt{u})$
      by~$\frac{1}{i}\log E(\wt{\alpha},\wt{u})$.
  \end{enumerate}
\end{rem}

The free energy is also given~\cite{KOS} by the \emph{Ronkin function}~$R(B_x,B_y)$ of the
characteristic polynomial~$P$ (up to a linear factor in~$(B_x,B_y)$
depending on the anchoring of~$N(P)$ and its relation to~$u_1$), \emph{i.e.}, by
\begin{equation*}
  R(B_x,B_y) = \iint_{\substack{|z|=e^{B_y}\\|w|=e^{-B_x}}}
  \log\bigl|P(z,w)\bigr| \frac{dz}{2i\pi z}\frac{dw}{2i\pi w}\,.
\end{equation*}
Note that the coordinates are `rotated' by 90 degrees when compared
with the original definition because of our choice of conventions,
as for the amoeba.
One can indeed check that the two expressions match by comparing the formula
from Corollary~\ref{cor:free_energy} with computations of~$R$ when an explicit
parametrization for the spectral curve~$\mathscr{C}$ is known, see for example
Theorem~7.5 from~\cite{Brunault_Zudilin} and references therein.

\section{Additional features, and perspectives}
\label{sec:more}

This final and slightly informal section deals with miscellaneous additional results, together with upcoming work.
We start in Section~\ref{sub:beyond} by explaining that under some natural hypothesis,
the construction of Gibbs measures extends beyond the periodic case, following and generalising~\cite[Section~6.1]{BCdTelliptic}.
In Section~\ref{sub:inv}, we check that~\cite[Section~7]{BCdTelliptic} extends without modification: the model is invariant under local transformations,
and this invariance is a consequence of (and in some precise sense, equivalent to) Fay's identity;
a possible extension of our results beyond minimal graphs is also discussed.
Finally, in Section~\ref{sub:known}, we relate these models on specific classes of minimal graphs to known models,
delaying their detailed study to future publications.

\subsection{Beyond the periodic case}
\label{sub:beyond}

It is natural to wonder whether some results of Section~\ref{sec:periodic}, in particular the classification of Gibbs measures of Corollary~\ref{cor:Gibbs}, extend to arbitrary minimal graphs.
We are not able to fully answer this question, but the discussion of~\cite[Section~6.1]{BCdTelliptic} applies, leading to the following result.

Let us assume that the minimal graph~$\Gs$ and angle map~$\mapalpha\in X_\Gs$ satisfy the following condition: any \new{finite} simply connected subgraph~$\Gs_0\subset\Gs$ extends to
 a periodic minimal graph~$\Gs'$, with the restriction of~$\mapalpha$ to the train-tracks of~$\Gs_0$ extending to an element~$\mapalpha'$ of~$X_{\Gs'}$.

\begin{thm}
\label{thm:Gibbs}
Consider a minimal graph~$\Gs$ and an element~$\mapalpha\in{X}^\mathit{per}_\Gs$ satisfying the assumption above.
Fix an M-curve~$\Sigma$, a real element~$t\in\Jac(\Sigma)$, and consider the dimer model on~$\Gs$ with corresponding Kasteleyn operator~$\Ks$.
Then, for every~$u_0\in\D$, the operator~$A^{u_0}$ defines a Gibbs measure~$\PP^{u_0}$ whose expression on cylinder sets is given
as follows: for any set~$\{\es_1=\ws_1 \bs_1,\dotsc,\es_k=\ws_k \bs_k\}$ of distinct edges of~$\Gs$,
\[
\PP^{u_0}(\es_1,\dotsc,\es_k)=
\Bigl(\prod_{j=1}^k
\Ks_{\ws_j,\bs_j}\Bigr)\times
\det_{1\leq i,j\leq k} \Bigl(\A^{u_0}_{\bs_i,\ws_j}\Bigr)\,.
\]
\end{thm}

Once again, the remarkable property of these Gibbs measures is that they are local, in the sense that the probability of occurrence of any set of edges only depends on the weighted graph
near these edges.
The computation of the probability of occurrence of a single edge done in
Remark~\ref{rem:proba_edge} also holds in the more general setting of
Theorem~\ref{thm:Gibbs}. This is an illustration of the strength of local
formulas, which allow for explicit computations also in the case of non-periodic
graphs.

\begin{rem}
\label{rem:t-embedding}
 We believe that the condition stated above holds for any minimal graph~$\Gs$ and map~$\mapalpha\in X_\Gs$.
 Proving that this is indeed the case would not only imply that Theorem~\ref{thm:Gibbs} holds for any minimal graph~$\Gs$.
 As another consequence, the \emph{t-embedding}~\cite{KLRR,CLR} determined by Fock's Kasteleyn operator would define an embedding of the dual graph~$\Gs^*$ (for~$u_0$ in the interior of~$\Sigma^+$),
 a fact that is currently known to hold only for infinite periodic graphs, and for finite graphs with outer face of degree~$4$.
\end{rem}

\subsection{Invariance under moves, and going beyond minimal graphs}
\label{sub:inv}

Dimer configurations behave in a controlled way under several local transformations of bipartite graphs.
A natural family of such moves was introduced by Kuperberg and studied by Propp~\cite{Propp} under the name of \emph{urban renewal}.
An equivalent set of moves was considered by Goncharov and Kenyon~\cite{GK} and called \emph{shrinking/expanding of a 2-valent vertex} and \emph{spider move},
see Figure~\ref{fig:spider}.

\begin{figure}[htb]
    \centering
    \includegraphics[width=\linewidth]{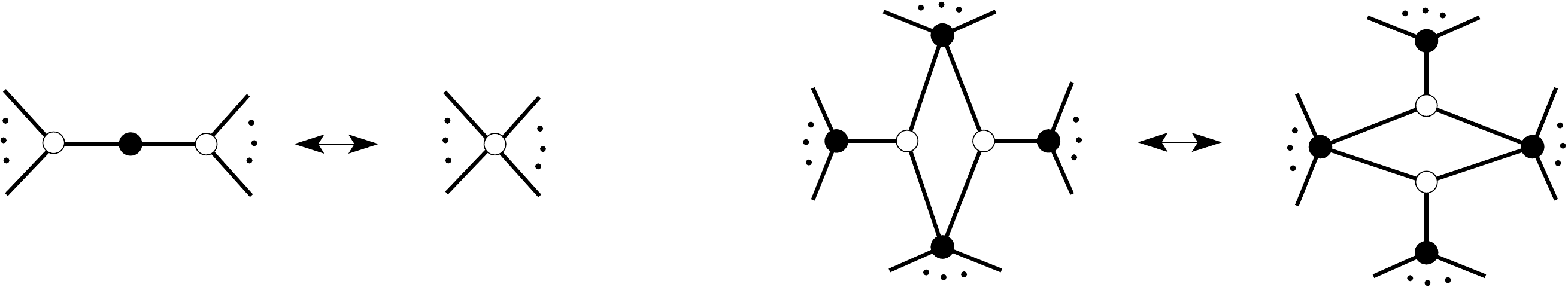}
    \caption{Shrinking/expanding of a~$2$-valent (black) vertex, and spider move (with black boundary vertices).}
    \label{fig:spider}
  \end{figure}

These moves play a crucial role in the theory. As was shown in~\cite[Theorem~2.5]{GK}, the work of Thurston~\cite{Thurston} implies that any
two periodic minimal graphs with the same Newton polygon are related by a finite sequence of these local transformations (recall Remark~\ref{rem:Harnack} above).
It is therefore natural to wonder how the dimer models studied in the present article behave under these moves.

The answer is the content of~\cite[Section~7]{BCdTelliptic}, which extends verbatim from the elliptic setting to the general case of arbitrary genus. We now give a brief summary of these results.

First of all, one easily checks that given a finite, bipartite, planar graph~$\Gs$ (not necessarily minimal) with Kasteleyn operator~$\Ks$ (not necessarily Fock's),
the associated partition function is invariant under shrinking/expanding of a~$2$-valent black vertex~$\bs$ with adjacent vertices~$\ws_1,\ws_2$ if and only~$\Ks$ satisfies
the equality~$\Ks_{\ws_1,\bs}+\Ks_{\ws_2,\bs}=0$ (and similarly for~$2$-valent white vertices, see~\cite[Proposition~50]{BCdTelliptic}).
The prime form being anti-symmetric, this holds in particular for Fock's Kasteleyn operator.

With this condition satisfied, it can be assumed via reduction of~$2$-valent white vertices and expansion of~$2$-valent black vertices
that all the white vertices of~$\Gs$ are trivalent. For such graphs, Fock's weights take a particularly simple form: it is precisely given by the function~$F_s(\alpha,\beta)=\theta(\alpha+\beta-s)E(\alpha,\beta)$ of
Section~\ref{subsub:Fay}, with~$s=s(\ws)$ constant on the four white vertices appearing in any spider move with black boundary vertices.

Finally, let us consider a dimer model on a bipartite, planar graph~$\Gs$, with Kasteleyn coefficients defined by some function~$F_s$ of
train-track parameters, as above.
Then, the corresponding partition function is invariant under spider moves with black boundary vertices if and only if these coefficients satisfy Equation~\eqref{eq:FayFock}.
In particular, Fay's identity directly implies that the dimer models given by Fock's weights are invariant under spider moves, a fact first proved (for urban renewal) by Fock~\cite[Proposition~1]{Fock}.

As a concluding remark, let us mention that
any (finite) bipartite graph whose train-tracks do not self-intersect can be reduced to a minimal one via
shrinking/expanding of a 2-valent vertices, spider moves, and \emph{merging parallel edges} as in~\cite[Figure~1]{KLRR}:
this can be checked using the theory developed by Postnikov~\cite{Postnikov} as in the proof of~\cite[Lemma~3]{KLRR},
or the work of Thurston~\cite{Thurston} as in~\cite[Lemma~33]{BCdTimmersions}.
Since our models are invariant under the first two transformations, one could hope that the whole theory applies to
any (possibly non-minimal) bipartite graphs whose train-tracks do not self-intersect.
This is not the case, for the simple reason that dimer models with Fock's weights are clearly \emph{not} invariant by the third move.
Another (similar) way to show that minimal graphs form the biggest class on which our work directly applies can be found
in~\cite[Theorem~31]{BCdTimmersions}.

Note however that it is in theory possible to study the dimer model on a periodic bipartite weighted graph~$(\Gs',\nu')$
with no self-intersecting train-track, as follows:
first use the three local moves to reduce~$(\Gs',\nu')$ to a periodic minimal weighted graph~$(\Gs,\nu)$, then harness Theorem~\ref{thm:Harnack}
to compute the parameters~$\Sigma,t,\mapalpha$ so that the corresponding Fock weights on~$\Gs$ are gauge-equivalent to~$\nu$,
and finally apply our results.

\subsection{Relation to known models, and perspectives}
\label{sub:known}

The dimer models studied in the present work are very general, as they are defined on arbitrary minimal graphs and cover all dimer models in the periodic case (recall Theorem~\ref{thm:Harnack}).
As it turns out, particular types of minimal graphs yield interesting classes of dimer models, recovering and extending known models.
This study was performed in Section~8.2 of~\cite{BCdTelliptic} in the genus~1 case, showing that the elliptic models of~\cite{BdTR1,BdTR2,dT_Mass_Dirac} could be recovered
by Fock's elliptic dimer model.

The extension of these results to higher genus is beyond the scope of this article and will be the subject of subsequent work~\cite{BCRdT}. Let us sketch these constructions very briefly.

Consider a planar graph~$G$, not necessarily bipartite. To this graph, one can associate two natural bipartite graphs: the \emph{double graph}~$\Gs=G^{\scriptscriptstyle{\mathrm{D}}}$,
see e.g.~\cite{Kenyon:crit}, and the graph~$\Gs=G^{\scriptscriptstyle{\mathrm{Q}}}$, see e.g.~\cite{WuLin}, both illustrated in Figure~\ref{fig:GDQ}.
One easily checks that if~$G$ is an isoradial graph, then the associated planar, bipartite graphs~$G^{\scriptscriptstyle{\mathrm{D}}}$ and~$G^{\scriptscriptstyle{\mathrm{Q}}}$ are minimal,
so Fock's dimer models can be defined and studied on these graphs.

\begin{figure}[ht]
  \centering
   \begin{overpic}[width=3.5cm]{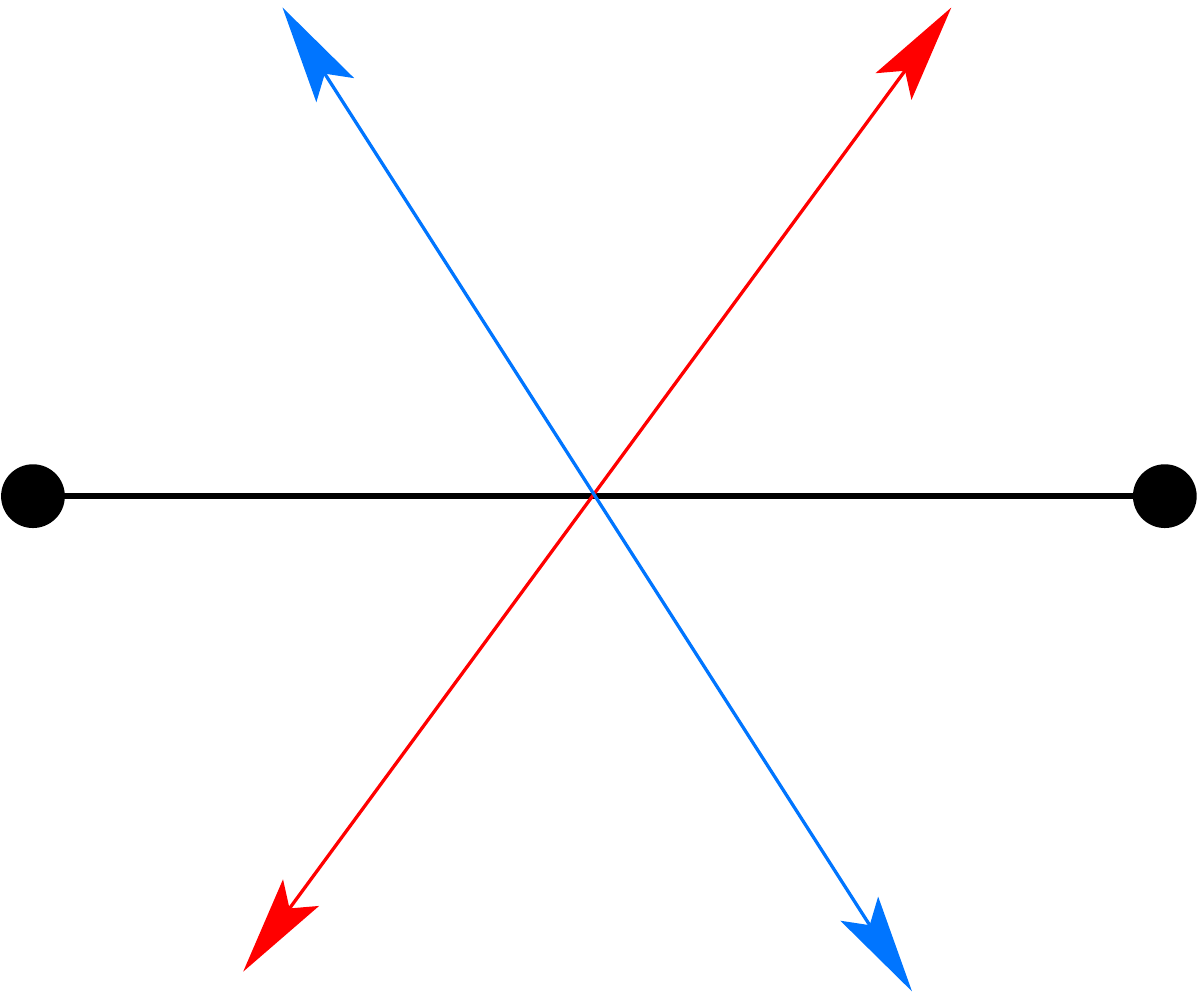}
   \put(0,60){$G$}
    \end{overpic}
   \hspace{1cm}
     \begin{overpic}[width=3.8cm]{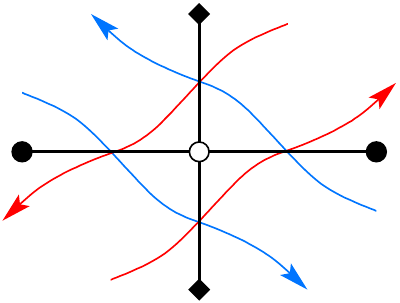}
     \put(5,65){$G^{\scriptscriptstyle{\mathrm{D}}}$}
     \end{overpic}
    \hspace{1cm}
      \begin{overpic}[width=3.8cm]{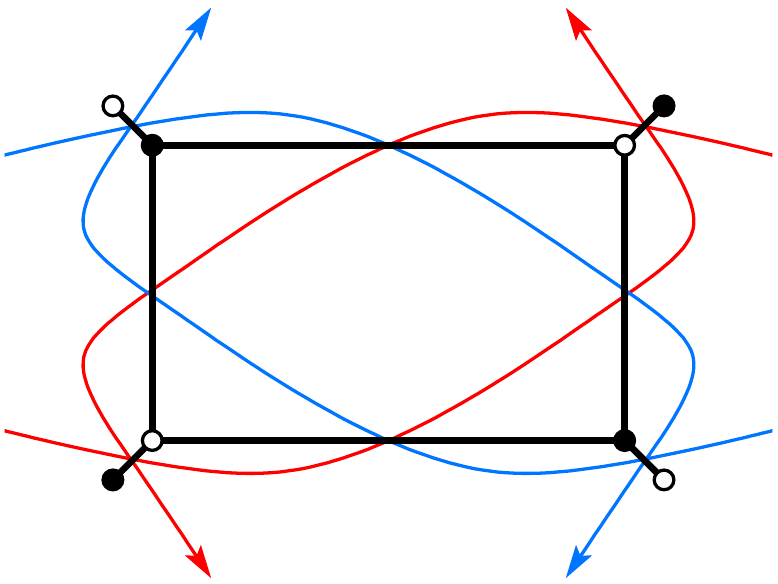}
      \put(5,65){$G^{\scriptscriptstyle{\mathrm{Q}}}$}
      \end{overpic}
    \caption{An edge of~$G$ with its two adjacent train-tracks (left), and the corresponding parts of~$\Gs=G^{\scriptscriptstyle{\mathrm{D}}}$ (center)
    and~$G^{\scriptscriptstyle{\mathrm{Q}}}$ (right) in black lines and white/black vertices, with the four adjacent train-tracks in red and blue lines.}
    \label{fig:GDQ}
\end{figure}

If the~M-curve~$\Sigma$ is endowed with a holomorphic involution, any isoradial embedding of~$G$ naturally defines a minimal immersion of~$\Gs=G^{\scriptscriptstyle{\mathrm{D}}}$,
\emph{i.e.}, an element of~$X_\Gs$. The study of the corresponding model can be undertaken using the theory of double (possibly ramified) coverings
of Riemann surfaces, as developed for example in~\cite{Fay}. In the ramified case (\emph{i.e.}, when the genus of~$\Sigma$ is even), it can be shown that Fock's Kasteleyn operator
on~$G^{\scriptscriptstyle{\mathrm{D}}}$ is gauge-equivalent to the direct sum of the discrete Laplacian of~\cite{George} on~$G$ and~$G^*$.
In the unramified case, it is a higher odd-genus generalisation of the massive Laplacian of~\cite{BdTR1} that appears.
These results, together with the study of the resulting Laplace operators and associated Green functions, will be the subject of the upcoming article~\cite{BCRdT}.

As for the dimer models with Fock's weights on minimal graphs of the form~$\Gs=G^{\scriptscriptstyle{\mathrm{Q}}}$, they yield higher genus extensions of the~$Z$-invariant elliptic Ising model of~\cite{BdTR2},
whose precise nature are yet to be understood and studied.

\bibliographystyle{alpha}
\bibliography{higher_genus_iso}

\begin{thebibliography}{KLRR18}

\bibitem[Ati71]{Atiyah}
Michael~F. Atiyah.
\newblock Riemann surfaces and spin structures.
\newblock {\em Ann. Sci. \'{E}cole Norm. Sup. (4)}, 4:47--62, 1971.

\bibitem[BCdT]{BCRdT}
C\'{e}dric Boutillier, David Cimasoni, and B\'{e}atrice de~Tili\`ere.
\newblock Integrable {L}aplacians on isoradial graphs.
\newblock In preparation.

\bibitem[BCdT20]{BCdTelliptic}
C{\'e}dric {Boutillier}, David {Cimasoni}, and B{\'e}atrice de~Tili{\`e}re.
\newblock {Elliptic dimers on minimal graphs and genus 1 Harnack curves}.
\newblock {\em arXiv e-prints}, July 2020.

\bibitem[BCdT21]{BCdTimmersions}
Cédric Boutillier, David Cimasoni, and Béatrice de~Tilière.
\newblock Isoradial immersions.
\newblock {\em J. Graph Theory}, pages 1--43, 2021.

\bibitem[BdT11]{BeaCed:isogen}
C\'{e}dric Boutillier and B\'{e}atrice de~Tili\`ere.
\newblock The critical {$Z$}-invariant {I}sing model via dimers: locality
  property.
\newblock {\em Comm. Math. Phys.}, 301(2):473--516, 2011.

\bibitem[BdTR17]{BdTR1}
C\'{e}dric Boutillier, B\'{e}atrice de~Tili\`ere, and Kilian Raschel.
\newblock The {$Z$}-invariant massive {L}aplacian on isoradial graphs.
\newblock {\em Invent. Math.}, 208(1):109--189, 2017.

\bibitem[BdTR19]{BdTR2}
C{\'e}dric Boutillier, B{\'e}atrice de~Tili{\`e}re, and Kilian Raschel.
\newblock The {Z}-invariant {I}sing model via dimers.
\newblock {\em Probab. Theory Relat. Fields}, 174(1-2):235--305, 2019.

\bibitem[Bru15]{Erwan}
Erwan Brugall\'{e}.
\newblock Pseudoholomorphic simple {H}arnack curves.
\newblock {\em Enseign. Math.}, 61(3-4):483--498, 2015.

\bibitem[BZ20]{Brunault_Zudilin}
Fran{\c{c}}ois Brunault and Wadim Zudilin.
\newblock {\em Many variations of Mahler measures: a lasting symphony},
  volume~28.
\newblock Cambridge University Press, 2020.

\bibitem[CL18]{Cretois-Lang}
R\'{e}mi Cr\'{e}tois and Lionel Lang.
\newblock The vanishing cycles of curves in toric surfaces {I}.
\newblock {\em Compos. Math.}, 154(8):1659--1697, 2018.

\bibitem[CLR20]{CLR}
Dmitry {Chelkak}, Beno{\^\i}t {Laslier}, and Marianna {Russkikh}.
\newblock {Dimer model and holomorphic functions on t-embeddings of planar
  graphs}.
\newblock {\em arXiv e-prints}, January 2020.

\bibitem[CT79]{Cook}
Roger~J. Cook and Alan~D. Thomas.
\newblock Line bundles and homogeneous matrices.
\newblock {\em The Quarterly Journal of Mathematics}, 30(4):423--429, 1979.

\bibitem[dT07]{BeaQuad}
B\'eatrice de~Tili\`ere.
\newblock {Quadri-tilings of the plane.}
\newblock {\em {Probab. Theory Relat. Fields}}, 137(3-4):487--518, 2007.

\bibitem[dT21]{dT_Mass_Dirac}
B{\'e}atrice de~Tili{\`e}re.
\newblock {The $Z$-Dirac and massive Laplacian operators in the $Z$-invariant
  Ising model}.
\newblock {\em Electron. J. Probab.}, 26, 2021.

\bibitem[{Fay}73]{Fay}
John~D. {Fay}.
\newblock {\em {Theta functions on Riemann surfaces.}}, volume 352.
\newblock Springer, Cham, 1973.

\bibitem[FK92]{Farkas-Kra}
H.~M. Farkas and I.~Kra.
\newblock {\em Riemann surfaces}, volume~71 of {\em Graduate Texts in
  Mathematics}.
\newblock Springer-Verlag, New York, second edition, 1992.

\bibitem[{Foc}15]{Fock}
V.~V. {Fock}.
\newblock {Inverse spectral problem for GK integrable system}.
\newblock {\em arXiv e-prints}, March 2015.

\bibitem[Geo19]{George}
Terrence George.
\newblock Spectra of biperiodic planar networks, 2019.

\bibitem[GK13]{GK}
Alexander~B. Goncharov and Richard Kenyon.
\newblock Dimers and cluster integrable systems.
\newblock {\em Ann. Sci. Éc. Norm. Supér.}, 46(5):747--813, 2013.

\bibitem[Har76]{Harnack}
Axel Harnack.
\newblock Ueber die {V}ieltheiligkeit der ebenen algebraischen {C}urven.
\newblock {\em Math. Ann.}, 10(2):189--198, 1876.

\bibitem[Jos06]{Jost_2006}
Jürgen Jost.
\newblock {\em Compact Riemann Surfaces}.
\newblock Springer Berlin Heidelberg, 2006.

\bibitem[Kas61]{Kasteleyn1}
Pieter~W. Kasteleyn.
\newblock The statistics of dimers on a lattice: {I}. the number of dimer
  arrangements on a quadratic lattice.
\newblock {\em Physica}, 27:1209--1225, December 1961.

\bibitem[Ken97]{Kenyon1}
Richard Kenyon.
\newblock Local statistics of lattice dimers.
\newblock {\em Ann. Inst. H. Poincar\'e Probab. Statist.}, 33(5):591--618,
  1997.

\bibitem[Ken02]{Kenyon:crit}
Richard Kenyon.
\newblock The {L}aplacian and {D}irac operators on critical planar graphs.
\newblock {\em Invent. Math.}, 150(2):409--439, 2002.

\bibitem[Ken04]{KenyonIntro}
Richard Kenyon.
\newblock An introduction to the dimer model.
\newblock In {\em School and {C}onference on {P}robability {T}heory}, ICTP
  Lect. Notes, XVII, pages 267--304. Abdus Salam Int. Cent. Theoret. Phys.,
  Trieste, 2004.

\bibitem[KLRR18]{KLRR}
Richard {Kenyon}, Wai~Yeung {Lam}, Sanjay {Ramassamy}, and Marianna {Russkikh}.
\newblock {Dimers and Circle patterns}.
\newblock {\em arXiv e-prints}, October 2018.

\bibitem[KO06]{KO:Harnack}
Richard Kenyon and Andrei Okounkov.
\newblock Planar dimers and {H}arnack curves.
\newblock {\em Duke Math. J.}, 131(3):499--524, 2006.

\bibitem[KOS06]{KOS}
Richard Kenyon, Andrei Okounkov, and Scott Sheffield.
\newblock Dimers and amoebae.
\newblock {\em Ann. of Math. (2)}, 163(3):1019--1056, 2006.

\bibitem[KS05]{K-S}
Richard Kenyon and Jean-Marc Schlenker.
\newblock Rhombic embeddings of planar quad-graphs.
\newblock {\em Trans. Amer. Math. Soc.}, 357(9):3443--3458, 2005.

\bibitem[Kup98]{Kuperberg}
Greg Kuperberg.
\newblock An exploration of the permanent-determinant method.
\newblock {\em Electron. J. Combin.}, 5:Research Paper 46, 34, 1998.

\bibitem[Law89]{Lawden}
Derek~F. Lawden.
\newblock {\em Elliptic functions and applications}, volume~80 of {\em Applied
  Mathematical Sciences}.
\newblock Springer-Verlag, New York, 1989.

\bibitem[Mer04]{Mercat:exp}
Christian Mercat.
\newblock Exponentials form a basis of discrete holomorphic functions on a
  compact.
\newblock {\em Bull. Soc. Math. France}, 132(2):305--326, 2004.

\bibitem[Mik00]{Mikhalkin1}
Grigory Mikhalkin.
\newblock Real algebraic curves, the moment map and amoebas.
\newblock {\em Annals of Mathematics-Second Series}, 151(1):309--326, 2000.

\bibitem[MR01]{M-R}
Grigory Mikhalkin and Hans Rullg{\aa}rd.
\newblock Amoebas of maximal area.
\newblock {\em Internat. Math. Res. Notices}, (9):441--451, 2001.

\bibitem[Mum07a]{ThetaTata1}
David Mumford.
\newblock {\em Tata lectures on theta. {I}}.
\newblock Modern Birkh\"{a}user Classics. Birkh\"{a}user Boston, Inc., Boston,
  MA, 2007.
\newblock With the collaboration of C. Musili, M. Nori, E. Previato and M.
  Stillman, Reprint of the 1983 edition.

\bibitem[Mum07b]{ThetaTata2}
David Mumford.
\newblock {\em Tata lectures on theta. {II}}.
\newblock Modern Birkh\"{a}user Classics. Birkh\"{a}user Boston, Inc., Boston,
  MA, 2007.
\newblock Jacobian theta functions and differential equations, With the
  collaboration of C. Musili, M. Nori, E. Previato, M. Stillman and H. Umemura,
  Reprint of the 1984 original.

\bibitem[{Pas}16]{Passare}
Mikael {Passare}.
\newblock {The trigonometry of Harnack curves}.
\newblock {\em {J. Sib. Fed. Univ., Math. Phys.}}, 9(3):347--352, 2016.

\bibitem[{Pos}06]{Postnikov}
Alexander {Postnikov}.
\newblock {Total positivity, Grassmannians, and networks}.
\newblock {\em arXiv Mathematics e-prints}, page math/0609764, September 2006.

\bibitem[Pro03]{Propp}
James Propp.
\newblock Generalized domino-shuffling.
\newblock {\em Theoret. Comput. Sci.}, 303(2-3):267--301, 2003.
\newblock Tilings of the plane.

\bibitem[{She}05]{Sheffield}
Scott {Sheffield}.
\newblock {\em {Random surfaces.}}, volume 304.
\newblock Paris: Soci\'et\'e Math\'ematique de France (SMF), 2005.

\bibitem[TF61]{TF}
Harold N.~V. Temperley and Michael~E. Fisher.
\newblock Dimer problem in statistical mechanics-an exact result.
\newblock {\em Philosophical Magazine}, 6(68):1061--1063, 1961.

\bibitem[Thu17]{Thurston}
Dylan~P. Thurston.
\newblock From dominoes to hexagons.
\newblock In {\em Proceedings of the 2014 {M}aui and 2015 {Q}inhuangdao
  conferences in honour of {V}aughan {F}. {R}. {J}ones' 60th birthday},
  volume~46 of {\em Proc. Centre Math. Appl. Austral. Nat. Univ.}, pages
  399--414. Austral. Nat. Univ., Canberra, 2017.

\bibitem[WL75]{WuLin}
Fa-Yueh Wu and Keh-Ying Lin.
\newblock Staggered ice-rule vertex model --- the {P}faffian solution.
\newblock {\em Phys. Rev. B}, 12:419--428, Jul 1975.

\end{thebibliography}

\end{document}